\documentclass[11pt]{amsart}
\usepackage{amsmath}
\usepackage{cases}
\usepackage{amssymb, verbatim}
\usepackage{comment}

\title[ASD on K3]{ASD on K3}

\theoremstyle{plain}
\newtheorem{thm}{Theorem}[section]
\newtheorem{prop}[thm]{Proposition}
\newtheorem{defn}[thm]{Definition}
\newtheorem{lem}[thm]{Lemma}
\newtheorem{cor}[thm]{Corollary}

\theoremstyle{defn}

\newtheorem{rk}{Remark}

\numberwithin{equation}{section}

\newcommand{\be}{\begin{equation}}
\newcommand{\bea}{\begin{eqnarray}}
\newcommand{\eea}{\end{eqnarray}} \newcommand{\ee}{\end{equation}}
 
 \def\ba{\begin{eqnarray}}
\def\ea{\end{eqnarray}}

\def\E{{\mathcal E}}

\def\Im{{\rm Im}}

\def\det{{\rm det}}

\def\log{\,{\rm log}\,}
\def\exp{\,{\rm exp}\,}

\def\ti{\tilde}

\def\cO{{\mathcal O}}

\def\cO{{\mathcal O}}
\def\cI{{\mathcal I}}

\def\[{{\bf [}}
\def\]{{\bf ]}}

\begin{document}
\author{Ved Datar$^{*}$} \thanks{$^{*}$ Supported in part by  NSF RTG grant DMS-1344991.}
\address{Department of Mathematics, Indian Institute of Science, Bangalore, India.}
\email{vvdatar@iisc.ac.in}
\author{Adam Jacob$^{\dagger}$}   \thanks{$^{\dagger}$ Supported in part  by a grant from the Hellman Foundation.}
\address{Department of Mathematics, UC Davis, One Shields Ave, Davis, CA.}
\email{ajacob@math.ucdavis.edu}
\author{Yuguang Zhang$^{\ddagger}$}
 \thanks{$^{\ddagger}$ Supported  by   the  Simons Foundation's program  Simons Collaboration on Special Holonomy in Geometry, Analysis and Physics (grant \#488620)}
\address{Department of Mathematical Sciences, University of Bath, Bath, BA2 7AY, UK.}
\email{yuguangzhang76@yahoo.com}

\title{Adiabatic limits   of  Anti-self-dual    connections on collapsed K3 surfaces}
\begin{abstract}
We prove a convergence result for a family of  Yang-Mills connections over an  elliptic $K3$ surface $M$ as the fibers collapse. In particular, assume $M$ is projective, admits a section, and has singular fibers of Kodaira type $I_1$ and type $II$. Let $\Xi_{t_k}$ be a sequence of $SU(n)$ connections on a principal $SU(n)$ bundle over $M$, that are anti-self-dual  with respect to a sequence of Ricci flat metrics collapsing the fibers of $M$. Given certain non-degeneracy assumptions on the spectral covers induced by $\bar\partial_{\Xi_{t_k}}$, we show that away from a finite number of fibers, the curvature $F_{\Xi_{t_k}}$ is locally bounded in $C^0$, the connections converge along a subsequence (and modulo unitary gauge change) in $L^p_1$ to a limiting $L^p_1$  connection $\Xi_0$, and the restriction of $\Xi_0$ to any fiber is $C^{1,\alpha}$ gauge equivalent to a flat connection with holomorphic structure determined by the sequence of spectral covers. Additionally, we relate the connections $\Xi_{t_k}$ to a converging family of special Lagrangian multi-sections in the mirror HyperK\"ahler structure, addressing a conjecture of Fukaya in this setting.
\end{abstract}
\maketitle

\maketitle

\section{Introduction}
The adiabatic limit of anti-self-dual connections on 4-manifolds has been extensively  studied by many  authors, with various interesting applications to problems in gauge theory, geometry, and physics. In \cite{DS,DS2}, Dostoglou and Salamon proved the Atiyah-Floer conjecture (see \cite{At2}) by showing that the adiabatic limits of self-dual connections on the product of $\mathbb{R}$ and the mapping cylinder of a principal  $SO(3)$-bundle over a compact Riemann surface of higher  genus  (greater than one) produce holomorphic curves  in the moduli space of flat connections on the $SO(3)$-bundle.  Later, the  behavior of anti-self-dual $SU(n)$-connections along the   adiabatic degenerations  of the product of two compact Riemann surfaces  of higher genus  was studied in \cite{Ch1} and \cite{Nis2} respectively, which gave mathematical rigorous proofs of the  reduction from the 4-dimensional Yang-Mills theory to  2-dimensional sigma models discovered by physicists (cf. \cite{BJSV}).
 Based on previous works of gauge theory on higher dimensional manifolds  \cite{DT,T},   \cite{Ch2} generalized the 4-dimensional case to  complex anti-self-dual connections on  products of Calabi-Yau surfaces.   The Atiyah-Floer conjecture was studied in
  \cite{Du}   for   principal   $PU(n)$-bundles.

Another motivation for the study of adiabatic limits of anti-self-dual connections arises in the context of  the mirror symmetry.
In  \cite{SYZ},  Strominger, Yau and Zaslow proposed  a  conjecture, called  the SYZ conjecture,  for   constructing   mirror Calabi-Yau manifolds via dual special Lagrangian fibrations.   Gross, Wilson, Kontsevich,   Soibelman and Todorov  \cite{GW,KS,KS2} proposed  an alternative  version of the SYZ conjecture by using   the collapsing of Ricci-flat K\"{a}hler metrics.
 Motivated by the study of homological mirror symmetry,
 a gauge theory analogue of the collapsing  of Ricci-flat K\"{a}hler metrics  was conjectured by Fukaya (Conjecture 5.5 in \cite{Fuk}), which relates the adiabatic limits   of  anti-self-dual    connections on Calabi-Yau manifolds  to special Lagrangian  cycles on the mirror Calabi-Yau manifolds.  This conjecture  was studied in the preprints \cite{Fu,Nis1} for Hermitian-Yang-Mills   connections on 2-dimensional complex torus, and in \cite{Ch3} for the case of Hermitian-Yang-Mills connections on higher dimensional semi-flat Calabi-Yau manifolds.
   The present  paper proves a version of   Fukaya's conjecture for anti-self-dual connections on elliptically fibered    K3 surfaces.

Let $M$ be a projective  elliptically fibered  $K3$ surface, $
f:M\rightarrow N\cong \mathbb{CP}^1,$  admitting    a section $\sigma:N\rightarrow M$.  Let $\alpha$ be an ample class on $M$,  $\alpha_{t} =t \alpha + f^*c_1(\mathcal{O}_{\mathbb{CP}^1}(1))$, $t\in (0,1]$, and let $\omega_t \in \alpha_{t} $ be the unique Ricci-flat K\"{a}hler-Einstein metric in this class (from  \cite{Ya}).  We denote by $ \mathrm{g}_t$ the  corresponding Riemannian metric of $\omega_t$, which is a HyperK\"{a}hler metric.
 The limit behavior of $\omega_t$  as $t\rightarrow  0$ was studied by Gross and Wilson in \cite{GW}, for K3 surfaces with only  type $I_1$ singular fibers. This was generalized to any  elliptically fibered  $K3$ surface in \cite{To1,GTZ,GTZ2}. More precisely, if $N_0\subseteq N$ denotes the complement of the discriminant locus of $f$, i.e. for any $w\in N_0$ the fiber $M_w=f^{-1}(w)$ is a smooth elliptic curve, then it is proved in \cite{GTZ} that $\omega_t$ converges to $f^*\omega$ in the locally  $C^\infty$-sense on $M_{N_0}=f^{-1}(N_0)$, where $\omega$ is a K\"{a}hler metric on $N_0$ with Ricci curvature ${\rm Ric}(\omega)=\omega_{WP}$ (obtained previously by \cite{ST,To1}),  and $\omega_{WP}$ denotes the Weil-Petersson metric of the fibers of $f$. Furthermore,
   $(M, \omega_t)$ converges to a compact metric space $Y$ homeomorphic to $N$ in the Gromov-Hausdorff sense  \cite{GTZ2}.

Assume that $f: M \rightarrow N$ has only   singular fibers of  Kodaira type $I_1$ and type  $II$.
   Let  $P$ be a principal $SU(n)$-bundle on $M$, and
 $(\mathcal{V}, H)$ be the  smooth  Hermitian vector bundle of rank $n$ obtained by the twisted product, i.e. $\mathcal{V}\cong P\times_{\rho}\mathbb{C}^n$ where $\rho$ is the standard $SU(n)$ representation on $\mathbb{C}^n$.
 Assume that there is a family of anti-self-dual connections $\Xi_t$ on $P$
  with respect to $\mathrm{g}_t$, for $t\in (0,1]$. This is equivalent to  the curvature $F_{\Xi_t}$ satisfying  $$F_{\Xi_t}\wedge \omega_t =0, \  \  \  {\rm and}  \  \    F_{\Xi_t}\wedge \Omega =0, $$ where $\Omega$ is a holomorphic symplectic form on $M$.   For each $t\in (0,1]$, $\Xi_t$ induces a holomorphic structure on $\mathcal{V}$, and we denote the resulting holomorphic bundle of rank $n$ as $V_t$.

  Under some non-degeneracy assumptions on the   behavior of $V_t$, the main result of this paper, Theorem \ref{thm-main}, 
   asserts that for any sequence $t_k \rightarrow 0$, there exists a Zariski open subset $N^o \subset N_0$ such that   $u_k(\Xi_{t_k})  $ converges subsequentially to $\Xi_0$  in the locally  $C^{0,\alpha}$-sense on $M_{N^o}$, for some sequence of  unitary gauge transformations  $u_k$  on $P$.  Furthermore,     the restriction  of the limit  $\Xi_0 $ to any fiber  is  unitary gauge equivalent to  a smooth  flat $SU(n)$-connection induced  by a holomorphic curve in $M$, which can be regarded as a multi-section of $f$. Furthermore, $\Xi_0 $ is the Fourier-Mukai transform of a certain flat $U(1)$-connection on the multi-section.
       We refer the reader to Theorem \ref{thm-main} and Theorem \ref{thm-main2} for more precise statements. By performing the HyperK\"{a}hler rotation, we can use this result to show a version of Fukaya's conjecture, relating the connections $\Xi_{t_k}$ to a converging family of special Lagrangian multi-sections in the mirror HyperK\"ahler structure.

In comparison to previous results on the adiabatic limits  of anti-self-dual connections, including, for example  \cite{DS,Ch1,Nis2,Fuk2}, one essential difficulty we encounter is that the moduli space $\mathfrak{M}_E(n)$ of flat $SU(n)$-connections on a smooth  elliptic curve  is not smooth, and actually, the whole  $\mathfrak{M}_E(n)$ is degenerated, i.e. there is no smooth point (cf. \cite{Nis11}).  Specifically, since every flat connection is gauge equivalent to a reducible connection, Poincar\'e type inequalities may not follow, creating immense analytic difficulties. The same issue also appears  for   the case of $T^4=\mathbb{C}^2/\mathbb{Z}^4$ as in \cite{Fu,Nis1}. To overcome this, we take  a totally  different approach   from \cite{Fu,Nis1}, which is inspired by the study of collapsing of Einstein  4-manifolds \cite{And,CT}. In addition we adapt some of the  arguments from \cite{DS,DS2}, as suggested in \cite{Fuk}.

Fortunately, in the literature there is a very satisfactory theory about the moduli spaces of semi-stable holomorphic bundles of rank $n$ on elliptic curves in algebraic geometry.  In the proof of Theorem \ref{thm-main}, we utilize the well understood results of  holomorphic bundles on elliptic fibered surfaces in \cite{FMW,FMW0,Fr1}, as opposed to the pseudo-holomorphic curve theory in  symplectic geometry used   in \cite{DS,Nis1}.  Additionally, in the course of our analysis, we obtain a Poincar\'{e} type inequality for the curvatures of  $SU(n)$-connections on smooth elliptic curves, which relies on the earlier work of the first two named authors (cf. \cite{DJ}). This enables  us to generalize certain arguments of  \cite{DS} to the present case. Finally, the small energy estimates for sufficiently collapsed Einstein 4-manifolds developed in \cite{And} can be adapted to the case of Yang-Mills connections on collapsed 4-manifolds, which is used to finish the proof of the main theorem.

Here we  outline  the paper briefly. Section 2 reviews   the background  notions, and preliminary results, which are needed  for the main theorem.  We recall the standard background on gauge theory in Section 2.1,  and the theory of holomorphic vectors bundles on elliptic curves   in Section 2.2.  Section 2.3 reviews the previous work about the gauge fixing on elliptic curves  by the first two named authors, which is one essential ingredient  in the proof of the main result of the present paper.  Section 2.4 recalls the   work of Friedman-Morgan-Witten \cite{FMW,FMW0}, where the relationship between  holomorphic bundles and spectral covers on elliptic surfaces is established. This work is the algebro-geometric input needed to overcome the difficulty of non-smoothness of the moduli spaces of flat connections. In Section 2.5, we setup some notations for the collapsing of Ricci-flat K\"{a}hler Einstein metrics on K3 surfaces, and leave more detailed discussions to the Appendix.   Section 2.6 reviews the notion of Fourier-Mukai transform.   We  adapt the small energy estimates for sufficiently collapsed Einstein 4-manifolds by Anderson  \cite{And} to the present case in Section 2.7.

 Section 3 is devoted to the main theorems of this paper.  We state the main theorems, and in Section 3.1, we apply the main theorems to the SYZ mirror symmetry for K3 surfaces, which proves a version of Fukaya's conjecture in \cite{Fuk}.  Section 4 contains the proof of  Theorem \ref{thm-main}  assuming some important a priori estimates, which are established  in the sections that follow. Section \ref{Poincare} contains the key analytic result  of the paper, namely the Poincar\'{e} type inequality mentioned above.  In Section 6, we obtain  a $C^0$-bound for curvature under the assumption of a certain decay rate of curvatures as the fibers collapse. Section 7 studies the relationship between the energy of curvature and the spectral covers. In Section 8, we use a blowup argument to prove the desired curvature decay rate, thereby completing the proof of Theorem \ref{thm-main}.  Section 9 proves Theorem \ref{thm-main2}.

 Finally, the appendix has some results of independent interest, where we study the collapsing rate of Ricci-flat K\"{a}hler-Einstein metrics on general  Abelian fibered Calabi-Yau manifolds. Here we improve on the previous results of \cite{GTZ,GTZ2,TZ}.
\\

\noindent {\bf Acknowledgements:}  We would like to thank   Mark Haskins for introducing the authors to the question, and some valuable comments. The work  was initiated when the second and the third named author attended the
First Annual Meeting 2017 of  the Simons Collaboration on Special Holonomy in Geometry, Analysis and Physics. We thank the  Simons Foundation and the organisers of the meeting  for  providing this opportunity.
We also thank Simon  Donaldson,   Mark Gross,  Valentino Tosatti, Yuuji Tanaka, and  Michael Singer for some discussions.

\section{Preliminaries}
In this section, we review  the various notions, and preliminary results, which are needed  for the main theorem. Although there is quite a bit of background to cover, we find it necessary to provide all the important details before we can state our results.

Let $M$ be a projective,  elliptically fibered  $K3$ surface. Denote the fibration by $
f:M\rightarrow N\cong \mathbb{CP}^1.$ Assume $f$ admits  a section $\sigma:N\rightarrow M$, and furthermore assume $f$ has only singular fibers of Kodaira type $I_1$ and type $II$. Let $I$ denote the holomorphic structure on $M$ for which $f$ is holomorphic.  We denote by $S_N$   the discriminant locus $f$, and  $N_0=N\backslash S_N$ the regular locus. The preimage of the regular locus is denoted by  $M_0:=f^{-1}(N_0)$. For any point $w\in N$, the fiber over this point is written $M_w:=f^{-1}(w)$. Additionally, for any subset $U\subset N$, we use the notation $M_U:=f^{-1}(U)$.

Let $P$ be a principal $SU(n)$-bundle on $M$, and $\mathcal{V} $ be the smooth vector   bundle of rank $n$ equipped with an Hermitian metric  $H$ induced by $P$, i.e. $\mathcal{V}=P\times_{\rho}\mathbb{C}^n$, where $\rho$ is the standard unitary   representation of $SU(n)$ on $\mathbb{C}^n$. Note that first Chern class of $\mathcal{V}$ vanishes, i.e. $c_1(\mathcal{V})=0$.

For computing norms it is convenient to use a fixed K\"ahler form $\omega$ on $M$, which lies in a fixed K\"ahler class $\alpha$. Unless otherwise specified, all norms are computed with respect to $\omega$ and $H$.  We let $\langle\cdot, \cdot\rangle_w$ denote the inner product of the space of forms  induced by $\omega|_{M_w}$ on the fiber $M_w$, and $\|\cdot\|_w $ the respective $L^2$-norm on $M_w$.

 Throughout the paper, we  let $C$ denote  constants,  which only depend  on fixed background data, whose value may change from line to line.  The constants  may   depend on a compact or open sets contained in $N$, and this dependence is either explicitly stated, or clear from context.


\subsection{Anti-self-dual  connections}  We begin by recalling the standard background on anti-self-dual connections, and readers  are referred to texts \cite{AtB,DK,FU,Kob} for details.

 Given the definition of $P$ above, let $\Xi$ be a connection on $P$, or  an  $SU(n)$-connection of $\mathcal{V} $.  If the curvature $F_{\Xi}$ satisfies
 $$ F_{\Xi}^{0,2}=0,  \  \  \ {\rm or \  \  equivalently} \  \   F_{\Xi}=F_{\Xi}^{1,1}, $$
 then $\Xi$ induces a holomorphic structure on $\mathcal{V}$. We denote the resulting holomorphic bundle as $V_{\Xi}$, and $\bar{\partial}_{\Xi}$ the corresponding Cauchy-Riemann operator. Specifically, we can write the covariant derivative $d_{\Xi}:C^\infty (\wedge^q T^*M \otimes\mathcal{V})\rightarrow C^\infty (\wedge^{q+1} T^*M \otimes\mathcal{V}) $  as  $d_{\Xi}= \partial_{\Xi}+\bar{\partial}_{\Xi}$, and the  Cauchy-Riemann operator is the $(0,1)$-component.
By construction $\Xi$ is the unique Chern connection induced by  $H$ and $\bar\partial_{\Xi}$.

Let $\mathcal{A}^{1,1}$ be  the space  of all unitary connections  with vanishing $(0,2)$-component of curvatures  on  $P$, so for any $\Xi \in \mathcal{A}^{1,1}$, we have $F_{\Xi}^{0,2}=0$. If $\mathcal{G}$ denotes the unitary gauge group, i.e. the space of unitary automorphisms of $\mathcal{V}$ covering the identity on $M$,  then $\mathcal{G}$ acts on $\mathcal{A}^{1,1} $ by $$u(\Xi)=\Xi+ u^{-1}(d_\Xi u),$$ for  $u\in \mathcal{G}$ and $\Xi \in \mathcal{A}^{1,1}$.  The $\mathcal{G}$-action extends to an action of
 the complex gauge group $\mathcal{G}_{\mathbb{C}}$, which consists all automorphisms of $\mathcal{V}$ covering the identity on $M$,  on $\mathcal{A}^{1,1}$ by  $$ g(\Xi)= \Xi +g^{-1}  (\bar{\partial}_{\Xi}g)-(g^{-1} (\bar{\partial}_{\Xi}g))^*,  $$ for   $g \in \mathcal{G}_{\mathbb{C}}$, where $(\cdot)^*$ denotes the conjugate transpose.
 Any two connections $\Xi_1$ and $\Xi_2 \in \mathcal{A}^{1,1}$ induce isomorphic holomorphic structures on $\mathcal{V}$ if and only if $\Xi_1=g(\Xi_2)$ for a certain $g\in \mathcal{G}_{\mathbb{C}}$. Therefore the quotient space $\mathcal{A}^{1,1}/ \mathcal{G}_{\mathbb{C}}$ parameterizes the holomorphic structures on  $\mathcal{V}$.

  Note that if $g \in \mathcal{G}_{\mathbb{C}}$ is an Hermitian gauge, i.e. $g=g^*$, then  for any $\Xi\in \mathcal{A}^{1,1}$, the curvature transforms via
 \bea   F_{ g(\Xi)}& = & F_{\Xi}+ \partial_{\Xi}(g^{-1}(\bar{\partial}_{\Xi}g)) - \bar{\partial}_{\Xi}( (\partial_{\Xi}g)g^{-1})   \nonumber\\ &  & + \partial_{\Xi}g g^{-2} \bar{\partial}_{\Xi}g - g^{-1}\bar{\partial}_{\Xi}g \partial_{\Xi}g g^{-1}.  \nonumber
\eea
  The transformation of $\Xi$ to $g(\Xi)$ by a Hermitian gauge $g$ is equivalent to fixing the holomorphic structure on a bundle $V$, and then changing the Hermitian metric (see \cite{Don1} for details).

Given a K\"{a}hler class $\alpha$ on $M$,  choose a K\"ahler form $\omega\in \alpha$, and let $\mathrm{g}$  be the  corresponding Riemannian metric.

\begin{defn}
An  $SU(n)$-connection $\Xi$ is called  anti-self-dual with respect to the K\"{a}hler metric $\omega$ if $\Xi$ satisfies the equation  \begin{equation}\label{asd}  \star_{ \mathrm{g}} F_{\Xi}=- F_{\Xi} , \end{equation} where $\star_{\mathrm{g}} $ denotes the Hodge star operator of  $\mathrm{g}$.
\end{defn}
For any anti-self-dual connection,  Chern-Weil theory gives
\begin{equation}\label{cw0} \int_M |F_{\Xi}|_{\omega}^2 \omega^2=- \int_M {\rm tr}(F_{\Xi}\wedge F_{\Xi})= 8 \pi^2 c_2(\mathcal{V}). \end{equation}  Furthermore, anti-self-dual connections are absolute minima  of the Yang-Mills functional on $P$, and thus satisfy the Yang-Mills equations  $$d_{\Xi}F_{\Xi}=0, \  \  \ {\rm and}  \  \  d_{\Xi}^* F_{\Xi}=0.$$ This implies  the following  Weitzenb\"{o}ck formula for the curvature of $\Xi$
 \begin{equation}\label{wei} 0=\Delta_{\Xi}F_{\Xi} =\nabla_\Xi^* \nabla_\Xi F_{\Xi}+R_\omega \# F_{\Xi}+ F_{\Xi}\# F_{\Xi}.  \end{equation}
 Here $R_\omega$ denotes the Riemannian curvature of $\omega$, and  $S\#T$  denotes some algebraic bilinear expression involving the tensors $S$ and $T$, where the exact form is not important for the present paper.

In complex  dimension 2, a connection $\Xi$ is anti-self-dual if and only if it is Hermitian-Yang-Mills \cite{DK},
 which is given by the following set of equations
  \begin{equation}\label{HYM0}  F_{\Xi}^{1,1}\wedge \omega =0,  \qquad  \  \ {\rm and} \qquad  \  \ F_{\Xi}^{0,2} =0.   \end{equation}   Thus an anti-self-dual connection $\Xi$ induces a holomorphic structure on $\mathcal{V} $, and we denote the resulting holomorphic vector bundle as $V_\Xi$.

For a given K\"{a}hler class $\alpha$ on $M$,  a holomorphic vector  bundle $V$ is called $\alpha$-stable (respectively $\alpha$-semi-stable), if for any proper torsion-free coherent  subsheaf  $\mathcal{F}$, the following inequality holds $$\frac{c_1(\mathcal{F})\cdot \alpha}{{\rm rank}(\mathcal{F}) }<  \frac{c_1(V)\cdot \alpha}{{\rm rank}(V) }  \   \   ({\rm respectively }  \  \  \leq).  $$
Fundamental work of Donaldson, Uhlenbeck, and Yau, asserts the equivalence between stability   and the existence of Hermitian-Yang-Mills connections (cf. \cite{Don1,UY}).  In partcular, we state the following Theorem, restricted to the $SU(n)$ case.

    \begin{thm}[Donaldson \cite{Don1}, Uhlenbeck-Yau \cite{UY}]\label{DUY} Let $(\mathcal{V}, H)$ be  the  smooth Hermitian bundle induced by a principal $SU(n)$-bundle $P$,  $\alpha$ be a K\"{a}hler class on $M$, and $\omega\in \alpha$  a K\"{a}hler metric.  If  the holomorphic bundle $V$ determined by a  $\mathcal{G}_{\mathbb{C}}$-orbit $ O$ in $\mathcal{A}^{1,1}$ is $\alpha$-stable, then $O$ contains an anti-self-dual connection (equivalently a Hermitian-Yang-Mills connection).  Furthermore, this connection is unique up to unitary gauge transformations. Conversely, if $\Xi$ is an anti-self-dual connection with respect to $\omega$,  and  the holomorphic bundle $V_\Xi$  induced by $\Xi$ is irreducible, then $V_\Xi$ is  $\alpha$-stable.
     \end{thm}

Note that if $\omega$ is a Ricci-flat K\"{a}hler-Einstein metric, then the corresponding Riemannian metric $\mathrm{g}$ is a HyperK\"{a}hler metric, and $(\omega, {\rm Re} (\Omega), {\rm Im} (\Omega))$ is a HyperK\"{a}hler triple (cf. \cite{GHJ}), where   $\Omega$ is a holomorphic symplectic form such that $$\omega^2={\rm Re} ( \Omega)^2= {\rm Im} (\Omega)^2, \  \  \omega\wedge \Omega=0,  \  \  {\rm and}  \  \ {\rm Re} ( \Omega)\wedge {\rm Im} (\Omega)=0. $$ Complex structures making $\mathrm{g}$ HyperK\"{a}hler are parameterized by $S^2$, and  any anti-self-dual connection $\Xi$ with respect to $ \mathrm{g}$ is also a Hermitian-Yang-Mills connection with respect to any such complex structure.  In the HyperK\"{a}hler case, the anti-self-dual equation (\ref{asd}) and the  Hermitian-Yang-Mills equation (\ref{HYM0})  are equivalent to the following system
  \begin{equation}\label{hyper}  F_{\Xi}\wedge \omega =0,  \  \  \ {\rm and} \  \  F_{\Xi}\wedge \Omega =0.  \end{equation} For the remainder of the paper, we mainly work with the above equations, as they are the most applicable to our setup.

The above equations \eqref{hyper} are given with respect to the complex structure $I$ making $f:M\rightarrow N$ holomorphic. By the HyperK\"{a}hler rotation, we have another complex structure $J$ such that the holomorphic symplectic form  $\Omega_J= {\rm Im} (\Omega)+ i \omega $, and the K\"{a}hler form $\omega_J= {\rm Re}( \Omega)$.  If $\Xi$ is an anti-self-dual connection  with respect to $ \mathrm{g}$,  then $\Xi$ also satisfies $ F_{\Xi}\wedge \omega_J =0, $ and $  F_{\Xi}\wedge \Omega_J =0$. Thus
     $\Xi$  induces a holomorphic bundle structure on $\mathcal{V}$ with respect to the complex structure $J$, denoted as $V_{\Xi, J}$, and  $\Xi$  is a  Hermitian-Yang-Mills connection on $V_{\Xi, J}$.

     We conclude this section by recalling Uhlenbeck's compactness theorems, which are divided into the cases of weak and strong compactness.

      \begin{thm}[Uhlenbeck \cite{U2,Weh2}]\label{Ucompact} Let $K$ be  a compact subset of $ M$. 
       \begin{itemize}
  \item[i)]{\rm [Weak compactness]}  If $\Xi_k$ is a sequence of unitary connections on $P|_K$ such that $\|F_{\Xi_k}\|_{L^p}\leq C$, for $p>2$, then there exists a sequence of unitary gauge transformations $u_k\in \mathcal{G}^{2,p}$ so that   $u_k(\Xi_k)$ converges along a subsequence in $L^p_{1,loc}$  to a  $L^p_1$-unitary connection $\Xi_\infty$  on $K$.
     \item[ii)]{\rm [Strong  compactness]}  If  we further assume that $\Xi_k$ is anti-self-dual with respect to a Riemannian metric $\mathrm{g}_k$, and   $\mathrm{g}_k$ converges smoothly to a smooth Riemannian metric $\mathrm{g}_\infty$ locally on $K$, then   $u_k(\Xi_k)$ converges to  $\Xi_\infty$ in the locally $C^\infty$-sense, and $\Xi_\infty$ is  anti-self-dual with respect to $\mathrm{g}_\infty$.  \end{itemize}  \end{thm}

   \subsection{Gauge theory  on  elliptic curves}
 While working with bundles over   $M$, we need several preliminary results dealing with the restriction of a bundle to a fixed elliptic fiber, which we detail here.

Fix a point $w\in N_0$, and consider the fiber $M_w=E$,   a   smooth  elliptic curve  with period $\tau$, i.e. $E=\mathbb{C}/{\rm Span}_{\mathbb{Z}}\{1, \tau\}$.  Equip $E$ with the flat metric   $\omega^F_w:=i{\rm Im}(\tau)^{-1}\,dz\wedge d\bar z$. Let $V$ be a holomorphic vector bundle of rank $n$ with trivial  determinant line bundle  $\bigwedge^{n} V \cong \mathcal{O}_{E} $, let $\bar\partial$ be the Cauchy-Riemann operator,  and fix a Hermitian metric $H$ on $V$. Let $A_{ch}$ be the unique  Chern connection determined by the holomorphic structure and the Hermitian metric  $H$, i.e. $A_{ch}=( \partial H )H^{-1}$ under a certain local holomorphic trivialization.    Recall that $\|\cdot\|_w$ denotes the $L^2$ norm on $E$.

\begin{prop}\label{semistablelemma}
There exists a $\delta>0$, dependent only on $E$ and $V$, so that if $A$ is in the complexified gauge orbit of $A_{ch}$ and satisfies $\|F_A\|_{w} <\delta$, then the holomorphic bundle  $V$ is semi-stable. \end{prop}
\begin{proof}
This proposition follows from the fact, proven by  R$\mathring{\rm a}$de, that the critical values of the Yang-Mills functional (the $L^2$ norm of the curvature) are discrete, and that in real dimension $2$ and $3$ the Yang-Mills flow converges in $L^2_1$  \cite{Rade}.  If $A$ satisfies $\|F_A\|_{w} <\delta$ for $\delta$ sufficiently small, then the Yang-Mills flow starting at $A$ must converge to a flat connection $A_0$, by discreteness of critical values. Thus $\|F_{A(t)}\|_{w}\rightarrow0$, where  $A(t)$ denotes the flow of connections.  Furthermore, the Yang-Mills flow preserves the complex gauge equivalence class of $A$, so   $A(t)$ all define isomorphic holomorphic structures on $V$. As  a result, $V$ admits an approximate Hermitian-Einstein structure, and is semi-stable \cite{Kob}.
\end{proof}
Although the Yang-Mills  flow preserves the complex gauge equivalence class of $A$, it is not immediately clear whether the limiting flat connection $A_0$ is contained in the complexified gauge orbit, or only strictly in the closure. To better understand this, we turn to Atiyah's classification of semi-stable bundles on an elliptic curve.

Let $0\in E$ the identity of the group law. Denote the trivial line bundle by $\cO_E $, and given a point $q\in E$, let $\mathcal{O}_E(q-0)$ be the line bundle associated to the divisor $q-0$. Define $\cI_r$ inductively, with $\cI_1=\cO_E$ and $\cI_r$ the unique nontrivial extension of $\cI_{r-1}$ by $\cO_E$.

\begin{thm}[Atiyah \cite{At}]
Any semi-stable  bundle $V$ over $E$ with trivial determinant bundle is isomorphic to a direct sum of bundles of the form $\mathcal{O}_E(q-0)\otimes\cI_r$, i.e. $$ V\cong \bigoplus_{j=1}^{\ell}\mathcal{O}_E(q_j-0)\otimes\cI_{r_j}. $$
\end{thm}

\begin{defn}
A semi-stable bundle  $V$ is called regular if it is of the form $V\cong \bigoplus\limits_{j=1}^{\ell}\mathcal{O}_E(q_j-0)\otimes\cI_{r_j}$  with $q_j \neq q_i$ for any $j\neq i$.
\end{defn}

Now,  in our setting one (and only one) of two things can happen. Either $V$ is isomorphic a direct sum of line bundles $V=\oplus \mathcal{O}_E(q-0)$,  and the limiting flat connection $A_0$ is in the complex gauge orbit of $A$, or $V$ is isomorphic a direct sum  of bundles of the form $\mathcal{O}_E(q-0)\otimes\cI_r$, with at least one   $r>1$. In the latter case, $\mathcal{O}_E(q-0)\otimes\cI_r$ is strictly semi-stable, since $\mathcal{O}_E(q-0)\subset \mathcal{O}_E(q-0)\otimes\cI_r$ has degree zero but $\mathcal{O}_E(q-0)\otimes\cI_r$ does not split holomorphically. As a result $V$ does not admit a flat connection, and so $A$ is not complex gauge equivalent to $A_0$.

 Note that if $V\cong \bigoplus\limits_{j=1}^{\ell}\mathcal{O}_E(q_j-0)\otimes\cI_{r_j}$, then $V$ is {\it S-equivalent} to the flat   bundle  $\bigoplus\limits_{j=1}^{\ell}\mathcal{O}_E(q_j-0)^{\oplus r_j}$ (see \cite{Fr1} for the precise definition of S-equivalence).    Every S-equivalence   class corresponds to a divisor $\sum\limits_{j=1}^{\ell} r_j q_j $ in the complete linear system $| n 0|$.  Conversely, any divisor $\sum\limits_{j=1}^{\ell} r_j q_j \in | n 0|$  on $E$ induces an S-equivalence class of  semi-stable bundles with trivial determinant,   which  contains  $\bigoplus\limits_{j=1}^{\ell}\mathcal{O}_E(q_j-0)^{\oplus r_j}$. Therefore, the moduli space of S-equivalence classes of  semi-stable bundles  with trivial determinant is given by the complete linear system $| n 0| \cong \mathbb{CP}^{n-1}$.

Furthermore, the moduli space of flat line bundles on $E$ is the dual torus $\check{E}\cong H^{0,1}(E)/ H^1(E, \mathbb{Z})$, and we identify $E$ and $\check{E}$ by $q \mapsto \mathcal{O}_E(q-0)$.  Another way to state this is that a point $q\in E$ corresponds to a flat connection $\pi ({\rm Im} \tau)^{-1} ( qd\bar{z}-\bar{q}dz)$ on the trivial Hermitian bundle $E\times \mathbb{C}$.  Therefore the flat bundle structure of $\bigoplus\limits_{j=1}^{n}\mathcal{O}_E(q_j-0)$ is given by the flat connection
 \be
 \label{flatconnection}
 A_0= \pi ({\rm Im} \tau)^{-1} ({\rm diag}\{q_1, \cdots, q_n\}d\bar{z}-{\rm diag}\{\bar{q}_1, \cdots, \bar{q}_n\}dz),
 \ee
 where $\sum\limits_{j=1}^{n}q_j \in | n 0|$. Note that the above connection has this form in a global unitary  frame for $V$. Let $\mathfrak{M}_E(n) $ denote the moduli space of flat $SU(n)$ connections on $V$, which is naturally identified with $| n 0|$,  the moduli space of S-equivalence classes of  semi-stable bundles  with trivial determinant.

We note that from the perspective of algebraic geometry, the linear system  $| n 0|$ is a  well behaved  object. On the other hand, from the perspective of symplectic geometry,  the moduli space $\mathfrak{M}_E(n) $ is quite complicated.  In particular, any flat $SU(n)$-connection on $E$ is degenerate, the virtual dimension of $\mathfrak{M}_E(n) $ is zero, and the whole space $\mathfrak{M}_E(n) $ is regarded as singular, i.e. there is no smooth point (cf. \cite{Nis1,Nis2}).  If we let  $\mathcal{A}$ denote the space of all unitary connections on the trivial bundle on $E$, and $\mathcal{G}$ the unitary gauge group, then following Atiyah-Bott \cite{AtB}, one can construct $\mathfrak{M}_E(n) $ as the symplectic reduction $\mathfrak{M}_E(n)=\{A\in \mathcal{A}|F_A=0\}/\mathcal{G} $. Using this construction $\mathfrak{M}_E(n) $ is in the singular locus of $\mathcal{A}/\mathcal{G}$.  Such ill behavior of $\mathfrak{M}_E(n)$ prevents us to generalize the arguments in \cite{Ch1,DS,Fuk2,Nis11} directly, where the moduli space of flat connections on Riemann surfaces of higher genus are considered.   Instead we follow an algebro-geometric approach  combined with  estimates for the above non-linear partial differential equations.

\subsection{Gauge fixing} In this section we continue to work on a single elliptic curve $(E,\omega)$. Let $V$ be a regular, semi-stable, holomorphic vector bundle of rank $n$ which admits a flat connection $A_0$, equipped with a Hermitian metric $H$. Suppose $A$ is another connection in the complex gauge orbit of $A_0$, i.e. $A=g(A_0)$ for some $g\in\mathcal{G}_{\mathbb C}$. It will be important for us to know under what conditions we have control over the $C^0$ norm of $g$. Since the action of a fixed unitary gauge transformation will not affect this norm, without loss of generality we assume that $A=e^s(A_0)$ for a trace free Hermitian endomorphism $s$.

In general it is not reasonable to expect direct control of $s$. For example, if $e^s$ were a diagonal matrix of constants $c_1,..., c_n$ in the trivial frame, then $e^s( A_0)$ will also be a flat connection. However, one eigenvalue $c_i$ can be arbitrarily large while still preserving the condition that $s$ be trace free, so $s$ cannot be controlled. What does end up being true is that under a small curvature assumption, there exists a normalized endomorphism $\hat s$, which may be distinct from $s$, that nevertheless gives the same connection under the complexified gauge group action, and is uniformly controlled in $C^0$. The key result of the first two named authors is as follows.

\begin{thm}[Datar-Jacob \cite{DJ}]
\label{gaugef}
Let  $e^s(A_0)$ be a connection on $V$ given by the action of a trace free Hermitian endomorphism $s.$ There exists constants $\epsilon_0>0$, and $C_0>0$, depending only on $\omega$,  $A_0,$ and $H$, so that if
\be
\qquad \|F_{e^s(A_0)}\|^2_{C^0(E)}\leq\epsilon_0,\nonumber
\ee
then there exists another trace free Hermitian endomorphism $\hat s$ satisfying that $\hat s$ is  perpendicular to the Kernel  of $d_{A_0}$, in addition to
\be
e^s( A_0)=e^{\hat s}(A_0)\qquad{\rm and}\qquad \|\hat s\|_{C^0(E) }\leq C_0.\nonumber
\ee
\end{thm}

We remark that the assumptions that $V$ be regular and admit a flat connection  are critical, as they imply that the holomorphic automorphism group of $V$ is precisely $n$ dimensional \cite{FMW}. The idea of the proof is that the linearization of  the complex gauge group action of a Hermitian endomorphism on $A_0$ is $\star d_{A_0}s$. Restricting to endomorphisms perpendicular to the Kernel of $d_{A_0}$, a Poincar\'e inequality gives that the linearized map is invertible with bounded inverse. Thus, if $e^s(A_0)$ is sufficiently close to $A_0$, via the contraction mapping principle the results of the theorem hold.  In order for the theorem to hold under the small curvature assumption, a connectedness  argument is applied. We direct the reader to \cite{DJ} for further details.

 \subsection{Spectral  covers}
 We now discuss holomorphic vector bundles over our elliptic fibration $M$, as opposed to a single elliptic curve.

We assume that $f: M \rightarrow N$ has only singular fibers of Kodaira type $I_1$ and type $II$. Then  $ M $ coincides with the Weierstrass model $\check{f}:\check{M}\rightarrow N$, i.e. $M=\check{M}$ and $f=\check{f}$.
 Let $V$ be  a holomorphic  vector bundle $V$ of rank $n$ on $M$  such that the determinant line bundle $  \bigwedge^{n} V$ is trivial, i.e. $\bigwedge^{n} V \cong \mathcal{O}_{M} $.  If the restriction of $V$ on the generic fiber of $f$ is regular semi-stable, then a multi-valued section of $f$ is constructed in \cite{FMW}, which is called the spectral cover associated to $V$. More precisely, we have the following theorem.

 \begin{thm}[\cite{FMW}]\label{special cover}
Assume that  the restriction of $V$ on the generic fiber of $f$ is  semi-stable and regular.  Then there exists a divisor $$D_V \in | n \sigma (N)+m  l|, $$ called the spectral cover associated to $V$, where  $ l$  denotes  effective  divisor class of the fibers  of $f$,    $m\in \mathbb{Z}$  satisfies $0\leq m \leq c_2(V)$, and  for a generic $w\in N_0$, $$ V|_{M_w}\cong \bigoplus_{j=1}^{\ell}\mathcal{O}_{M_w}(q_j-0)\otimes\cI_{r_j}, \  \   \  D_{V}\cap M_w =\sum_{j=1}^{\ell}r_j q_j\in | n \sigma (w)|. $$
\end{thm}

We recall the construction in \cite{FMW}.
Since $h^{0}(M_{w}, \mathcal{O}_{M_w}(n \sigma (w)))=n$ for any fiber $M_w$, the push forward $f_{*}\mathcal{O}_{M}(n \sigma)$ is a vector bundle of rank $n$ on $N$, and more precisely,  $$f_{*}\mathcal{O}_{M}(n \sigma)=\mathcal{O}_{N}\oplus L^{-2}\oplus \cdots \oplus L^{-n},$$  where $L^{-1}=\sigma^* \mathcal{O}_M(\sigma)$ by Lemma 4.1 of \cite{FMW}.
 We denote $p: \mathcal{P}_{n-1}\rightarrow N$ the projection bundle, so $\mathcal{P}_{n-1}=\mathbb{P}f_{*}\mathcal{O}_{M}(n \sigma)$ (cf. Section 4.1 of \cite{FMW}). For any $w\in N$, the fiber $p^{-1}(w)$ is the complete linear system  $|n \sigma (w)| \cong \mathbb{CP}^{n-1}$, i.e. $p^{-1}(w)= |n \sigma (w)|$, and   is identified as the coarse moduli space for semi-stable bundles of rank $n$ on $M_{w}$ (cf. Section 1 of \cite{FMW}). Since the restriction of $V$ to the generic fiber is   semi-stable, there is a non-empty  Zariski open subset $N'\subset N$ such that for any $w\in N'$, $V|_{M_w}$ is semi-stable, which defines a point $\varrho (V|_{M_w})\in |n \sigma (w))|$ by Theorem 1.2 in \cite{FMW}.  Then Lemma 4.2 of \cite{FMW} defines a section $$\mathcal{A}_{V}: N' \rightarrow p^{-1}(N'), \  \ \  \  {\rm by}  \  \  \ \mathcal{A}_{V}(w)=\varrho (V|_{M_w}),$$ and by Lemma 6.1 in \cite{FMW}, $\mathcal{A}_{V}$ extends to $N$ as a section of $\mathcal{P}_{n-1}$, denoted still by $\mathcal{A}_{V}: N \rightarrow \mathcal{P}_{n-1}$.

Section 4.3 in \cite{FMW} constructs an $n$-sheeted branched covering $ \varrho : \mathcal{T}\rightarrow \mathcal{P}_{n-1}$, which admits a $\mathbb{CP}^{n-2}$--fibration $r: \mathcal{T}\rightarrow M$. For any smooth fiber $M_w$, $\mathcal{T}_w= r^{-1}(M_w) \rightarrow M_w$ coincides with the construction in Section 2.1 of \cite{FMW} as  follows.  Let $\Pi_w \subset M_{w}^{\otimes n}$ be the subset such that $(q_1, \cdots, q_n)\in \Pi_w$ if and only if the divisor $q_1+ \cdots +q_n $ is linearly equivalent to $n \sigma (w)$. If $\mathbb{S}_n$ denotes the symmetric group, and $\mathbb{S}_{n-1} \subset \mathbb{S}_n$ is the subgroup fixing the last element, then $\mathbb{S}_n$ acts on $\Pi_w$, and the quotient  $\Pi_w / \mathbb{S}_n = |n \sigma (w)| \cong \mathbb{CP}^{n-1}$.  Also $\mathcal{T}_w=\Pi_w / \mathbb{S}_{n-1}$, $r|_{\mathcal{T}_w}:\mathcal{T}_w  \rightarrow M_w$ is given by $(q_1, \cdots, q_{n-1},  q_n) \mapsto q_n$, and $\varrho|_{\mathcal{T}_w}: \mathcal{T}_w \rightarrow |n \sigma (w)|$ is a branched $n$-sheeted cover such that $\varrho|_{\mathcal{T}_w}$ is unbranched over $q_1+ \cdots +q_n \in | n \sigma (w)|$ if and only if $q_i\neq q_j$ for any $i\neq j$. Clearly, $r|_{\mathcal{T}_w}(\varrho|_{\mathcal{T}_w}^{-1}(q_1+ \cdots +q_n))=\{q_1, \cdots ,q_n\}\subset M_w$ for any  $q_1+ \cdots +q_n \in | n \sigma (w)|$.

The spectral cover $D_{V}$ is defined as the scheme-theoretic inverse image of $\mathcal{A}_{V}(N)$, i.e. $D_{V}=\varrho^{-1}(\mathcal{A}_{V}(N)) $, which is a subscheme of $\mathcal{T}$, and $p\circ \varrho|_{D_{V}}:D_{V}\rightarrow N$ is finite and flat of degree $n$ (cf. Definition 5.3 in \cite{FMW}).  By Lemma 5.4 of  \cite{FMW},  $r|_{D_{V}}$ embeds  $D_{V}$ in $M$ as an effective  Cartier divisor, and $f\circ r|_{D_{V}}= p \circ \varrho|_{D_{V}}$. Therefore, we always regard $D_{V}$ as a divisor of $M$ in the present paper. Furthermore, Lemma 5.4 in \cite{FMW} shows that $\mathcal{O}_{M}(D_{V})\cong \mathcal{O}_{M}(n \sigma(N) )\otimes f^{*} \mathcal{L}_{V}$ where $\mathcal{L}_{V}=\mathcal{A}_{V}^{*}\mathcal{O}_{\mathcal{P}_{n-1}}(1)$.  Thus $$D_{V} \in | n \sigma(N)+ m  l |,$$ where $l$ denotes the effective  divisor class of the fibers of $f$, and $m={\rm deg} \mathcal{L}_{V}\in \mathbb{Z}$.

The arguments in Section 6.1 of  \cite{FMW} show  that \begin{equation}\label{chern2} 0 \leq m={\rm deg} \mathcal{L}_{V} \leq c_{2}(V), \end{equation} which is sketched  as  follows.  Since  the restriction of $V$ to the generic fiber is regular semi-stable, there are only finite possible fibers such that the restrictions of $V$ are unstable. Lemma 6.2 of \cite{FMW} proves that by preforming  finite allowable elementary modifications to $V$, one    obtains a new bundle $V'$ such that the restriction of $V'$ to any fiber is semi-stable. Furthermore $c_2(V')\leq c_2(V)$, and equality holds if and only if $V'=V$, i.e. there is no elementary modification preformed.

The proof of Corollary 6.3 in \cite{FMW} shows that there is a coherent sheaf  $V_0$, whose restriction on any fiber is regular semi-stable, and
a morphism $\psi:V_0 \rightarrow V'$, which is an isomorphism on $f^{-1}(U)$ for a nonempty Zariski open set $U\subset N$.  The cokernel coherent  sheaf $Q$ is a torsion sheaf supported on finite fibers, and admits a filtration by degree zero sheaves.  Consequently, $c_2(V_0)=c_2(V')$.   Note that $V_0$ is isomorphic to $V$ on $f^{-1}(U')$ for a nonempty Zariski open set $U'\subset N$, as the above two processes only change the restrictions  of $V$ on finite fibers. Therefore we have  $\mathcal{A}_{V_0}=\mathcal{A}_{V}$, $D_{V_0}=D_{V}$, and $\mathcal{L}_{V_0}=\mathcal{L}_{V} $.  By Proposition 5.15 of \cite{FMW}, ${\rm deg} \mathcal{L}_{V_0} = c_{2}(V_0)$, and we obtain the inequality (\ref{chern2}).

 The spectral cover $D_{V}$ gives a criterion of $V$ being stable.

 \begin{thm}[Theorem 7.4 of \cite{FMW}]\label{stable criterion}
If  $D_{V}$ is reduced and irreducible, then $V$ is stable with respect to  $f^{*} c_{1}(\mathcal{O}_{\mathbb{CP}^{1}}(1))+t \alpha$, for all $0<t\leq (\frac{n^3}{4}c_{2}(V))^{-1}$, where $\alpha$ is an ample class on $M$.
\end{thm}

This theorem can be used  to construct stable bundles on $M$ as  follows. If $D\in | n \sigma (N)+m l|$, $m>2n$,  is an effective  reduced and irreducible divisor, then Lemma 5.4 in \cite{FMW} asserts that $D$ is the spectral cover  of a unique section $\mathcal{A}$  of $\mathcal{P}_{n-1}$, which satisfies $m ={\rm deg}\mathcal{A}^* \mathcal{O}_{\mathcal{P}_{n-1}}(1)$.  A holomorphic vector bundle $V$ is constructed from $\mathcal{A}$ (cf. Definition 5.2 in \cite{FMW}) such that the restriction of  $V$ on every fiber is regular semi-stable with trivial determinant line bundle, and $D$ is the spectral cover of $V$, i.e. $D_{V}=D$.

 We recall the construction in Section 5.1 of \cite{FMW} by assuming that $D$ is smooth, and does not intersect with any singular set of the singular fibers of $f$. If $\tilde{M}= D\times_{N} M$ denotes the base change, which is smooth,  then there are morphisms $\tilde{f}: \tilde{M} \rightarrow D$ and $\nu_D: \tilde{M} \rightarrow M$ such that $f\circ \nu_D= f|_D \circ  \tilde{f}$.  We regard $\tilde{M}= D\times_{N} M\subset M\times_N M$ via the natural  embedding $D\hookrightarrow  M$.   Then $\Sigma_D = \nu_D^* \sigma$ and $\Delta= \tilde{M} \cap \Delta_0$ are divisors, where  $\Delta_0$ is the diagonal of $ M\times_N M$. For any $w\in N_0$, and $q_j(w)\in M_w\cap D$, we have $\tilde{M}_{(w, q_j(w))}=M_w$, $\Sigma_D \cap \tilde{M}_{(w, q_j(w))}= \{\sigma(w)\}$, and $\Delta \cap \tilde{M}_{(w, q_j(w))}= \{q_j(w)\}$.
 Lemma 5.5 of \cite{FMW}  asserts that the push forward  $(\nu_D)_* \mathcal{O}_{\tilde{M}}(\Delta - \Sigma_D)$ satisfies that its  restriction  on every fiber is regular semi-stable with trivial determinant line bundle.  Furthermore, for any line bundle $\tilde{L}$ on $D$, $(\nu_D)_* (\mathcal{O}_{\tilde{M}}(\Delta - \Sigma_D)\otimes \tilde{f}^*\tilde{L})$  also satisfies the required conditions.

  Conversely, if $V$ is a holomorphic vector bundle whose restriction of  $V$ on every fiber is regular semi-stable with trivial determinant line bundle, and $D$ is the spectral cover of $V$, then $$V=(\nu_D)_* (\mathcal{O}_{\tilde{M}}(\Delta - \Sigma_D)\otimes \tilde{f}^*\tilde{L})$$ for a certain line bundle $\tilde{L}$ on $D$ by Proposition 5.7 in \cite{FMW}.
 Now, since ${\rm deg}L=-\sigma^2=2$, Proposition 5.12 of \cite{FMW} asserts that one can choose $V$ via  a suitable $\tilde{L}$ on $D$  such that the first Chern class $c_1(V)=0$, and therefore, $V$ has  trivial determinant line bundle on $M$.
   Now Theorem 7.4 of \cite{FMW} shows that $V$ is stable with respect to  $f^{*} c_{1}(\mathcal{O}_{\mathbb{CP}^{1}}(1))+t \alpha$ for $0<t \ll 1$.
In summary, we have

\begin{thm}\label{stable construction}  If $D\in | n \sigma (N)+m l|$, $m>2n$,  is an effective  reduced and irreducible divisor, then there exists a  holomorphic vector bundle $V$ of rank $n$  with $c_1(V)=0$ on $M$ such that  the restriction of  $V$ on every fiber is regular semi-stable, and $D$ is the spectral cover of $V$, i.e. $D_{V}=D$. Furthermore,
 $V$ is stable with respect to  $f^{*} c_{1}(\mathcal{O}_{\mathbb{CP}^{1}}(1))+t \alpha$, for all $0<t\leq (\frac{n^3}{4}c_{2}(V))^{-1}$, where $\alpha$ is an ample class on $M$.
\end{thm}

\subsection{Collapsing  of Ricci-flat   K\"{a}hler-Einstein metrics}
\label{basemetric}

We now introduce  some preliminary results on our family of collapsing base metrics on $M$, and highlight a new decay estimate necessary for our main theorem. The reader is directed to  Appendix A for a proof of this particular asymptotic decay.

Let $\alpha$ be an ample class on $M$,  $\alpha_{t} =t \alpha + f^*c_1(\mathcal{O}_{\mathbb{CP}^1}(1))$, $t\in (0,1]$, and $\omega_t \in \alpha_{t} $  the unique Ricci-flat K\"{a}hler-Einstein metric, which satisfies the complex Monge-Amp\`ere equation $$ \omega_t^2=c_t t \Omega \wedge \overline{\Omega}.  $$ Here $\Omega$ is a holomorphic symplectic  form on $M$, and $c_t $ tends to a positive number $ c_0$ when $t\rightarrow 0$.

For any $t\in(0,1]$, there exists a family of K\"ahler metrics $\omega_t^{SF}$ on $M_0$,  such that $\omega_t^{SF}|_{M_w}$ is the flat metric in the class $t\alpha|_{M_w}$. Such metrics are called  {\it semi-flat}, and we recall their construction here.
 Note that $M_0$ is obtained by the quotient of the holomorphic cotangent bundle $T^*N_0$ by a lattice subbundle $\Lambda$. More precisely, we have a covering map $p:T^*N_0\rightarrow M_0$, so that $p(\Lambda)=\sigma(N_0)$, and the pull-back  $p^*\Omega$ is the canonical holomorphic symplectic form on $T^*N_0$.
 If $U\subset N_0$ is a small open disk, we can choose a holomorphic coordinate $w$ on $U$ so that $\Lambda\cap T^*U={\rm Span}_{\mathbb Z}\{dw,\tau(w) dw\}$, where $\tau(w)$ is the period of the elliptic curve $M_w$. Under the trivialization $T^*U\cong U\times\mathbb C$ given by $zdw\mapsto(w,z)$, we see $p^*\Omega=dw\wedge dz$.  Note that the $(1,1)$-form $$ i\partial\overline{\partial}{\rm Im}(\tau)^{-1}({\rm Im}(z))^2 =\frac i2W(dz+bdw)\wedge\overline{(dz+bdw)}$$ is invariant under the translation of any local constant section of $\Lambda$ (cf. Section 3 in \cite{GTZ}),  where
\be
W={\rm Im}(\tau)^{-1}\qquad{\rm and}\qquad b= -\frac{{\rm Im}(z)}{{\rm Im}(\tau)}\frac{\partial\tau}{\partial w}.\nonumber
\ee
Thus the above $(1,1)$-form  can be regarded as living on $f^{-1}(U)$.   The semi-flat metric is defined as
\be
\label{SFmetric}
\omega_t^{SF}  = \frac i2\left(tW(dz+bdw)\wedge\overline{(dz+bdw)}+W^{-1} dw\wedge d\bar w\right).
\ee  For simplicity we denote $\omega^{SF}:=\omega_1^{SF}$, which we use as a fixed base metric, and $\theta= dz+bdw$.

We now state our decay result for $\omega_t$ as $t\rightarrow 0$, which is contained in Theorem \ref{ttm-decay} (see Appendix A below).  Given $U\subset N_0$,
 \cite{GTZ} asserts that
there exists a local section $\sigma_0$ such that for any $\ell\geq 0$,   $$ \|T_{\sigma_0}^*\omega_t- \omega^{SF}_t\|_{C_{\rm loc}^\ell(M_U, \omega^{SF}_t)}\rightarrow 0,$$ when $t\rightarrow \infty$,  where $T_{\sigma_0}$ denotes the fiberwise  translation by $\sigma_0$ (cf. Lemma 4.7 in \cite{GTZ}).  Theorem \ref{ttm-decay} shows that there is
a $(1,1)$-form     $\chi_t$   satisfying $\chi_t \rightarrow 0$ in  $C^\infty$ as $t\rightarrow 0$,  so that  $T_{\sigma_0}^*\omega_t$ approaches to $ \omega^{SF}_t+ f^*\chi_t$ faster than any polynomial rate, i.e.   $$ T_{\sigma_0}^*\omega_t= \omega^{SF}_t+ f^*\chi_t+ o(t^{\frac{\nu}{2}}), $$ for any $\nu\gg 1$.

In the proof of the main theorem we need a slightly stronger statement.  The difference between $T_{\sigma_0}^*\omega_t$ and $\omega^{SF}_t$ can be written out in components in the fiber and base directions:
$$T_{\sigma_0}^*\omega_t- \omega^{SF}_t=\varphi_{t, z \bar{z}}dz\wedge d\bar{z}+\varphi_{t, w \bar{w}}dw\wedge d\bar{w}+\varphi_{t, w \bar{z}}dw\wedge d\bar{z}+ \varphi_{t, z \bar{w}}dz\wedge d\bar{w}. $$ We need the following important lemma, which  is a direct consequence of Lemma  \ref{decay}.

 \begin{lem}\label{lem-decay}
For any $\nu\gg  1$ and $\ell\geq 0$, there is a constant $C_{\ell,\nu}>0$ such that on $M_{U'}$, $U' \subset U$,  $$\|\varphi_{t, w \bar{w}}-\chi_{t,w\bar{w}}\|_{C^0}\leq C_{0,\nu}t^{\frac{\nu}{2}}, $$  $$\|\frac{\partial}{\partial z}\varphi_{t, w \bar{w}}\|_{C^\ell}+\|\frac{\partial}{\partial \bar{z}}\varphi_{t, w \bar{w}}\|_{C^\ell}+ \|\varphi_{t, z \bar{z}}\|_{C^\ell}+ \|\varphi_{t, z \bar{w}}\|_{C^\ell}+  \|\varphi_{t, w \bar{z}}\|_{C^\ell}\leq C_{\ell,\nu} t^{\frac{\nu}{2}}, $$ and $\chi_{t, w \bar{w}} \rightarrow 0$ in the $C^\infty$-sense when $t\rightarrow 0$. Here $\chi_t=\chi_{t,w\bar{w}}dw\wedge d\bar{w} $, and the $C^\ell$-norms are calculated using the   fixed K\"{a}hler metric $\omega^{SF}$ on $M_U$.
\end{lem}

In this section we also recall the blow-up limit of $t^{-1} \omega_t$, which shows up in the analysis to follow.   Let $t_k \rightarrow 0$ and $w_k \rightarrow w_0$ in $U\subset N_0$.
  By \cite{GTZ}, $$(M, t_k^{-1} \omega_{t_k}, p_k) \rightarrow (\mathbb{C} \times  M_{w_0}, \omega_\infty =\omega_{w_0}^{F}+\frac{i}{2}W^{-1}(w_0)d\tilde{w}\wedge d\bar{\tilde{w}}, p_0), $$ in the $C^\infty$-Cheeger-Gromov sense, where $w_k=f(p_k)$,  $p_k\rightarrow p_0\in M_{w_0} $,  $\omega_{w_0}^{F}$ is the flat K\"{a}hler metric representing $\alpha|_{M_{w_0}}$, i.e. $\omega_{w_0}^{F}=\omega^{SF}|_{M_{w_0}} $,  and $\tilde{w}$ denotes the coordinate of $ \mathbb{C}$.  More precisely, if
    $D_r=\{\tilde{w} \in \mathbb{C}| |\tilde{w}|<r\}$,    we define smooth embeddings $\Phi_{k,r}: D_r \times  M_{w_0}  \rightarrow M_U$ by $$( \tilde{w},  a_1+a_2 \tau (w_0))\mapsto (w_k +\sqrt{t_k}\tilde{w},  a_1+a_2 \tau(w_k+\sqrt{t_k}\tilde{w})),  \  \  a_1,a_2 \in\mathbb{R}/\mathbb{Z},$$  where we identify $M_U$ with $(U\times\mathbb{C})/{\rm Span}_{\mathbb{Z}}\{1, \tau\}$.
     If $z=a_1+a_2 \tau (w_0)$, then $ a_1+a_2 \tau(w_k+\sqrt{t_k}\tilde{w})=z +h_k$, where $$h_k=i(2{\rm Im}\tau(w_0))^{-1}(\bar{z}-z)(\tau(w_k+\sqrt{t_k}\tilde{w})-\tau(w_0)), $$  which satisfies that
     $\|h_k\|_{C^{\ell}}\rightarrow 0$ when $t_k \rightarrow 0$.
     Therefore
      $$\Phi_{k,r}^*(dz +bdw)= dz+dh_k +\sqrt{t_k }(b-{\rm Im}h_k({\rm Im}\tau)^{-1}\partial_w \tau )d\tilde{w}\rightarrow dz,$$ in the $C^\infty$-sense.    Clearly,   $d\Phi_{k,r}^{-1}I d\Phi_{k,r} \rightarrow I_\infty$, where $I$ is the complex structure of $M$ and $I_\infty$ denotes the complex structure of $  \mathbb{C}\times M_{w_0}$, and \begin{equation}\label{blowup1}\Phi_{k,r}^*  t_k^{-1} \omega_{t_k}^{SF} \rightarrow \omega_\infty=\frac i2\left(W(w_0)dz \wedge d\bar{z} +W^{-1}(w_0) d\tilde{w}\wedge d\bar {\tilde{w}}\right),  \end{equation} in the $C^\infty$-sense on $ D_r\times M_{w_0}$.   Furthermore,  \begin{equation}\label{blowup2}(T_{\sigma_0} \circ \Phi_{k,r})^*  t_k^{-1} \omega_{t_k} =\Phi_{k,r}^*  t_k^{-1}T_{\sigma_0}^* \omega_{t_k} \rightarrow \omega_\infty,  \end{equation} in the $C^\infty$-sense, on $ D_r\times M_{w_0}$,    when $t_k\rightarrow 0$ by \cite{GTZ}.

\subsection{Fourier-Mukai transform} In this section, we recall a notion, called the  Fourier-Mukai transform (cf.  \cite{AP,LYZ,Ch3,Ch4} etc.), and we present a little variant of the standard construction  for the convenience of  the proof of Theorem \ref{thm-main2}.

  Let $N^o \subset N_0$ be a Zariski open subset, and $D^o \subset M_{N^o}$ be a smooth curve such that $f|_{D^o}: D^o \rightarrow N^o$ is a unbranched $n$-sheets  cover.
Note that the moduli space  of flat $U(1)$-connections on $D^o$ is the cohomology group $H^1(D^o, \mathcal{U}_c(1))\cong H^1(D^o,U(1))$, where $\mathcal{U}_c(1)$ is the $U(1)$-valued locally  constant sheaf. For any $\Theta \in H^1(D^o, \mathcal{U}_c(1))$, the Fourier-Mukai transform takes the pair $(D^o, \Theta)$ to    a unitary gauge equivalent class  $ \mathcal{FM}(D^o, \Theta)$ of  $U(n)$-connections on $M_{N^o}$. We review the construction   as the following.

 If $\tilde{M}^o= D^o\times_{N^o} M_{N^o}$ is  the base change, then the projection  $\tilde{f}: \tilde{M}^o\rightarrow D^o$ is a fibration with the fiber $\tilde{M}_p^o=M_{f(p)}$,
 and $\nu_D: \tilde{M}^o \rightarrow M_{N^o}$ is a unbranched $n$-sheets cover  satisfying
   $f\circ \nu_D= f|_{D^o} \circ  \tilde{f}$.
    We embed  $\tilde{M}^o= D^o\times_{N^o} M_{N^o} \hookrightarrow  M_{N^o}\times_{N^o} M_{N^o}$ via the natural  inclusion  $D^o\hookrightarrow  M_{N^o}$.
       Let $\Sigma = \nu_D^* \sigma$ and $\Delta= \tilde{M}^o \cap \Delta_0$, where  $\Delta_0$ denotes the diagonal of $  M_{N^o}\times_{N^o} M_{N^o}$. For any $x\in N^o$, and $q(x)\in M_x\cap D^o$, we have $\tilde{M}_{(x, q(x))}^o=M_x$, $\Sigma \cap \tilde{M}_{(x, q(x))}^o= \{\sigma(x)\}$, and $\Delta \cap \tilde{M}_{(x, q(x))}^o= \{q(x)\}$.  We regard  $\Sigma$ as the zero section of $\tilde{f}$, which is used to identify the fibers with elliptic curves, and view $\Delta$ as the pull back the multi-section $D^o$, which is a section of $\tilde{f}$.

   There is a $U(1)$-connection $A^o$ on the smooth trivial line bundle $   \tilde{M}^o \times \mathbb{C}$, which is obtained by the
       restriction of the Poincar\'{e} line bundle (cf. \cite{AP}) on $ M_{N^o}\times_{N^o} M_{N^o}$ by identifying $M_{N^o}$ with the Jacobian $\check{M}_{N^o}$.  We exhibit  $A^o$ explicitly.

 If $U\subset D^o$ is an open disc such that $f|_U:U\rightarrow f(U)$ is biholomorphic,  we choose the coordinate $w$ such that $\tilde{M}_{U}^o\cong T^*U/ {\rm Span}_{\mathbb{Z}}\{dw, \tau dw\}$, where $\tau(w)$ is the period of $\tilde{M}_w^o$. Here the section $\Sigma\equiv 0$ under this identification.   If $z$ denotes  the coordinate on the fiber, then  the holomorphic symplectic form $\nu_D^*\Omega=dw\wedge dz$, and $\Delta \cap \tilde{M}_{U}^o$ is given by a holomorphic function $q=q(w)$ on $U$, i.e. $\Delta \cap \tilde{M}_{U}^o=\{(w, q(w))\}\subset U\times \mathbb{C}/ {\rm Span}_{\mathbb{Z}}\{1, \tau \} $. We have the  $U(1)$-connection \begin{equation}\label{u(1)} A^o=  \pi ( {\rm Im} (\tau))^{-1} (q\bar{\theta}-\bar{q}\theta),  \end{equation} on  $\tilde{M}_{U}^o $, where $\theta=dz+bdw$.

    If $y_1$ and $y_2$ are real functions  defined  on $U\times \mathbb{C}$ by $z=y_1+\tau y_2$, then $dy_1$ and $dy_2$ are well-defined 1-forms on $\tilde{M}_U^o$. Note that $\tilde{M}_U^o$ is diffeomorphic to $U\times (\mathbb{R}/\mathbb{Z})^2$, and we can regard $y_1$ and $y_2$ as the  angle coordinates of $\mathbb{R}/\mathbb{Z}$.  We have the decomposition of the cotangent bundle $T^*\tilde{M}_U^o= {\rm Span}_{\mathbb{R}}\{dy_1, dy_2\}\oplus {\rm Span}_{\mathbb{R}}\{dx_1, dx_2\}$, where $w=x_1+ix_2$.
     Since  $dz=dy_1+\tau dy_2+y_2 d\tau$,  $2i{\rm Im} (\tau )y_2=z-\bar{z}$, we have  $\theta=dy_1+\tau dy_2$.  If we write  $q=q_1+\tau q_2$, then  \begin{equation}\label{u(1)1} A^o=  2\pi i (q_2dy_1-q_1dy_2). \end{equation} If we choose  another basis of the lattice ${\rm Span}_{\mathbb{Z}}\{1, \tau \}$, and let $y_1'$ and $y_2'$ be the corresponding angle coordinates, then $y_j'=\sum c_{ji}y_i$ and  $q_j'=\sum c_{ji}q_i$ with $\det (c_{ji})=1 $ and $c_{ji}\in\mathbb{Z}$.  A direct calculation shows that $A^o$ is independent of the choice of the basis, and therefore $A^o$ is a global defined  $U(1)$-connection on the trivial line bundle $   \tilde{M}^o \times \mathbb{C}$.

    Let $\check{y}_1$ and $\check{y}_2$ be the dual coordinates of $y_1$ and $y_2$ on the dual space $(\mathbb{R}^2)^*$, i.e. if we view  $(\mathbb{R}^2)^*$ as the cotangent space, then $\check{y}_1$ and $\check{y}_2$ are coordinates with respect to the basis  $dy_1$ and $dy_2$.  We
     identify $\mathbb{R}^2/\mathbb{Z}^2$ with the dual torus
    $(\mathbb{R}^2)^*/(\mathbb{Z}^2)^*$ via the   symplectic form $\omega=dy_2\wedge dy_1$, i.e. $v \mapsto \omega(v, \cdot)$. Then $q=(q_1,q_2)\in \mathbb{R}^2/\mathbb{Z}^2$ is mapped to  $\check{q}=(q_2,-q_1)$ in $(\mathbb{R}^2)^*/(\mathbb{Z}^2)^*$.  The  Poincar\'{e} line bundle is a line bundle on $U\times \mathbb{R}^2/\mathbb{Z}^2 \times (\mathbb{R}^2)^*/(\mathbb{Z}^2)^*$ with the  $U(1)$-connection $$A_P=2\pi i (\check{y}_1dy_1+\check{y}_2dy_2).  $$
    Thus $A^o= A_P|_{U\times \mathbb{R}^2/\mathbb{Z}^2\times \{\check{q}\}}$, which coincides with the constructions  \cite{AP,LYZ}.

  Let $\{U_\lambda| \lambda\in\Lambda\}$ be a locally finite open cover of $N^o$ such that any intersection $U_{\lambda_1}\cap \cdots \cap U_{\lambda_h}$ is contractible.  On any $U_\lambda$, there are holomorphic functions   $q_{1}, \cdots, q_{n}$, such that  $D^o \cap M_{U_\lambda}=\{(w, q_{1}(w)), \cdots, (w, q_{n}(w))| w\in U_\lambda\}$.  Furthermore,  $D^o \cap M_{U_\lambda}=U_\lambda^1 \cup \cdots \cup  U_\lambda^n$ is a disjoint union of open sets biholomorphic to $U_\lambda$ where $U_\lambda^j=\{(w, q_j(w))\}$, and $\{U_\lambda^j|\lambda\in\Lambda, j=1, \cdots, n \}$ is an open cover of $D^o $ such that any intersections are contractible.

    If  $\Theta\in H^1(D^o, \mathcal{U}_c(1))$, then we let $\{g_{\mu\lambda}^{ij_i}\}\in \mathcal{C}^1(\{U_\lambda^j\}, \mathcal{U}_c(1))$ be the cocycle  representing  $\Theta$, where $U_\mu^i \cap U_\lambda^{j_i} \neq\emptyset$, and $g_{\mu\lambda}^{ij_i}$ are  $U(1)$-valued   constant functions  on $U_\mu^i \cap U_\lambda^{j_i} $.  If $U_\mu^i \cap U_\lambda^j \cap U_\nu^k\neq\emptyset$, then $g_{\mu\lambda}^{ij}g_{\lambda\nu}^{jk}g_{\nu\lambda}^{ki}=1$.  We identify $\tilde{f}^*g_{\mu\lambda}^{ij_i}=g_{\mu\lambda}^{ij_i}$, and regard $g_{\mu\lambda}^{ij_i}$ as $U(1)$-valued  constant functions on $\tilde{M}^o_{U_\mu^i}\cap \tilde{M}^o_{U_\lambda^{j_i}}$.
     Note that $(g_{\mu\lambda}^{ij_i})^{-1}A^o g_{\mu\lambda}^{ij_i}+(g_{\mu\lambda}^{ij_i})^{-1}dg_{\mu\lambda}^{ij_i}=A^o$.  If $L_\Theta$ denotes the line bundle on $\tilde{M}^o $ given by the cocycle  $\{(\tilde{M}^o_{U_\mu^i}\cap \tilde{M}^o_{U_\lambda^{j_i}}, g_{\mu\lambda}^{ij_i})\}$,  then $A^o$ induces a $U(1)$-connection on $L_\Theta$ locally given  by (\ref{u(1)})  denoted still by $A^o$.

     The pushforward  $(\nu_D)_* L_\Theta $ is a rank $n$ bundle on $ M_{N^o}$ given by the transitions $g_{\mu\lambda}={\rm diag}\{g_{\mu\lambda}^{1,j_1}, \cdots, g_{\mu\lambda}^{n,j_n}\}$ on $M_{U_\mu }\cap M_{U_\lambda}$, where $U_\mu^{i} \cap U_\lambda^{j_i}\neq\emptyset$.  There is a natural $U(n)$-connection $\Xi$ on $(\nu_D)_* L_\Theta $ induced by $A^o$    given  locally by  \begin{eqnarray*} \Xi|_{M_{U_\lambda}}& =& {\rm diag}\{(\nu_D)_* A^o|_{\tilde{M}^o_{U_\lambda^{1}}}, \cdots, (\nu_D)_*  A^o|_{\tilde{M}^o_{U_\lambda^{n}}}\} \\ & =& \pi ( {\rm Im} (\tau))^{-1} ({\rm diag}\{q_{1}, \cdots, q_{n}\}\bar{\theta}-{\rm diag}\{\bar{q}_{1}, \cdots, \bar{q}_{n}\}\theta),  \end{eqnarray*} which satisfies  $g_{\mu\lambda}^{-1}\Xi|_{M_{U_\mu}}g_{\mu\lambda}+g_{\mu\lambda}^{-1}dg_{\mu\lambda}=\Xi|_{M_{U_\lambda}}$.
If $\{g_{\mu\lambda}'^{ij_i}\}$ is an  another cocycle representing  $\Theta$, then there is a cycle $\{s_{\lambda}^{j}\}\in \mathcal{C}^0(\{U_\lambda^j\}, \mathcal{U}_c(1))$ such that $g_{\mu\lambda}'^{ij_i} s_{\lambda}^{j_i}=s_{\mu}^{i}g_{\mu\lambda}^{ij_i}$ when  $U_\mu^{i} \cap U_\lambda^{j_i}\neq\emptyset$.
   If we define   $s_\lambda={\rm diag}\{s_{\lambda}^{1},\cdots, s_{\lambda}^{n} \}$ on $M_{U_\lambda}$, then  $g_{\mu\lambda}' s_{\lambda}=s_{\mu}g_{\mu\lambda}$, and   $\{s_\lambda| \lambda\in\Lambda\}$ induces a unitary gauge change of  $(\nu_D)_* L_\Theta $.

 \begin{defn}  The Fourier-Mukai transform $\mathcal{FM} (D^o, \Theta)$ of $(D^o, \Theta)$ is defined as the unitary gauge  equivalent  class $[\Xi] $ of the $U(n)$-connection  $\Xi$ on   $(\nu_D)_* L_\Theta $, i.e.  $$ \mathcal{FM} (D^o, \Theta)= [\Xi].$$
 \end{defn}

 For any $t\in (0,1]$, the semi-flat metric  $\omega^{SF}_t$ is HyperK\"{a}hler, and by using the HyperK\"{a}hler rotation, we can find a new complex structure and a symplectic form such that  $D^o$ is a special lagrangian submanifold.
   In \cite{LYZ}, it is shown  that the connection obtained by the Fourier-Mukai transform of a special lagrangian section satisfies the deformed Hermitian-Yang-Mills equation, and in the case of dimension 2,  the standard Hermitian-Yang-Mills  equation. In the present case, the bundle $(\nu_D)_* L_\Theta $ with the connection $\Xi$ splits locally, and therefore, it is a corollary of  \cite{LYZ} that $\Xi$ is an anti-self-dual connection.
    We  give  a   direct calculation  proof of this assertion.

\begin{prop}
\label{propFM} If $\Theta\in H^1(D^o, \mathcal{U}_c(1))$, then for any $\Xi\in \mathcal{FM} (D^o, \Theta)$, $\Xi$ is an anti-self-dual connection with respect to the semi-flat HyperK\"{a}hler structure $(\omega^{SF}_t, \Omega)$, $t\in (0,1]$, i.e. the curvature $F_\Xi$ satisfies that  $$ F_\Xi\wedge \omega^{SF}_t=0,  \  \  and  \  \  F_\Xi\wedge \Omega =0.$$
 \end{prop}

 \begin{proof}  Since the anti-self-dual equation is unitary gauge invariant, we only need to verify the split  case, i.e. $\Xi|_{M_{U_\lambda}} =  {\rm diag}\{ A^o|_{\tilde{M}^o_{U_\lambda^{1}}}, \cdots,  A^o|_{\tilde{M}^o_{U_\lambda^{n}}}\}$, where we identify $M_{U_\lambda}$ and $\tilde{M}^o_{U_\lambda^{j}}$ via $\nu_D$.  The curvature $$F_{\Xi}|_{M_{U_\lambda}} =  {\rm diag}\{F_{ A^o}|_{\tilde{M}^o_{U_\lambda^{1}}}, \cdots,  F_{A^o}|_{\tilde{M}^o_{U_\lambda^{n}}}\} , $$ and thus we need to prove that $F_{ A^o}$ satisfies the  anti-self-dual equation.

By
$\overline{\partial}\tau=0$, we have  $0=\overline{\partial}\tau_1+i\overline{\partial}\tau_2$, where
$\tau=\tau_1+i\tau_2$, and  $\partial_{\bar{w}} \bar{\tau}=\partial_{\bar{w}} \tau_1 -i\partial_{\bar{w}} \tau_2=-2i\partial_{\bar{w}} \tau_2$. Thus $$F_{A^o}^{0,2}=\pi\overline{\partial} (\tau_2^{-1} q\bar{\theta})=\frac{\pi q}{2i\tau_2^2}\partial_{\bar{w}} \bar{\tau} d \bar{z}\wedge d\bar{w}+\frac{\pi q}{\tau_2^2}\partial_{\bar{w}}\tau_2d \bar{z}\wedge d\bar{w}=0,  $$ which is equivalent to $F_{A^o}\wedge \Omega =0$.
By (\ref{u(1)1}),  $$F_{A^o}=d A^o=  2\pi i \sum_{j=1,2}(\partial_{x_j} q_2 dx_j \wedge dy_1-\partial_{x_j} q_1 dx_j \wedge dy_2),$$ and by (\ref{SFmetric}),
$$\omega_t^{SF}  = t dy_1 \wedge dy_2 +W^{-1} dx_1\wedge dx_2.$$
Thus $$F_{A^o} \wedge \omega_t^{SF}=0,$$ and we obtain the conclusion.
 \end{proof}

 Finally, we remark that   the split $\Xi $ obtained by the Fourier-Mukai transform is $T^2$-invariant, and thus reduces to
  solutions of Hitchin equations \cite{Hit} on the base $N^o$.

\subsection{Small energy estimates on collapsed K3 surfaces}  Finally, we review small energy estimates for curvatures of anti-self-dual connections with respect to collapsed metrics.

 As above, for $t\in (0,1]$, let   $\omega_t \in \alpha_{t} =t \alpha + f^*c_1(\mathcal{O}_{\mathbb{CP}^1}(1))$  be the unique Ricci-flat K\"{a}hler-Einstein metric in $\alpha_t$, and consider $\Xi_t$,  a family of anti-self-dual connections on $P$ with  with respect to $\omega_t$.

For any $p\in M$, and $r>0$, we define  the local energy of the curvature $F_{\Xi_t}$ as
\be
\label{energy}
\E_t(p,r)=\frac{r^4}{{\rm Vol}_{\omega_t}(B_{\omega_t}(p,r))}\int_{{B_{\omega_t}(p,r)}} |F_{\Xi_t}|^2_{\omega_t}\,\omega_t^2.
\ee
This energy is a continuous function of $p$ and $r$.    By the Bishop-Gromov comparison Theorem, for  $r_1\leq r_2$ it holds
\be
\E_t(p,r_1)\leq\E_t(p,r_2)\qquad{\rm and}\qquad \E_t (p,0)=0.\nonumber
\ee

We have the following small energy estimate for curvatures of anti-self-dual connections, which is essentially Theorem 4.4 in  \cite{And}.

\begin{lem}
\label{energybound} There exists  a universal constant $\varepsilon >0$, independent of $t$, such that if \be
\E_t (p,r)\leq \varepsilon, \nonumber
\ee for  $p$ and $r$ satisfying  $p \in M_K$ and $B_{\omega_t}(p,r)\subset M_{K'}$ (for fixed compact subsets $K\subset K'\subset N_0$),  then \be
\sup_{{B_{\omega_t}(p,r/2)}} |F_{\Xi_t}|_{\omega_t}\leq \frac {C_{K'}\varepsilon^{\frac12}}{r^2} \nonumber
\ee for a constant $C_{K'}>0$.
\end{lem}

\begin{proof}
 By Lemma 4.4 of \cite{GTZ}, the curvature $R_{\omega_t}$  is bounded by a uniform constant $c_{K'}$ on $M_{K'}$. The  Weitzenb\"{o}ck formula (\ref{wei}) implies the  Bochner formula
 \be
\Delta_{\omega_t}  |F_{\Xi_t}|_{\omega_t}\geq -|F_{\Xi_t}|^2_{\omega_t}-c_{K'} |F_{\Xi_t}|_{\omega_t}.\nonumber
\ee
One can now carry over the exact argument from \cite{And}, consisting of  Moser iteration with the local Sobolev inequality $$ \frac{c_S}{3}\big(\frac{B_{\omega_t}(p,r)}{r^4}\big)^{\frac{1}{4}}\| \xi\|_{L^{4}(\omega_t)} \leq \| d \xi\|_{L^{2}(\omega_t)}$$ for any compactly supported function $\xi$ on $B_{\omega_t}(p,r)$, where $c_S$ is a universal constant (cf. (4.1) and  Theorem 4.1 in \cite{And}).
If we
keep track of the extra $c_{K'}$ term,  because this term is of lower order,  it does not effect the choice of the uniform constant $\tau$, which is thus independent of $K$ and  $K'$.\end{proof}

Choose $\varepsilon\ll 1$ such that $C_{K'}\varepsilon^{\frac12} \leq 4$.
This allows us to make the following definition.

\begin{defn}
For any $t\in(0,1]$, we define   $R_t(p)>0$ be the minimal  number such that
\be
\E_t (p, R_t(p))=\varepsilon. \nonumber
\ee
\end{defn}

In particular, for any compact set $K\subset N_0$, and $p\in M_K$, as long as $R_t(p)$ is small enough, it holds
\be
\label{curvradconclusion}
|F_{\Xi_t}|_{\omega_t}(p)\leq  4 R_t(p)^{-2},
\ee
 and for any $r\geq R_t(p)$,  $$\E_t(p, r)\geq \varepsilon.$$

\section{The main theorems}
In this section, we present the main theorems  of this paper, and demonstrate its applications to SYZ mirror symmetry of K3 surfaces.

 \begin{thm}\label{thm-main}

Let $M$ be a projective  elliptically fibered  $K3$ surface with fibration $
f:M\rightarrow N\cong \mathbb{CP}^1.$ Assume $f$ has    a section $\sigma:N\rightarrow M$, and assume it has only singular fibers of Kodaira type $I_1$ and type $II$.
  Let $\Omega$ be a holomorphic symplectic form on $M$, and let  $\omega_t \in \alpha_{t} $ be the unique Ricci-flat K\"{a}hler-Einstein metric in   $\alpha_{t} =t \alpha + f^*c_1(\mathcal{O}_{\mathbb{CP}^1}(1))$, $t\in (0,1]$, where $\alpha$ is an ample class on $M$. Let $P$ be a principal $SU(n)$-bundle on $M$, and let $\mathcal{V} $ be the smooth vector bundle of rank $n$ equipped with a Hermitian metric  $H$ induced by $P$, i.e. $\mathcal{V}=P\times_{\rho}\mathbb{C}^n$.

Assume there exists a family of anti-self-dual $SU(n)$-connections $\Xi_t$ on $P$ with respect to $(\omega_t, \Omega)$, i.e. $$   F_{\Xi_t}\wedge \omega_t=0, \  \  and \  \  F_{\Xi_t}\wedge \Omega=0,$$ with $t\in (0,1]$. Let $V_t$ denote the   holomorphic bundle of $\mathcal{V}$ equipped with the holomorphic structure induced by $\Xi_t$.
   Furthermore, assume:
  \begin{itemize}
  \item[i)]  The restriction of $V_t$ to a generic fiber of $f$ is semi-stable and regular.
  \item[ii)]  Let $D_{t} \in | n \sigma(N)+ m l |$ be the corresponding spectral cover  of $V_t$, where $0<m \leq c_2(\mathcal{V})$. As $t\rightarrow 0$, $$ D_{t} \rightarrow D_{0}  \  \  \  in  \  \  | n \sigma(N)+ m  l |.$$
    \item[iii)]  The limit $D_0$ can be written $$D_0= D_{0}^o+ D_0',  $$  where $D_{0}^o \in | n \sigma(N)+ m'  l |$   is reduced,  for some $0\leq m' \leq m$, and   $ D_0'\in |(m-m')l|$  consists of all  irreducible components of $D_0$ supported on   fibers.
         \end{itemize}
      Then  the following holds:
     \begin{itemize}
  \item[i)]
       For any sequence $t_k \rightarrow 0$, and any $p>2$, there exists  a Zariski open subset $N^o \subset N_0$,  a subsequence (still denoted $t_k$), a sequence of $L^{p}_2$ unitary gauge changes $u_k\in \mathcal{G}^{2,p}$ of $P|_{M_{N^o}}$,  and a  $L^{p}_1$  $SU(n)$-connection $\Xi_0$ on $P|_{M_{N^o}}$ so that on $M_{N^o}$
  $$u_k(\Xi_{t_k})  \rightarrow \Xi_0 $$
  in the locally  $L^{p}_{1}$ sense.  Here the norms are calculated   using a fixed K\"{a}hler metric on $M$, and the Hermitian metric $H$ on $\mathcal{V}$.
  \item[ii)]  The curvature $F_{\Xi_{t_k}}$ of $\Xi_{t_k}$ is locally  bounded, 
  i.e.  for any compact subset $K\subset N^o$, there exists a constant $C_K$ so that $$ \|F_{\Xi_{t_k}}\|_{C^0(M_K)}\leq C_K.$$
      \item[iii)] For any  $w \in N^o$ and   $0<\alpha <1$, there is a $C^{1,\alpha}$  unitary gauge $u_\infty$ on $M_w$ so that
    $u_\infty(\Xi_0|_{M_w})$ is a smooth  flat connection. This limiting connection satisfies  that the bundle $\mathcal{V}|_{M_{w}}$ equipped with the holomorphic structure  induced by $u_\infty(\Xi_0|_{M_w})$  is bi-holomorphic to $$\bigoplus_{q \in D_{0}^o \cap M_w}\mathcal{O}_{M_w}(q-\sigma(w)). $$\end{itemize}
   \end{thm}

\begin{rk}\label{remark1}\rm
We  remark that  $D_0' \in |(m-m')l|$ is supported on fibers over a finite number of points, and we refer to these fibers as  type III bubbles, which is the terminology used in the  previous  relevant works \cite{DS,Nis1,Nis2}.
\end{rk}

\begin{rk}\label{remark1+}\rm   There is a  topological constraint on $\mathcal{V}$ built into the above theorem, namely that  $$c_2(\mathcal{V} )\geq 2n-2.$$
 To see this, note that if $\sigma (N)$ is not  an irreducible component of $D_{0}^o$, then
  $D_{0}^o \cdot \sigma (N)=-2n+m' \geq 0$. Otherwise, $(D_{0}^o-\sigma (N))\cdot \sigma (N)=-2n+2+m' \geq 0$. In both cases, we have $m' \geq 2n-2$, which implies the inequality for the second Chern number.
\end{rk}

Let us demonstrate a case in which the hypotheses  of Theorem \ref{thm-main} hold. For  a given $m  \in \mathbb{N}$ and $s\in (0,1]$, let $D_s$  be  a family of effective  reduced irreducible divisors in the complete linear system $|n \sigma (N)+m l|$ such that as $s\rightarrow 0,$ $$D_s \rightarrow D_0=D_{0}^o+\sum_j D_j  \  \  \  {\rm in} \  \ |n \sigma (N)+m l|, $$  where $D_{0}^o$ is  reduced and  irreducible,  $D_{0}^o\in  |n \sigma (N)+m' l|$ for some $m'  \leq m$, and $\sum_j D_j \in |(m-m')l|$.  For example, we can take $D_s\equiv D$ for some fixed divisor.
 By Theorem \ref{stable construction}, we can construct a family of  holomorphic bundles $V_s$ of rank $n$ satisfying $c_1(V_s)=0$, the restriction of $V_s$ to any fiber $M_w$  is semi-stable and regular, and $D_s$ is the spectral cover of $V_s$.  Furthermore, Proposition 5.15 of \cite{FMW} asserts that $c_2(V_s)=m$, and therefore, all of $V_s$ are smoothly isomorphic to the same smooth bundle, since $SU(n)$ is simply connected. Now,
   Theorem 7.4 of \cite{FMW} shows that for any $s$ the bundle $V_s$ is stable with respect to  $f^{*} c_{1}(\mathcal{O}_{\mathbb{CP}^{1}}(1))+t \alpha$ for $0<t \ll 1$ and $t\leq s$. As a result, by Theorem \ref{DUY} (and taking a diagonal sequence) we obtain a family of anti-self-dual connections $\Xi_t$, for which the hypotheses  of Theorem \ref{thm-main} are verified.

\begin{thm}\label{thm-main2}
Under the  setup of Theorem \ref{thm-main}, the unitary  gauge equivalent class of  the limit connection $\Xi_0$ is  the Fourier-Mukai transform of   a  $\Theta\in H^1(D_{0}^o\cap M_{N^o}, \mathcal{U}_c(1))$, i.e.    $$\Xi_0 \in\mathcal{FM}(D_{0}^o\cap M_{N^o}, \Theta),$$ where $\mathcal{U}_c(1)$ is the  $U(1)$-valued locally constant sheaf.
\end{thm}

\subsection{Strominger-Yau-Zaslow mirror symmetry with   anti-self-dual connections}

We now apply Theorem \ref{thm-main} to Fukaya's Conjecture 5.5 in \cite{Fuk}, which relates the adiabatic limits of anti-self-dual connections to
  special Lagrangian cycles on the mirror Calabi-Yau manifolds. While describing the mirror symmetry background, we first consider the more general setup where $M$ is any projective  elliptically fibered  $K3$ surface admitting     a section. 

We normalize $\alpha_t$ by multiplying a constant, so that the normalized class $\tilde{\alpha}_t$ satisfies $ \tilde{\alpha}_t ^2=[{\rm Re}\Omega]^2=[{\rm Im}\Omega]^2$. Let $\tilde{\omega}_t\in \tilde{\alpha}_t$ be the Ricci-flat K\"{a}hler-Einstein metric in this class, and so $(\tilde{\omega}_t,  {\rm Re}\Omega,  {\rm Im}\Omega)$ is a HyperK\"{a}hler triple.   Using the HyperK\"{a}hler rotation, we have a family of complex structures $J_t$ with corresponding K\"{a}hler form and the holomorphic symplectic from
$$\omega_{J_t}={\rm Im}\Omega  \  \ {\rm and}  \  \    \Omega_{J_t}=\tilde{\omega}_t+i {\rm Re}\Omega.$$
Using $\Omega|_{M_w}=0$ and $\Omega|_{\sigma (N)}=0$, under $J_t$   the fibration $f$ becomes a special Lagrangian fibration, and the section $\sigma$ is a special Lagrangian section with respect to $\omega_{J_t}$ and $\Omega_{J_t}$.

Mirror symmetry for  K3 surfaces  is well understood (cf. \cite{AM,Do,GW0,GroII,ABC}), and in particular the SYZ mirror symmetry of K3 surfaces was studied in Section 7 of Gross \cite{GroII} and  in Gross-Wilson \cite{GW0}. For the reader's convenience we elaborate further on this setup.    Let $[\sigma]$ denotes the class of the section $\sigma(N)$ in $ H^2(M, \mathbb{Z})$ and $l$  the fiber class. Then we have the following intersection pairings:
$$l^2=0, \qquad [\sigma]\cdot l=1, \qquad  [\sigma]^2=-2,  \qquad[\omega_{J_t}]\cdot [\sigma]=0,$$
$$[{\rm Im}\Omega_{J_t}]\cdot [\sigma]=0,  \qquad [\omega_{J_t}]\cdot l=0,\qquad{\rm  and}\qquad[{\rm Im}\Omega_{J_t}]\cdot l=0.$$
Now, the  SYZ construction from Section 7 of \cite{GroII}  uses the choice of a B-field  $\mathbb{B}\in l^\bot /l \otimes \mathbb{R}/\mathbb{Z}$. However, Gross' assumptions are slightly different than those of the present paper. Namely, Gross assumes the K3 surface $M$ is generic, i.e. the picard group ${\rm Pic}(M)\cong \mathbb{Z}$, while in our case we have  $\dim {\rm Pic}(M)\geq 2$. Nevertheless, the proof of Theorem 7.3 of \cite{GroII} shows  that, in our case, if we further assume that $[\sigma]+(1+\frac{1}{2}[\omega_{J_t}]^2)l$ is an  ample class on $M$, and the  B-field $\mathbb{B}$ vanishes,  then the SYZ  mirror of $(M, \tilde{\omega}_t, \Omega_{J_t}) $ is  $f: M \rightarrow N$ equipped with the HyperK\"{a}hler structure $(\check{\omega}_t, \check{\Omega}_t)$ and the B-field $\check{\mathbb{B}}_t$ satisfying  $$ [\check{\Omega}_{t}]= (l\cdot [{\rm Re}\Omega_{J_t}])^{-1}([\sigma]+(1+\frac{1}{2}[\omega_{J_t}]^2)l-i[\omega_{J_t}]), \  \ \  [\check{\omega}_t]=(l\cdot [{\rm Re}\Omega_{J_t}])^{-1}[{\rm Im}\Omega_{J_t}],$$
  $${\rm and}  \  \  \   \check{\mathbb{B}}_t=(l\cdot [{\rm Re}\Omega_{J_t}] )^{-1} [{\rm Re}\Omega_{J_t}]-[\sigma]+{\rm mod} (l),  $$ on the cohomological level.

We  study the  case that  $[\sigma]+(1+\frac{1}{2}[\omega_{J_t}]^2)l$ is not necessarily ample. Recall that the Weierstrass model $\check{f}: \check{M}\rightarrow N$ of $f: M\rightarrow N$ is obtained by contracting the irreducible components of singular fibers of $f$, which do  not intersect  with the section $\sigma$  (cf. Chapter 7 in \cite{Fr1}). Denote by  $\pi: M \rightarrow \check{M}$  the contraction  morphism. Since $\pi$ contracts finitely many $(-2)$-curves,    $ \check{M}$ has only orbifold A-D-E singularities.

   \begin{prop}\label{lemma-mirror}  Normalize $\Omega$ so that $[{\rm Im}\Omega]^2=4$. The SYZ  mirror of $(M, \omega_{J_t}, \Omega_{J_t})$ with vanishing B-field  is  $(M, (l\cdot \tilde{\alpha}_t)^{-1} \check{\omega}, (l\cdot \tilde{\alpha}_t)^{-1}\check{\Omega})$ with the B-field $ \check{\mathbb{B}}_t$, where $$ \check{\Omega}= \pi^* \omega_{\check{M}}- i{\rm Im}\Omega, \  \  \  \check{\omega}={\rm Re}\Omega,  \  \  {\rm and} $$  $$ \check{\mathbb{B}}_t=(l\cdot \tilde{\alpha}_t)^{-1} \tilde{\alpha}_t-[\sigma]+{\rm mod} (l). $$ Here $ \omega_{\check{M}}$ is the  Ricci-flat K\"{a}hler-Einstein metric, possibly in the orbifold sense, such that     $\pi^* \omega_{\check{M}} \in c_1(\mathcal{O}_{M}(\sigma (N)+3l))$.
 \end{prop}

 \begin{proof} Firstly,  note that $([\sigma]+3l)^2= 4 >0$.
  Now, let $D$ be   an irreducible  curve   such that $( [\sigma]+3l)\cdot [D]\leq 0$. If $[D]\cdot l >0$, then $[\sigma]\cdot [D]<0$. Thus   $D=\sigma$, and   $( [\sigma]+3l)\cdot [D]=1>0$, which is a contradiction.
    We obtain that $[D]\cdot l\leq 0$, and  $D$ is an irreducible component of a fiber. Thus  $[D]\cdot l=0$, and  $[\sigma]\cdot [D]\leq 0$, which implies that $[\sigma]\cdot [D]=0$, and  $D$ is an irreducible component of a singular fiber of $f$ which does not intersect with $\sigma$. Therefore $[\sigma]+3l$ is nef and big, and  an irreducible  curve $D$  satisfies   $( [\sigma]+3l)\cdot [D]= 0$ if and only if $D$ is an irreducible component of a singular fiber of $f$ which does not intersect with $\sigma$.   There is an ample  class  $\alpha_{\check{M}}$ on the  Weierstrass model $\check{M}$ such that
      $[\sigma]+3l=\pi^* \alpha_{\check{M}}$, and by \cite{KT}, there exists a unique Ricci-flat K\"{a}hler-Einstein metric $\omega_{\check{M}}\in \alpha_{\check{M}}$ on $\check{M}$ in the orbifold sense.

       Since $[\pi^*\omega_{\check{M}}]^2=([\sigma]+3l)^2=  [{\rm Im}\Omega]^2= [{\rm Re}\Omega]^2, $  $(\pi^*\omega_{\check{M}}, {\rm Re}\Omega, {\rm Im}\Omega)$ is a HyperK\"{a}hler triple on $\pi^{-1}(\check{M}_{reg})$.
 By using the  HyperK\"{a}hler rotation, we can find new complex structure $K$, and define a family of HyperK\"{a}hler  structures
 $$ \check{\Omega}_{t}= (l\cdot \tilde{\alpha}_t)^{-1}(\pi^* \omega_{\check{M}}-i{\rm Im}\Omega), \  \  \  \check{\omega}_t=(l\cdot \tilde{\alpha}_t)^{-1} {\rm Re}\Omega,  $$
 which satisfy
  $$ [\check{\Omega}_{t}]= (l\cdot [{\rm Re}\Omega_{J_t}])^{-1}([\sigma]+3l-i[\omega_{J_t}]), \  \ {\rm and}  \   \  [\check{\omega}_t]=(l\cdot [{\rm Re}\Omega_{J_t}])^{-1}[{\rm Im}\Omega_{J_t}].$$ By letting $$ \check{\mathbb{B}}_t=(l\cdot [{\rm Re}\Omega_{J_t}] )^{-1} [{\rm Re}\Omega_{J_t}]-[\sigma]+{\rm mod} (l), $$  the proof of Theorem 7.3 in  \cite{GroII} shows that
  $(M,  \check{\omega}_{t},  \check{\Omega}_t)$ with $\check{\mathbb{B}}_t$  is the SYZ
 mirror of $(M, \omega_{J_t}, \Omega_{J_t})$, i.e.  $(f: M_{N_0}\rightarrow N_0,  \check{\omega}_{t},  \check{\Omega}_t)$ is the dual special Lagrangian fibration of $(f: M_{N_0}\rightarrow N_0,  \omega_{J_t}, \Omega_{J_t})$.
  \end{proof}

 We now assume that $M$ satisfies the hypotheses of Theorem \ref{thm-main}, which gives $M=\check{M}$ and $\pi$ is the identity. We can now see how Theorem \ref{thm-main} applies to  Conjecture 5.5 in \cite{Fuk}.  In our setup,  the anti-self-dual connection $\Xi_t$ and the complex structure $J_t$ induce a holomorphic structure on $\mathcal{V}$ for any $t\in (0,1]$, and $\Xi_t$ satisfies the  Hermitian-Yang-Mills equation $$ F_{\Xi_t}\wedge \omega_{J_t} =0,  \  \  {\rm and}  \  \   F_{\Xi_t}\wedge \Omega_{J_t} =0.$$ The spectral cover $D_t$ and the limit $D_0$ are special Lagrangian cycles with respect to the mirror HyperK\"{a}hler structure $( \check{\omega},  \check{\Omega})$.  We now rephrase Theorem \ref{thm-main} and Theorem \ref{thm-main2} in the context of SYZ mirror symmetry.

 \begin{thm}\label{conj} Under the assumptions   of Theorem \ref{thm-main},   for any sequence $t_k \rightarrow 0$ and any $p>2$, there exists  an   open dense  subset $N^o \subset N_0$,  a subsequence (still denoted $t_k$), a sequence of $L^{p}_2$ unitary gauge changes $u_k$ of $P$,  and a  $L^{p}_1$  $SU(n)$-connection $\Xi_0$ on $P|_{M_{N^o}}$ so that
  $$u_k(\Xi_{t_k})  \rightarrow \Xi_0 $$   in the locally   $L^{p}_1$ sense  on $M_{N^o}$. Here the norms are calculated by using a fixed metric on $M$.

   For any  $w \in N^o$,  the restriction of $\Xi_0$ to the fiber $M_w$, denoted $\Xi_0|_{M_w}$, is $C^{1,\alpha}$ gauge equivalent to a smooth  flat $SU(n)$-connection
 $$u_\infty ( \Xi_0|_{M_w} )=  \frac{\pi}{{\rm Im} (\tau)}({\rm diag}\{q_{1}(w), \cdots, q_{n}(w)\}d\bar{z}-{\rm diag}\{\bar{q}_{1}(w), \cdots, \bar{q}_{n}(w)\}dz), $$
 where $u_\infty\in\mathcal{G}^{1,\alpha}(M_w)$,
      $M_w \cong \mathbb{C}/\Lambda_{\tau}$,  $\Lambda_{\tau}={\rm Span}_{\mathbb{Z}}\{1, \tau\}$,  $\sigma(w)=0$, 
      and   $z$ denotes the coordinate on $\mathbb{C}$.
    As $w$ varies, $\{q_{1}(w), \cdots, q_{n}(w)\}\subset M_w$ forms a special Lagrangian multisection of $f^{-1}(N^o) \rightarrow N^o$ with respect to the SYZ mirror HyperK\"{a}hler structure  $( \check{\omega},  \check{\Omega})$, and its closure $D_{0}^o$ is a  special Lagrangian cycle, i.e. $$\check{\omega}|_{D_{0}^o}\equiv0,  \  \  {\rm and}  \  \  {\rm Im}\check{\Omega}|_{D_{0}^o}\equiv0. $$
 The  family of special Lagrangian submanifolds $D_t$ with respect to  $( \check{\omega},  \check{\Omega})$ converges  to $D_{0}^o$ on $f^{-1}(N^o)$ in the locally  Hausdorff sense.
Furthermore,
 the unitary  gauge equivalent class of  the limit connection $\Xi_0$ is  the Fourier-Mukai transform of   a flat $U(1)$-connection  $\Theta$ on $D_{0}^o\cap M_{N^o}$, i.e.    $$\Xi_0 \in\mathcal{FM}(D_{0}^o\cap M_{N^o}, \Theta).$$
 \end{thm}

 Conversely, if $D$ is a smooth   special Lagrangian submanifold with respect to $( \check{\omega},  \check{\Omega})$ on $M$ such that $D$  represents  $n [\sigma ]+m l \in H_2(M, \mathbb{Z})$ for some  $m  \in \mathbb{N}$, and $\Theta$ is   a flat $U(1)$-connection  on $D$,  then  $D$ is a smooth holomorphic curve in $M$.
 The argument in Section 3.1 shows that there is a stable bundle  $V$ of rank $n$ with respect to  $f^{*} c_{1}(\mathcal{O}_{\mathbb{CP}^{1}}(1))+t \alpha$ for $0<t \ll 1$. The anti-self-dual connections $\Xi_t$ on $V$ are also Hermitian-Yang-Mills with respect to  $(\omega_{J_t}, \Omega_{J_t})$.

 In the context of mirror symmetry, a special Lagrangian submanifold with a flat $U(1)$-connection is called an A-cycle, and a Hermitian-Yang-Mills connection  on a complex submanifold is called a B-cycle (cf. \cite{LYZ,JY,Va}).  The   correspondence between B-cycles  and A-cycles is motivated by the study of homological mirror symmetry via  the SYZ construction in \cite{AP,Fuk1,Fuk}, and the extended mirror symmetry with bundles \cite{LYZ,Va}.
  Theorem \ref{conj} says that in the current case, the adiabatic limit of B-cycles is corresponding to an A-cycle on the mirror K3 surface.

\subsection{Remarks}
We conclude this section with a few more remarks.

\begin{rk}\label{remark2} \rm  Note that the Levi-Civita connection of the Ricci-flat K\"{a}hler-Einstein metric $\omega_t$ is an anti-self-dual connection. However Theorem \ref{thm-main}  does not apply to this case due to the following.  If $M_w$ is a smooth fiber, then the restriction of the tangent bundle of $M$ satisfies a short exact sequence $$0 \rightarrow TM_w \rightarrow TM|_{M_w} \rightarrow f^*T_w N \rightarrow0,$$ and $TM|_{M_w}$ is S-equivalent to $\mathcal{O}_{M_w}\oplus \mathcal{O}_{M_w}$.  Thus  the special cover of $TM$ is $D_{TM}=2\sigma(N)$, and is not reduced.  Consequently,   the hypotheses of  Theorem \ref{thm-main} are not satisfied.

The curvature $F_{\Xi_t}$ in Theorem \ref{thm-main} behaves very differently  from the curvature of the Ricci-flat K\"{a}hler-Einstein metric $\omega_t$.  In the metric case, the  curvature $R_{\omega_t}$ of $\omega_t$ is bounded away from the singular fibers along the collapsing of $\omega_t$, i.e. $$\sup_{M_K}| R_{\omega_t}|_{\omega_t}\leq C_K,$$ for any compact subset $K\subset N_0$, by \cite{GW,GTZ}.  Furthermore, there is a more general result in \cite{CT} that asserts the boundedness of curvatures  of  sufficiently collapsed  Ricci-flat Riemannian  Einstein metrics  $\mathrm{g}$ on 4-manifolds away from finite metric balls.   The readers are referred to \cite{CT} for details.

In Theorem \ref{thm-main},
it is shown  that the curvature  $F_{\Xi_t}$ is bounded with respect to any fixed metric on $M_U$.
However,
  $F_{\Xi_t}$ can not be bounded with respect to the collapsed metric $\omega_t$ as the following demonstrates. If it were bounded, then  Proposition \ref{prop2+0} of Section 7  shows that on any $U\subset N^o$,  \bea  \int_{U}\sum_{j=1,2}\|\partial_{x_j}A_{0,t}\|_w^2  dx_1dx_2 &  \leq &  C( \|F_{\Xi_t}\|_{L^2(M_{U}, \omega_t)}^2  +t)\nonumber\\ & \leq &  C (\sup_{M_U}| F_{\Xi_t}|_{\omega_t}^2{\rm Vol}_{\omega_t}(M_U)+t)  \nonumber \\ & \leq & C t \rightarrow 0,   \nonumber\eea  where $x_1$ and $x_2$ are coordinates on $U$, which implies  $\partial_{x_j}A_{0}\equiv 0$, $j=1,2$.  Thus $\partial_{x_j}({\rm Im} (\tau)^{-1} q_i(w))\equiv 0$, $j=1,2$, and $q_i(w)=c_i (\tau(w) -\bar{\tau}(w))$ for constants $c_i\in \mathbb{C}$, $i=1, \cdots, n$.  Note that $q_i(w)$ is holomorphic, and $\tau(w)$ is not constant as the fibration $f$ is a Weierstrass fibration. We have $c_i=0$ and $q_i(w)\equiv 0$, $i=1, \cdots, n$. Hence $D_0^o\cap M_U=n \sigma (U)$, which contradicts to the assumption of $D_0^o$ being  reduced.
\end{rk}

\begin{rk}\label{remark3-}\rm   Theorem \ref{thm-main} is a compactness result, i.e. the convergence of $\Xi_t $ occurs along  subsequences $t_k$. The convergence along the parameter $t$ may hold   under  certain  stronger assumptions, for example  the follows.
For any $t\ll 1$, we assume  that $V_{t}|_{M_w}$  is regular semi-stable for any $w\in N$.
 As in Section 2.4, Proposition 5.7 of \cite{FMW} shows that $$V_t= (\nu_{D_t})_* (\mathcal{O}_{\tilde{M}}(\Delta_t - \Sigma_{D_t})\otimes \tilde{f}^*\tilde{L}_t)$$ for a  line bundle $\tilde{L}_t$ on $D_t$. If we assume further that $\tilde{L}_t$ converges to a $\tilde{L}_0$ on $D_0$ as divisors along the convergence of $D_t$ to $D_0$, then we expect that $\Xi_t $ converges away from finite fibers without passing to  any subsequence, which would be left  for  the future study.
\end{rk}

\begin{rk}\label{remark3} \rm There are many more  questions that the authors would like to investigate  in the future. Firstly,  we would like to understand what are the corresponding  algebraic geometric descriptions  of the type $I$  and type $II$ bubbles in the proof of Proposition \ref{estimate 1}. Secondly, we like to have  an  explicit formula for the second Chern number $c_2(\mathcal{V} )$ via  the  bubbles and the limit special cover $D_0$. Here a certain bubble tree convergence is expected.

Finally, we like to study the metric geometry of the moduli space of anti-self-dual Yang-Mills connections on collapsed K3 surfaces, inspired by the F-theory/heterotic string theory duality as in \cite{FMW0}.  For any $0<t \leq (\frac{n^3}{4}c)^{-1}$, let $\mathfrak{M}_t(n, c)$ be the  moduli space of anti-self-dual connections on $\mathcal{V}$ with respect to the HyperK\"{a}hler structure $(\omega_t, \Omega)$, where $c=c_2 (\mathcal{V})$, which is not empty (cf. Theorem \ref{stable criterion}).   By Theorem 7.10 in \cite{Kob},  $(\omega_t, \Omega)$ induces a HyperK\"{a}hler structure $(\omega_{\mathfrak{M},t}, \Omega_{\mathfrak{M},t}) $ on the regular locus  $\mathfrak{M}_t(n, c)^o$ of $\mathfrak{M}_t(n, c)$. Furthermore, it is expected that there is a holomorphic lagrangian fibration $\mathfrak{f}:\mathfrak{M}_t(n, c)^o\rightarrow \mathfrak{U}\subset |n \sigma (N)+ml|$ (cf. Section 2.4 of \cite{FMW0}).  For example, if $D\in |n \sigma (N)+ml|$ is smooth, then the fiber $\mathfrak{f}^{-1}(D)$ is the Jacobian $\mathfrak{J}(D)$ of $D$, which parameterises the flat $U(1)$-connections on $D$. We would investigate the degeneration  behavior of  $(\omega_{\mathfrak{M},t}, \Omega_{\mathfrak{M},t}) $ when $t\rightarrow0$ in   future study.
\end{rk}

\section{The proof of Theorem \ref{thm-main}}
\label{mainproof}

In this section we prove Theorem \ref{thm-main}, assuming some important estimates which will be proved in the subsequent sections. We begin with a bubbling result, which gives a decay estimate for curvature away from a finite  set. This set may depend on the chosen sequence of times $t_k\rightarrow 0$.

Since we are interested in the behavior of the restriction of the connections $\Xi_{t_k}$ to a fiber $M_w$, we use the notation $A_{t_k}(w)=\Xi_{t_k}|_{M_w}$. In general we write this fiberwise connection as $A_{t_k}$, as the dependence on $w$ is clear from context.


\begin{prop}
\label{estimate 1} If $\Xi_t$ 
is a family of anti-self-dual connections on $P$ with respect to $(\omega_t, \Omega)$, then for any sequence $t_k \rightarrow 0$, there is a Zariski open subset $N_1 \subset N_0$, and  a subsequence (still denoted $t_k$), so that  the curvature $F_{\Xi_{t_k}}$  satisfies  \be
\sup_{M_K}|F_{\Xi_{t_k}}|_{\omega_{t_k}}\leq  \frac{\epsilon_k}{t_k}  \nonumber
\ee
on any compact subset $K \subset N_1$.  Here the constants $\epsilon_k$ may depend on $K$, and  satisfy $\epsilon_k \rightarrow 0$ as $k \rightarrow \infty$. Consequently, for any $w\in K$ and $t_k \ll 1$,
     $$\|F_{A_{t_k}}\|_{C^0(\omega^{SF}|_{M_w})}\rightarrow 0,$$   and
 $V_{t_k}|_{M_w}$ is semi-stable. 
\end{prop}

Note that the above assumptions are slightly weaker than those used in Theorem \ref{thm-main}.  To prove the proposition, we follow a bubbling argument similar to arguments seen previously (for example \cite{DS}), however we present the details here for completeness.

\begin{proof} Suppose that  there exists a sequence of points $p_k \in M$ so that $f(p_k) \rightarrow w$ in $N_0$, and furthermore
  \begin{equation}\label{eq01}\liminf_{k \rightarrow\infty}t_k|F_{\Xi_{t_k}}|_{\omega_{t_k}}(p_k) >0.   \end{equation} We claim that there is a universal constant $\varepsilon>0$ such that   for any neighborhood $U_w$ of $w$,  $$ \int_{M_{U_w}}|F_{\Xi_{t_k}}|_{\omega_{t_k}}^2\omega_{t_k}^2 \geq \varepsilon,$$ for $k \gg 1$. Once this is demonstrated, by \eqref{cw0} there can only be a finite number of such $w$.

By \cite{GTZ}, for some $p\in M_w$ we have
  $$(M, t_k^{-1} \omega_{t_k}, p_k) \rightarrow (M_x \times \mathbb{C}, \omega_\infty =\omega_{w}^{F}+\frac{i}{2}W^{-1}(w)d\tilde{w}\wedge d\bar{\tilde{w}}, p)$$ in the pointed  $C^\infty$-Cheeger-Gromov sense, where $\omega_{w}^{F}$ is the flat K\"{a}hler metric representing $\alpha|_{M_w}$, i.e. $\omega_{w}^{F}=\omega^{SF}|_{M_w} $,  and $\tilde{w}$ denotes the scaled coordinate of $ \mathbb{C}$ (see Section 2.4).  More precisely, if
    $D_r=\{\tilde{w} \in \mathbb{C}| |\tilde{w}|<r\}$,   there are  smooth embeddings $\Phi_{t_k,r}: M_w \times D_r \rightarrow M_U$
   such that  $$\Phi_{t_k,r}^* t_k^{-1}\omega_{t_k} \rightarrow \omega_{\infty}, \  \  \ \Phi_{t_k,r}^* I\Phi_{t_k,r,*} \rightarrow I_{\infty},$$ in the $C^{\infty}$-sense on $M_w \times D_r$, where $I$ (resp. $I_\infty$) denotes the complex structure on $M$ (resp. $M_w \times \mathbb{C}$).

 We have two cases.  In the first case, for any compact subset $K\subset M_w \times \mathbb{C}$, there is a constant $C_K>0$ such that $$|F_{\Xi_{t_k}}|_{t_k^{-1}\omega_{t_k}}=t_k|F_{\Xi_{t_k}}|_{\omega_{t_k}}\leq C_K,$$  on $\Phi_{t_k,r}(K)$, $r\gg 1$.  By passing a subsequence, Uhlenbeck's strong  compactness   theorem shows that there is a sequence of unitary gauge transformations $u_{K,k}$, and an anti-self-dual $SU(n)$-connection $\Xi_\infty$ on $M_w \times \mathbb{C}$ such that
 $u_{K,k}(\Phi_{t_k,r}^*\Xi_{t_k})$ converges to $\Xi_\infty$ in the locally  $C^\infty$-sense on $K$.  Thus, in the $C^0$-sense on $K$,
  $$ \Phi_{t_k,r}^*|F_{\Xi_{t_k}}|_{t_k^{-1}\omega_{t_k}} \rightarrow |F_{\Xi_\infty}|_{\omega_\infty}, \  \  {\rm and}  \  \
 |F_{\Xi_\infty}|_{\omega_\infty}(p)>0. $$
  By \cite{Weh}, there is a constant $\mu=\mu(n)$ depending only on the group $SU(n)$, such that $$\int_{M_x \times \mathbb{C}} |F_{\Xi_\infty}|_{\omega_\infty}^2\omega_\infty^2 \geq \mu.$$ Furthermore if $n=2$, we know $\mu (2)=4\pi^2$.   This is called the bubble  of type $II$ in \cite{DS}. By choosing $K$ large enough,   $$   \int_{M_{U_w}}|F_{\Xi_{t_k}}|_{\omega_{t_k}}^2\omega_{t_k}^2 \geq  \int_{\Phi_{t_k,r}(K)}|F_{\Xi_{t_k}}|_{t_k^{-1}\omega_{t_k}}^2t_k^{-2}\omega_{t_k}^2\geq \frac{\mu}{2},$$ for $k\gg 1$.

  The second case is that there are $p_k'\in M$ such that $$d_{t_k^{-1}\omega_{t_k}}(p_k,p_k')<C< \infty, \  \  {\rm  and }, \ \ t_k|F_{\Xi_{t_k}}|_{\omega_{t_k}}(p_k')\rightarrow\infty,$$ when $k \rightarrow\infty$. In order to preform the bubbling argument, recall the following  point choosing lemma.

 \begin{lem}[Lemma 9.3 in \cite{DS}]
\label{pointpick}
Let $(Y,d_Y)$ be a complete metric space, and $\zeta$ be a continuous non-negative function. For any $y\in Y$, there exist $y' \in Y$ and $0<\rho \leq 1$ such that $$d_Y(y,y')\leq 1,  \  \  \sup_{B_{d_Y}(y',\rho)}\zeta \leq 2 \zeta (y'),  \  \ {\rm and}  \  \ 2\rho \zeta(y')\geq  \zeta(y).$$
\end{lem}

We apply this lemma to $\zeta=|F_{\Xi_{t_k}}|_{t_k^{-1}\omega_{t_k}}$, $y=p_k'$, and obtain $y'=p_k''$ and $0\leq \rho \leq 1$.
We further rescale the metric, and $(M, |F_{\Xi_{t_k}}|_{\omega_{t_k}}^{-1}(p_k'')\omega_{t_k}, p_k'')$ converges to the standard Euclidean space $(\mathbb{C}^2,\omega_E, 0)$ in the smooth  Cheeger-Gromov sense by passing to a subsequence.  The same argument as above shows that $\Xi_{t_k}$ smoothly converges to an non-trivial  anti-self-dual $SU(n)$-connection $\Xi_\infty'$ on $\mathbb{C}^2$ by passing to certain unitary gauge changes and subsequences. We now have $$\int_{\mathbb{C}^2} |F_{\Xi_\infty'}|_{\omega_E}^2\omega_E^2 \geq \tau,$$ where $\tau$ is the constant in Lemma \ref{energybound}.  This is called a bubble of type $I$, and is standard  in the study of Yang-Mills fields on 4-manifolds (cf. \cite{DK,FU}).   Just as above,   $$   \int_{M_{U_w}}|F_{\Xi_{t_k}}|_{\omega_{t_k}}^2\omega_{t_k}^2 \geq  \int_{\Phi_{K,k}(K)}|F_{\Xi_{t_k}}|_{t_k^{-1}\omega_{t_k}}^2t_k^{-2}\omega_{t_k}^2\geq \frac{\tau}{2},$$ for $k\gg 1$, where $K$ satisfies that $p_k'\in \Phi_{K,k}(K)$.  We obtain the claim by letting $\varepsilon= \frac{1}{2}\min \{\mu, \tau\}$.

Let $S_1$ be the set of points  $x\in N_0$ for which  there is a sequence $p_k \in M$ such that $f(p_k) \rightarrow w$ in $N_0$, and  (\ref{eq01}) is satisfied. By \eqref{cw0} $$ 8 \pi^2 c_2(\mathcal{V})= \lim_{k\rightarrow\infty}\int_{M}|F_{\Xi_{t_k}}|_{\omega_{t_k}}^2\omega_{t_k}^2\geq \sharp (S_1)\varepsilon,$$  and as a result $S_1$ is a finite set.  Therefore $N_1=N_0\backslash S_1$ is a Zariski open subset, and for any compact subset $K\subset N_1$,
$$
\sup_{M_K}t_k|F_{\Xi_{t_k}}|_{\omega_{t_k}}\leq   \epsilon_k \rightarrow 0,
  $$    when $k \rightarrow \infty$.

  Since $\Phi_{t_k,r}^*t_k^{-1}\omega_{t_k}$ converges smoothly  to $\omega_\infty$ on $M_w\in \mathbb{C}$ for $w\in K$, we have $$ \|F_{A_{t_k}}\|_{C^0(\omega^F)}\leq 2 \|F_{A_{t_k}}\|_{C^0(t_k^{-1}\omega_{t_k}|_{M_w})}\leq 2 \sup_{M_K}|F_{\Xi_{t_k}}|_{t_k^{-1}\omega_{t_k}}\rightarrow 0. $$  By Proposition \ref{semistablelemma},
    $V_{t_k}|_{M_w}$ is semi-stable, where as above $V_{t_k}$ denotes $\mathcal{V}$ equipped with the holomorphic structure induced by $\Xi_{t_k}$.
\end{proof}

Restricting to a fiber $M_w$, by the above proposition, weak Uhlenbeck compactness gives that for any $p>2$, there exists  a sequence of unitary gauge $u_{w,k}$ such that along a subsequence of  times, $u_{w,k}(A_{t_k})$ converges in $L^p_1$ to a flat $L^p_1$-connection $\Xi_{\infty,w}$ on $M_w$. In other words, we have  fiberwise convergence of $\Xi_{t_k}$ up  to gauge changes.  However, it is not clear yet that $\Xi_{t_k}$ has any limit when $t_k \rightarrow 0$ on $M_K$. For this, we need the stronger  assumptions in Theorem \ref{thm-main}, and further results and estimates.

We now work under the setup of Theorem \ref{thm-main}, and consider a sequence of connections $\Xi_{t_k}$ where $t_k\rightarrow 0$ as $k\rightarrow\infty$. Before we turn to the key estimates, we need to describe the explicit form of the holomorphic structure of the bundle $V_t$ in a local trivialization.

Note that   $f|_{D_{0}^o}: D_{0}^o \rightarrow N$ is an $n$-sheeted branched covering. If $S_{D_{0}^o}$ denotes  the subset of $D_{0}^o$ consisting all singular points of $D_{0}^o$ and all branch points of $f|_{D_{0}^o}$, then $f(S_{D_{0}^o})$ is a finite subset of $N$.
   We define a Zariski open subset  \begin{equation}\label{subset}N^o=N_1\backslash (f(D_0-D_{0}^o)\cup f(S_{D_{0}^o})).  \end{equation}



    On $N^o$,   $f|_{D_{0}^o}: f|_{D_{0}^o}^{-1}(N^o) \rightarrow N^o$ is an $n$-sheeted   unbranched covering, since $D_{0}^o$ is reduced. For any $w\in N^o$, $D_{0}^o \cap M_w$ consists $n$ distinct  points in $M_w$, i.e. $D_{0}^o \cap M_w= \{q_1, \cdots, q_n\}$ where $q_i \neq q_j$ for any $i\neq j$.  The trivial bundle  $\mathcal{V}|_{M_w}$ equipped with the holomorphic structure induced by $D_{0}^o \cap M_w$ is isomorphic to the flat holomorphic bundle
    $$ \mathcal{O}_{M_w}(q_1 - \sigma (w))\oplus \cdots \oplus  \mathcal{O}_{M_x}(q_n - \sigma (w)).$$
     Since $D_{t}$ converges to $D_0$ and $D_0-D_{0}^o \in |(m-m')l|$ is supported on   fibers, for any compact subset $K\subset N^o$ we have that $f: D_{t} \cap M_{K} \rightarrow K$ is an  $n$-sheeted unbranched  covering    for $t\ll 1$.  For any $w\in K$, $D_{t} \cap M_{w}=\{q_{1,t}, \cdots, q_{n,t}\}$ such that $q_{i,t} \neq q_{j,t}$ for any $i\neq j$,   and $q_{i,t} \rightarrow q_{i}$  when $t \rightarrow 0$. Furthermore,  $V_{t}|_{M_w}$ is semi-stable, which implies that   $V_{t}|_{M_w}$ is regular by Proposition 6.4 in \cite{FMW}, and  $$ V_{t}|_{M_w} \cong   \mathcal{O}_{M_w}(q_{1,t} - \sigma (w))\oplus \cdots \oplus  \mathcal{O}_{M_w}(q_{n,t} - \sigma (w)).$$

For any $t\ll 1$, there is a Zariski open subset $N_t^o \supset K$ such that $V_{t}|_{M_w}$, $w\in N_t^o$,   is regular semi-stable.
 Proposition 5.7 of \cite{FMW} asserts that $$V_t|_{M_{N_t^o}}= (\nu_{D_t})_* (\mathcal{O}_{\tilde{M}_{N_t^o}}(\Delta_t - \Sigma_{D_t})\otimes \tilde{f}^*\tilde{L}_t)$$ for a certain line bundle $\tilde{L}_t$ on $D_t\cap M_{N_t^o}$. Here, as in Section 2.4, $$\nu_{D_t}:\tilde{M}_{N_t^o}=D_t\times_{N_t^o} M\rightarrow M_{N_t^o},$$
 $\Sigma_{D_t} = \nu_{D_t}^* \sigma,$ and $\Delta_t= \tilde{M}_{N_t^o} \cap \Delta_0$ for the diagonal $\Delta_0$ of $ M\times_{N_t^o} M$ via  the natural  embedding $ \tilde{M}_{N_t^o} =D_t\times_{N_t^o} M \hookrightarrow  M\times_{N_t^o} M$.

Let $U\subset K \subset  N_t^o$ be an open subset biholomorphic to the unit disk, and $w$ be a coordinate on $U$. Then $M_U\cong (U\times \mathbb{C})/\Lambda$ for lattice subbundle $\Lambda={\rm Span}_{\mathbb{Z}}\{1, \tau\}$, where $\tau=\tau(w)$ varies holomorphically and is the period of the elliptic curve $M_w$. Furthermore under this identification the section $\sigma$ satisfies $\sigma \equiv 0$.  If $z$ is the coordinate on $\mathbb{C}$, we define real functions  $y_1$ and $y_2$ on $U\times \mathbb{C}$ by $z=y_1+\tau y_2$. Then $dy_1$ and $dy_2$ are well-defined 1-forms on $M_U$, and we have the decomposition of cotangent bundle  $T^*M_U\cong {\rm Span}_{\mathbb{R}}\{dy_1,dy_2\}\oplus {\rm Span}_{\mathbb{R}}\{dx_1,dx_2\}$, where $w=x_1+ix_2$.  Let $\theta=dy_1+\tau dy_2$, whose  restriction  $\theta|_{M_w}=dz$ on any fiber $M_w$.   Note that
  $\overline{\partial}\tau=0$, $d\tau=\partial \tau$ and $0=\overline{\partial}\tau_1+i\overline{\partial}\tau_2$, where
$\tau=\tau_1+i\tau_2$. Thus  $dz=dy_1+\tau dy_2+y_2 d\tau$,  $2i\tau_2y_2=z-\bar{z}$, and  $\theta=dz-\frac{z-\bar{z}}{2i\tau_2}\partial_w \tau dw=dz+bdw.$

  We fix the trivializations $P|_{M_U}\cong  M_U \times SU(n)$ and $\mathcal{V}|_{M_U}\cong  M_U \times \mathbb{C}^n$.  The unitary gauge group consists of $SU(n)$ valued functions, in other words $\mathcal{G}=C^\infty(M_U, SU(n))$, and the complex gauge group is  $\mathcal{G}_{\mathbb{C}}=C^\infty(M_U, SL(n, \mathbb{C}))$ under this trivialization. The respective Lie algebras are $\mathfrak{g}=C^\infty(M_U, \mathfrak{su}(n))$ and $\mathfrak{g}_{\mathbb{C}}=C^\infty(M_U, \mathfrak{sl}(n, \mathbb{C}))$.  Note that there is the decomposition $\mathfrak{g}_{\mathbb{C}}=\mathfrak{g}\oplus i\mathfrak{g}$ induced by $\mathfrak{sl}(n, \mathbb{C})=\mathfrak{su}(n)+i\mathfrak{su}(n)$, and if $s\in \mathfrak{g}_{\mathbb{C}}$ is Hermitian (given by $s^*=s$),  then $s\in i\mathfrak{g}$. Therefore any complex gauge $g$ can be written as $g=\exp(v+s)$, for a certain $v\in \mathfrak{g}$ and an $s \in i \mathfrak{g}$.

  Note that $D_{0}^o \cap M_U$ (resp. $D_{t} \cap M_U$) is given by n distinct holomorphic functions $q_j(w)$ (resp. $q_{j,t}(w)$),  and for any $j$, $q_{j,t}(w)\rightarrow q_{j}(w)$ in the $C^{\infty}$-sense as $t\rightarrow 0$. Thus $D_t \cap M_U$ consists of $ n $ distinct unit disks, and because  $\tilde{L}_t|_{D_t \cap M_U}$ is holomorphically trivial, we obtain  $$V_t|_{M_{U}}\cong \bigoplus_{j=1}^{n} \mathcal{O}_{M_{U}}(q_{j,t} (U)- \sigma (U)).  $$

  We define the background connections on the trivial bundle $\mathcal{V}|_{M_U}$   \begin{equation}\label{bconnections} A_{0,t}=  \pi ({\rm Im} (\tau))^{-1} ({\rm diag}\{q_{1,t}, \cdots, q_{n,t}\}\bar{\theta}-{\rm diag}\{\bar{q}_{1,t}, \cdots, \bar{q}_{n,t}\}\theta),  \end{equation} \begin{equation}\label{bconnections2} A_{0}=  \pi ( {\rm Im} (\tau))^{-1} ({\rm diag}\{q_{1}, \cdots, q_{n}\}\bar{\theta}-{\rm diag}\{\bar{q}_{1}, \cdots, \bar{q}_{n}\}\theta).   \end{equation}  Thus $ A_{0,t} \rightarrow  A_{0}$ in the $C^\infty$-sense when $t\rightarrow 0$, $V_t|_{M_w}$ is isomorphic to $\mathcal{V}|_{M_w}$ equipped with the holomorphic structure induced by  the flat connection $A_{0,t}|_{M_w}$, and $A_{0}|_{M_w}$ induces the holomorphic bundle structure $\bigoplus \limits_{i=1}^{n} \mathcal{O}_{M_w}(q_{i}(w)-\sigma(w))$.

 \begin{lem}
\label{background connection}   The unitary connection $A_{0,t}$ on $\mathcal{V}|_{M_U}$  induces the  holomorphic structure isomorphic to $V_t|_{M_U}$.
\end{lem}

\begin{proof}
In general, if $L$ is a holomorphic line bundle, and $h$ determines a Hermitian metric on $L$ in a local holomorphic trivialization, then the unique Chern connection is given by $A_h=\partial \log h$.  If $\rho$ is a local unitary frame, i.e. $|\rho|_h^2=h|\rho|^2\equiv 1$, then we have smooth trivialization of $L$ via $\rho\mapsto 1$, and under such trivialization, $A_h$ is  transformed to  $A=\overline{\partial}\log \rho- \partial \log \bar{\rho}$. A different choice of $\rho$ gives a unitary gauge  transformation of $A$.

Note that the holomorphic line bundle $ \mathcal{O}_{M_{U}}(q_{j,t} (U)- \sigma (U))$ is given by the multiplier $\{e_1\equiv 1, e_\tau =\exp (-2\pi i q_{j,t}(w))\}$, i.e. $ \mathcal{O}_{M_{U}}(q_{j,t} (U)- \sigma (U))$ is obtained by the quotient of $U\times \mathbb{C} \times \mathbb{C}$  via
 $$ (w,z, \xi) \sim (w,z+1, e_1 \xi),  \  \  \   (w,z, \xi) \sim (w,z+\tau, e_\tau \xi)$$
  (cf. Section 6 in Chapter 2 of \cite{GH}).  On $U\times \mathbb{C}$, if we let $$h=\exp \pi \big({\rm Im} (\tau)^{-1}(z-\bar{z})(q_{j,t}-\bar{q}_{j,t})\big), $$ then $h(w,z+1)=h(w,z)$ and $h(w,z+\tau)=|\exp (2\pi i q_{j,t}(w)) |^2h(w,z)$, and thus $h$ defines a Hermitian metric on $ \mathcal{O}_{M_{U}}(q_{j,t} (U)- \sigma (U))$.  If $$\rho=\exp \big( -\pi {\rm Im} (\tau)^{-1}  (z-\bar{z})q_{j,t}\big) , $$ then $h|\rho|^2=1$, $\rho(w,z+1)=\rho(w,z)$ and $\rho(w,z+\tau)=e_\tau \rho(w,z)$. Thus $\rho$ is a global
   unitary frame, and under the trivialization induced by $\rho$, the Chern connection $\Xi_{0,t,j}=\Xi^{1,0}_j+\Xi^{0,1}_j$ is given by $\Xi^{1,0}_j=- \overline{\Xi^{0,1}_j} $ and  $$\Xi^{0,1}_j=\overline{\partial}\log \rho=\pi {\rm Im} (\tau)^{-1}q_{j,t}d\bar{z}-\pi (z-\bar{z})q_{j,t}\overline{\partial}{\rm Im} (\tau)^{-1}=\pi {\rm Im} (\tau)^{-1}q_{j,t}\bar{\theta},  $$ by  $$\bar{\theta}=d\bar{z}-\frac{z-\bar{z}}{2i{\rm Im} (\tau)}\partial_{\bar{w}} \bar{\tau} d\bar{w}=d\bar{z}+\frac{z-\bar{z}}{{\rm Im} (\tau)}\partial_{\bar{w}} {\rm Im} (\tau) d\bar{w}.$$ We obtain  the desired conclusion.
\end{proof}

Since  $\Xi_t$ and $A_{0,t}$ induce the same holomorphic structure on $\mathcal{V}|_{M_U}$ over $M_U$, there is a complex gauge $g\in \mathcal{G}_{\mathbb{C}}$ such that $g(\Xi_t)=A_{0,t}$.  Note that $gg^*$ is   Hermitian, and $gg^*=e^{2s_t}$ for some  $s_t \in C^\infty (M_U, \mathfrak{sl}(n, \mathbb{C}))$ with  $s_t^*=s_t$.  If we let $u=e^{-s_t}g$, then $u^*=u^{-1}$, i.e. $u$ is a unitary gauge, and $g=e^{s_t} u$.  Therefore, by  a further unitary gauge change if necessary,  we assume that
\be
\label{hermitiangauge}
e^{s_t}(\Xi_t)=A_{0,t}
\ee
 for a Hermitian gauge $e^{s_t}$ on $M_U$.

In order to prove the main theorem, we need to improve the curvature estimates of Proposition \ref{estimate 1}.

\begin{prop}
\label{type3bubbling}  For any compact set $K\subset N^o$, there exists a constant  $C_K$ such that
\begin{equation}\label{eq02}
\sup_{M_K}|F_{\Xi_{t_k}}|_{\omega_{t_k}}\leq C_K {t_k}^{-\frac12}.\nonumber
\end{equation}
\end{prop}
The proof of this proposition can be found in Section 7. This implies the subsequence of connections $\Xi_{t_k}$ satisfies \eqref{hyp:curv}, which is the main assumption of
Proposition \ref{prop2} in Section \ref{C0boundsection}. Thus we can apply Proposition \ref{prop2} to $\Xi_{t_k}$ and achieve uniform $C^0$ control of the curvature, from which we conclude:

\begin{prop}
\label{thm01}
Along the sequence of connection $\Xi_{t_k}$, there exists a constant $C_1>0$ such that $$  \|F_{A_{t_k}}\|_{C^0(M_w)} \leq C_1t_k,  \  \  and \  \   \|F_{\Xi_{t_k}}\|_{C^0(M_K)} \leq C_1, $$ for  any $w\in K$. Consequently, for any $p>2$, by the  weak Uhlenbeck compactness theorem  \cite{U2}  there exists a subsequence (still denoted $t_k$),    a sequence of unitary gauge transformations $u_k\in \mathcal{G}^{2,p}$, and a limiting $L^p_1$ connection $\Xi_\infty$, so that   $$ u_k (\Xi_{t_k} ) \rightarrow \Xi_\infty$$  in $L^p_1(M_K)$. Here all norms are calculated by using a fixed K\"{a}hler metric on $M$.
\end{prop}
In order to prove Theorem \ref{thm-main}, we also  need a generalization of Theorem 1.1 in \cite{DJ}, which is a direct consequence of Lemma \ref{connectionC0lem}.

\begin{prop}
\label{lemma01}  For any  $w\in K$ and  $0<\alpha <1$, there exists a constant $C_2>0$ so that  $$\|A_{t_k}-A_{0, t_k}\|_{C^{0,\alpha}(M_w)} \leq C_2t_k.$$
\end{prop}

Granted these three propositions, we are ready to prove Theorem \ref{thm-main}.

\begin{proof}[Proof of Theorem \ref{thm-main}]
By Proposition \ref{thm01} and the Sobolev embedding theorem,  there exists $u_k\in \mathcal{G}^{1,\alpha}$ and a limiting $C^{0,\alpha}$-connection $\Xi_0$, so that
$$ u_k (\Xi_{t_k} ) \rightarrow \Xi_0$$  in $C^{0,\alpha}(M_K)$.   Thus, for any $w\in K$, the restriction $\Xi_0|_{M_w}$ of $\Xi_0$ is a $C^{0,\alpha}$-connection on $M_w$, and $ u_k (\Xi_{t_k} ) |_{M_w} $ converges to $\Xi_0|_{M_w}$ in the $C^{0,\alpha}$-sense. Proposition \ref{lemma01}, along with the fact that $ A_{0,t} \rightarrow  A_{0}$ in the $C^\infty$-sense, gives
   $$A_{t_k}\rightarrow A_0,$$ on $M_w$ in the $C^{0, \alpha}$-sense, where $A_0$ is given by  (\ref{bconnections2}).

Since  $$du_k= u_k \Xi_{t_k} |_{M_w}-  u_k (\Xi_{t_k} ) |_{M_w} u_k,$$ and the $u_k$ are unitary,  we have a $C^1$-bound for $u_k$, i.e. $\| u_k\|_{C^1(M_w)} \leq C$. As a result, the $C^{0, \alpha}$-convergence of $ u_k (\Xi_{t_k} ) |_{M_w} $ and $ \Xi_{t_k}  |_{M_w} $ imply the $C^{1,\alpha}$-bound of  $u_k$, i.e. $\| u_k\|_{C^{1,\alpha}(M_w)} \leq C'$.  Thus by passing a subsequence, for  $\alpha'<\alpha$ we have  $u_k $ converges to a $C^{1,\alpha'}$-unitary gauge $u_\infty$ in the $C^{1,\alpha'}$-sense, which satisfies that $u_\infty( \Xi_0|_{M_w})=A_0$.  This concludes the theorem.
\end{proof}

\section{A Poincar\'e inequality for $F_{A_t}$}\label{Poincare}
We continue to work under the setup of  Theorem \ref{thm-main}, and choose a sequence of connections $\Xi_{t_k}$. We work on the fiber $M_w$ over a point $w\in N^o$, which is away from the discriminant locus of $f$, the bubbling points, and the ramification points and singularities of the spectral cover. As above we let $A_{t_k}$ denote the restriction of the anti-self-dual connection  $\Xi_{t_k}$ to the smooth fiber $M_w$. The goal of this section is to  derive a Poincar\'e type  inequality for the curvature  $F_{A_{t_k}}$, when $F_{A_{t_k}}$ is sufficiently small in the $C^0$-sense.  The following proposition is the key analytic input to overcome the difficulty of the non-smoothness of the moduli spaces of flat connections on elliptic curves.

For notational simplicity we drop the subscript $k$, and denote our connections by $A_t$. We do this because, aside from being used to define $N^o$, the explicit sequence of times $t_k$ does not have any bearing on the results in this section.
\begin{prop}
\label{poincare} For any compact set $K\subset N^o$, there are constants $\epsilon_K>0$ and $C_K>0$ such that if
\be
\|F_{A_t}\|_{C^0{(M_w, \omega^{SF})}}\leq \epsilon_K\nonumber
\ee
for a certain  $t\in (0,1]$ and $w\in K$,  then
\be
\|F_{A_t}\|_w\leq C_K  \|d_{A_t}^* F_{A_t}\|_{w}.\nonumber
\ee
\end{prop}

  We begin  by recalling part of our setup, as described in  Theorem \ref{thm-main}. Fix  an open subset  $U\subset N^o$ biholomorphic to a disk in $\mathbb{C}$, satisfying  $f^{-1}(U)\cong (U\times \mathbb{C})/{\rm Span}_{\mathbb{Z}}\{1, \tau\}$, where $\tau$ is a holomorphic function on $U$.  Fix  trivializations $P|_{M_U}\cong  M_U \times SU(n)$ and $\mathcal{V}|_{M_U}\cong  M_U \times \mathbb{C}^n$.
  In Section \ref{mainproof} we define the  connections   $A_{0,t}={\rm diag}\{\alpha_{t,1}, \cdots, \alpha_{t,n}\}$ and  $A_{0}={\rm diag}\{\alpha_{0,1}, \cdots, \alpha_{0,n}\}$  associated to the spectral covers, where
   $$ \alpha_{t,j}=  \pi {\rm Im} (\tau)^{-1} (q_{j,t}(w)\bar{\theta}-\bar{q}_{j,t}(w)\theta), \  \  \  \alpha_{0,j}=  \pi {\rm Im} (\tau)^{-1} (q_{j}(w)\bar{\theta}-\bar{q}_{j}(w)\theta), $$ and $\theta|_{M_w}=dz$.
   Here all points vary holomorphically in the base, and satisfy $$\sum\limits_{j=1}^n  q_{j,t}(w)\equiv 0,\qquad  \sum\limits_{j=1}^n  q_{j}(w)\equiv 0.$$  We also have that $q_{j,t}$ converges to $q_{j}$  as $t\rightarrow 0$ as holomorphic functions. Furthermore, for any $w\in U$,  $$q_{i,t}(w)\neq q_{j,t}(w) {\rm mod}(\mathbb{Z}+ \tau(w) \mathbb{Z}), \  \  \ q_{i}(w)\neq q_{j}(w){\rm mod}(\mathbb{Z}+ \tau(w) \mathbb{Z})$$ if $i\neq j$. The connections $d_{A_{0,t}}$ and $d_{A_{0}}$ act on $\eta\in C^\infty(M_w, \mathfrak{sl}(n, \mathbb{C}))$ via
  $$d_{A_{0,t}} \eta=d\eta+[A_{0,t}, \eta],  \  \ \ d_{A_{0}} \eta=d\eta+[A_{0}, \eta]. $$
 Note that if $d_{A_{0,t}} \eta=0$, then $d\eta_{jj}=0$ and $d\eta_{ij}+(\alpha_{t,i}-\alpha_{t,j})\eta_{ij}=0$, which implies that $\eta_{ij}=0$ for $i\neq j$, and $\eta_{jj}$ are constants.  Therefore $\ker d_{A_{0,t}}=\{{\rm diag} \{\eta_1, \cdots, \eta_n\} \in \mathfrak{sl}(n, \mathbb{C})\}  $, and the same argument gives also $\ker d_{A_{0}}=\{{\rm diag} \{\eta_1, \cdots, \eta_n\} \in \mathfrak{sl}(n, \mathbb{C})\}  $.

  Since $A_{0,t}$ is  flat   ($F_{A_{0,t}}=d_{A_{0,t}}^2=0$), we have a de Rham complex $$  C^\infty(M_w, \mathfrak{sl}(n, \mathbb{C})) \stackrel{d_{A_{0,t}}}\longrightarrow C^\infty(T^* M_w \otimes \mathfrak{sl}(n, \mathbb{C})) \stackrel{d_{A_{0,t}}}\longrightarrow C^\infty(\wedge^2 T^* M_w\otimes \mathfrak{sl}(n, \mathbb{C})).  $$
  Furthermore, there is a well behaved  Hodge theory (cf.  \cite{AtB}). If $ \star_w$ denotes the Hodge star operator with respect to the flat metric $\omega^F_w:=\omega^{SF}|_{M_w}$, then $d^*_{A_{0,t}}=-\star_w d_{A_{0,t}} \star_w$
   is the adjoint of $d_{A_{0,t}}$, and  $d_{A_{0,t}}^*d_{A_{0,t}}+d_{A_{0,t}}d_{A_{0,t}}^*$ is the Hodge Laplacian.  If  we denote $\mathcal{H}_{A_{0,t}}^q(M_w,  \mathfrak{sl}(n, \mathbb{C}))$ the space of $\mathfrak{sl}(n, \mathbb{C})$ valued  harmonic $q$-forms, the Hodge theory asserts an   orthogonal decomposition $$C^\infty(\wedge^q T^* M_w \otimes \mathfrak{sl}(n, \mathbb{C})) \cong  \mathcal{H}_{A_{0,t}}^q(M_w,  \mathfrak{sl}(n, \mathbb{C})) \oplus {\rm Im} d_{A_{0,t}} \oplus  {\rm Im} d_{A_{0,t}}^*,$$ for $q=0,1,2$.

 If we replace $\mathfrak{sl}(n, \mathbb{C})$ by the subalgebra $\mathfrak{su}(n)$, then we have the  subcomplex $(C^\infty(\wedge^q T^* M_w \otimes \mathfrak{su}(n)),  d_{A_{0,t}})$,  the harmonic space of $\mathfrak{su}(n)$ valued $q$-forms $\mathcal{H}_{A_{0,t}}^q(M_w,  \mathfrak{su}(n))$, and the respective Hodge decomposition. Note that we have the connection $A_t \in C^\infty(T^* M_w\otimes \mathfrak{su}(n))$ and the curvature $F_{A_t} \in C^\infty(\wedge^2 T^* M_w\otimes \mathfrak{su}(n))$.  The virtual dimension of the  moduli space $ \mathfrak{M}_{M_w}(n)$ of flat $SU(n)$-connections on $M_w$ is zero due to the Euler number of the complex $(C^\infty(\wedge^q T^* M_w \otimes \mathfrak{su}(n)),  d_{A_{0,t}})$ vanishing, and thus the whole $ \mathfrak{M}_{M_w}(n)$ is regarded as degenerated, which causes many difficulties in the global analysis. However,
  the flat connection $A_{0,t}$ belongs to the regular part  of  $ \mathfrak{M}_{M_w}(n)$, and  $\mathcal{H}_{A_{0,t}}^1(M_w,  \mathfrak{su}(n))$ is the tangent space at $A_{0,t}$.  The infinitesimal deformation space under the  action of the unitary gauge group is ${\rm Im} d_{A_{0,t}} \cap C^\infty(T^* M_w\otimes \mathfrak{su}(n))$, and by using the decomposition $\mathfrak{sl}(n, \mathbb{C})=\mathfrak{su}(n) \oplus i \mathfrak{su}(n)$,  the space ${\rm Im} d_{A_{0,t}}^* \cap C^\infty(T^* M_w\otimes \mathfrak{su}(n))$ is identified with the infinitesimal deformation space induced  by Hermitian  gauges. The readers are referred to \cite{Nis11} for details of the above discussion.

We denote by   $\Delta_{A_{0,t}}=-d_{A_{0,t}}^*d_{A_{0,t}}$  the Laplacian operator acting on $C^\infty(M_w, \mathfrak{sl}(n, \mathbb{C})) $, and have $\ker \Delta_{A_{0,t}}=\ker d_{A_{0,t}}$, ${\rm Im} \Delta_{A_{0,t}}={\rm Im} d_{A_{0,t}}^*,$ and $\ker \Delta_{A_{0,t}}\bot {\rm Im} d_{A_{0,t}}^*$ by the Hodge decomposition.   We need a uniform  estimate for the lower bounds of the first eigenvalue of $\Delta_{A_{0,t}}$.

\begin{lem}
\label{eigen}  For any $w\in U$ and $t\in (0,1]$, if $\lambda_{w,t}$ is
the first eigenvalue of $-\Delta_{A_{0,t}}$ on the fiber $M_w$, then there is a constant $C_1>0$ independent of $t$ and $w$  such that $$\lambda_{w,t}\geq C_1.$$
\end{lem}

\begin{proof} If the above bound does not hold, there are sequences $w_k$ and $t_k$ such that  $t_k \rightarrow t_0$ in $[0,1]$,  $w_k \rightarrow w_0$ in $U$, and $$\lambda_{w_k,t_k}\rightarrow 0$$ when $k\rightarrow\infty$.   Let $\psi_k \in C^\infty (M_{w_k}, \mathfrak{sl}(n, \mathbb{C}))$ be a normalized  eigenvector of $\Delta_{A_{0,t_k}}$, i.e. $\Delta_{A_{0,t_k}} \psi_k =-\lambda_{w_k,t_k} \psi_k $ and $\|\psi_k\|_{w_k}=1$.

We regard $M_w$ as the 2-torus $T^2$ equipped with the complex structure $I_w$, and the K\"{a}hler metric $\omega^{F}_{w}$ as a metric on $T^2$ with respect to $I_w$. Since $\tau(w_k)\rightarrow \tau(w_0)$, we have that $I_{w_k}\rightarrow I_{w_0}$ and $\omega^{F}_{{w_k}} \rightarrow \omega^{F}_{{w_0}}$ in the $C^\infty$-sense. Note that $A_{0,t_k} \rightarrow A_{0,t_0}$ in the $C^\infty$-sense, and if $t_0=0$, then $A_{0,t_0}=A_{0}$. Standard elliptic estimates show that $\| \psi_k\|_{C^{\ell}}\leq C_{\ell}$ for  constants  $C_{\ell}>0$ independent of $k$, where the $C^{\ell}$-norms are calculated by using any fixed metric on $T^2$.  By passing to a subsequence, we have that $ \psi_k \rightarrow  \psi_\infty$ smoothly on $T^2$,  $\|\psi_\infty\|_{w_0}=1$,   and $\Delta_{A_{0,t_0}}\psi_\infty=0$. Thus $\psi_\infty \in \ker \Delta_{A_{0,t_0}}$ and can be represented as ${\rm diag} \{\eta_1, \cdots, \eta_n\} \in \mathfrak{sl}(n, \mathbb{C})$.

Since $\psi_k \bot \ker \Delta_{A_{0,t_k}}$, for any $\psi \in \ker \Delta_{A_{0,t_0}}= \ker \Delta_{A_{0,t_k}}$ we have $$0=\langle \psi_k, \psi\rangle_{w_k}\rightarrow \langle \psi_\infty, \psi\rangle_{w_0}.$$  So $\langle \psi_\infty, \psi\rangle_{w_0}=0$ yet $\|\psi_\infty\|_{w_0}=1$. This is a contradiction, and we obtain the conclusion.
\end{proof}

Again restricting our attention to a single fiber $M_w$ for $w\in U$, we can compute the norm of the fiberwise component of the curvature $F_{A_t}$ with respect to the semi-flat metric
\be
\|F_{A_t}\|^2_{C^0(M_w,\omega_t^{SF})}=\frac{1}{t^2}\|F_{A_t}\|^2_{C^0(M_w,\omega^{SF})}.\nonumber
\ee  Because the error terms relating $\omega_t$ and $\omega_t^{SF}$ decay fast enough (see Theorem \ref{ttm-decay}), we have
\be
\|F_{A_t}\|^2_{C^0(M_w,\omega^{SF})}\leq Ct^2\|F_{A_t}\|^2_{C^0(M_w,\omega_t)}\leq Ct^2\|F_{\Xi_t}\|^2_{C^0(M_w,\omega_t)}.\nonumber
\ee
We assume  that there is a constant  $0< \epsilon \ll 1$, which is  determined later, such that   for a certain  $t$ small enough it holds \be
 \label{smallcurve}
 \|F_{A_t}\|_{C^0(M_w, \omega^{SF})}\leq \epsilon,
 \ee
 for $w\in U$.
By  Proposition \ref{estimate 1},  there is a sequence   $t_k \rightarrow 0$  such that $$\|F_{A_{t_k}}\|^2_{C^0(M_w,\omega^{SF})}\leq Ct_k^2\|F_{\Xi_{t_k}}\|^2_{C^0(M_w,\omega_{t_k})}\leq \epsilon_k \rightarrow 0. $$
Here we used that $U$ is away from the bubbling set.  Therefore, for any fixed  $\epsilon >0$, if we take $t$ to be some time $t_k \ll 1$ such that $\epsilon_k <\epsilon$, then (\ref{smallcurve}) holds.

Recall by \eqref{hermitiangauge} that there exists a Hermitian gauge transformation $e^{-s_t}$ so that  $e^{-s_t}(A_t)=A_{0,t}$.  Although given above, we include the definition of this action here to emphasize that we are working exclusively on a fiber:
\be
 \label{complexaction}
e^{-s_t}(A_{t})= A_{t}+e^{-s_t}\bar\partial_{ A_{t}} e^{ s_t}+\left( e^{- s_t}\bar\partial_{\ A_{t}} e^{s_t}\right)^*.
\ee
Given inequality \eqref{smallcurve}, the assumptions of Theorem 6.1 from \cite{DJ} are satisfied,  which yields a new sequence of Hermitian gauge transformations $e^{\hat s_t}$ which are perpendicular to the kernel of $d_{A_{0,t}}$,  bounded in $C^0$, and define the same connection $e^{-\hat s_t}_*A_t=A_{0,t}$.

For the remainder of this section we work on the fiber $M_w$, and so we may drop  it from adorning norms when it is clear from context. Similarly, all norms in this section are computed with respect to $H$ and $\omega^{F}_w$.

\begin{lem}
\label{C0slem}
Given \eqref{smallcurve}, for every $w\in U$ the Hermitian endomorphism $\hat s_t$ satisfies
\be
\label{C0s}
\|\hat s_t\|_{C^0{(M_w, \omega^{SF})}}\leq C_2\epsilon
\ee
for a uniform constant $C_2$.
\end{lem}

\begin{proof}
To begin, we use that  $\hat s_t$ is uniformly bounded in $C^0$. Following Appendix A of   \cite{JW}, the fact that $A_{0,t}$ is flat, along with a standard formula for curvatures related by a complex gauge transformation, yields
\be
\label{thing4}
-\Delta_{w} |\hat s_t|^2\leq-|\partial_{A_{0,t}} \hat s_t|^2+{\rm Tr}\left( e^{  {\hat s_t}} \star_w F_{A_t}e^{-  {\hat s_t}}\hat s_t\right),
\ee where $\Delta_{w}$ is the Laplacian with respect to the flat K\"{a}hler  metric $\omega^{F}_{w}$.
Integrating the above equality over $M_w$, and using Lemma \ref{eigen} along with the fact that $\hat s_t$ is perpendicular to the kernel of $d_{A_{0,t}}$, gives
\be
\|\hat s_t\|_w^2\leq C\|d_{A_{0,t}} \hat s_t\|_w^2\leq C\epsilon \|\hat s_t\|_w.\nonumber
\ee
Therefore $\|\hat s_t\|_w\leq C\epsilon$.
 Now we argue $\|\hat s_t\|_{C^0({M_w})}$ is also bounded by $C\epsilon$.

Note that \eqref{thing4} implies
 \be
-\Delta_w |\hat s_t|^2\leq C\epsilon |\hat s_t|.\nonumber
 \ee
 Now, suppose the desired bound does not hold, so we can find a sequence of constants $C_t\rightarrow \infty$ so $\|\hat s_t\|_{C^0{}}\geq C_t \epsilon$. Set $\phi_t=|\hat s_t|^2/\|\hat s_t\|^2_{C^0{}}$. For $t$ small enough it holds
 \be
 -\Delta_w \phi_t\leq \frac {C \epsilon |\hat s_t|}{\|\hat s_t\|^2_{C^0{}}}\leq\frac{C}{C_t}\leq 1.\nonumber
 \ee
  If $y_t$ denotes the point in $M_w$ realizing $\sup|\hat s_t|^2$, in a fixed neighborhood of radius $a$ of $y_t$ we see $\phi_t$ is a $C^2$ function satisfying $-\Delta_w \phi_t\leq 1$,  $0\leq \phi\leq 1,$ and $\phi_t(y_t)=1.$    Let $u_t$ be a $C^2$ function satisfying both $\Delta_w u_t=-1$ and $u_t(y_t)=1$. By making $a$ smaller if necessary we can guarantee that $u_t$ is strictly positive on $B_a(y_t)$, and this choice will only depend on $\omega^{F}_w$. Thus we have  $-\Delta_w(\phi_t-u_t)\leq   0$ and $\phi_t(y_t)-u_t(y_t)=0$. Applying  the    mean value inequality to $\phi_t-u_t$ gives \be 0\leq  \int_{B_a(y_t)}(\phi_t-u_t).\nonumber \ee By the positivity of $u_t$,  there exists a constant $\delta>0$ independent of $t$  so that \be \delta\leq \int_{B_a(y_t)}u_t\leq \int_{B_a(y_t)}\phi_t.\nonumber \ee  Rearranging terms gives \be \|\hat s_t\|^2_{C^0{}}\leq \frac1\delta \int_{B_a(y_t)}|\hat s_t|^2\leq \frac1\delta\|\hat s_t\|^2_w\leq {C\epsilon^2} ,\nonumber \ee which is our desired bound.
 \end{proof}

 The above lemma has some  strong consequences, which we now detail. First we need a few key formulas on $M_w$. The complex gauge action by a Hermitian endomorphism \eqref{complexaction} gives
 \be
A_t=e^{\hat s_t}_*A_{0,t}= A_{0,t}+e^{\hat s_t}\bar\partial_{ A_{0,t}} e^{-\hat s_t}+\left( e^{\hat s_t}\bar\partial_{\ A_{0,t}} e^{-\hat s_t}\right)^*.\nonumber
\ee
For a given $s$ define ${\rm ad}_s:=[s,\cdot]$, and let $\Upsilon(s)\in{\rm End}({\rm End}(V_t))$ denote the power series
\be
\Upsilon(s)=\frac{e^{{\rm ad}_s}-1}{{\rm ad}_s}=\sum_{m=0}^{\infty}\frac{(-1)^m}{(m+1)!}({\rm ad}_s)^m.\nonumber
\ee
Note that the first term from the power series $\Upsilon(\hat s_t)$ is the identity, allowing us to write $\Upsilon(\hat s_t)=Id+\ti\Upsilon(\hat s_t)$. Now, recall the standard formula for the derivative of the exponential map
\be
e^{\hat s_t}\bar\partial_{ A_{0,t}} e^{-\hat s_t}=-\Upsilon(\hat s_t)\bar\partial_{A_{0,t}}\hat s_t.\nonumber
\ee
Following Appendix A in \cite{JW} we see
\bea
\label{connectionformula}
A_t&=&A_{0,t}-\bar\partial_{A_{0,t}}\hat s_t+\partial_{ A_{0,t}}\hat s_t-\ti\Upsilon(\hat s_t)\bar\partial_{A_{0,t}}\hat s_t+\ti\Upsilon(-\hat s_t)\partial_{ A_{0,t}}\hat s_t\nonumber\\
&=&A_{0,t}-i\star_w  d_{A_{0,t}} \hat s_t+o(\hat s_t,\nabla_{A_{0,t}}\hat s_t),
\eea and \bea
\label{bigcurveexpression}
F_{ A_t} &=&F_{ A_{0,t}}+ \Upsilon(-\hat s_t ) \bar\partial_{A_{0,t}}\partial_{A_{0,t}} \hat s_t- \Upsilon(\hat s_t)\partial_{A_{0,t}}\bar\partial_{A_{0,t}}\hat s_t \nonumber\\
&&      + \bar\partial_{A_{0,t}} \Upsilon(-\hat s_t) \wedge \partial_{A_{0,t}}\hat s_t -  \partial_{A_{0,t}}\Upsilon(\hat s_t ) \wedge \bar\partial_{A_{0,t}}\hat s_t \\
&&  -   \Upsilon(-\hat s_t )\partial_{A_{0,t}}\hat s_t \wedge \Upsilon(\hat s_t )\bar\partial_{A_{0,t}}\hat s_t+  \Upsilon(\hat s_t)\bar\partial_{A_{0,t}}\hat s_t\wedge \Upsilon(-\hat s_t)\partial_{A_{0,t}}\hat s_t  .\nonumber
\eea
This formula, along with the fact that $A_{0,t}$ is flat, leads to the following characterization of the curvature $F_{A_t}$
\be
\label{curvatureformula}
F_{A_t}=-i\, d_{A_{0,t}}\star_w d_{A_{0,t}} \hat s_t+T_1(\hat s_t,\nabla_{A_{0,t}}^2\hat s_t)+T_2(\partial_{A_{0,t}}\hat s_t,\bar\partial_{A_{0,t}}\hat s_t).
\ee
Thus we conclude
\be
\star_w F_{A_t}=-i\Delta_{A_{0,t}}\hat s_t+T_1+T_2\nonumber
\ee
where the tensors $T_1$ and $T_2$ satisfy
\be
\label{Tcontrol}
|T_1|\leq C\epsilon |\nabla^2_{A_{0,t}}\hat s_t|\qquad{\rm and}\qquad |T_2|\leq |\nabla_{A_{0,t}}\hat s_t|^2.
\ee
\begin{lem}
\label{connectionC0lem}
Given   \eqref{smallcurve} and \eqref{C0s}, the following bound holds
\be
\label{connectionC0}
\|A_t-A_{0,t}\|_{C^{0,\alpha}{(M_w, \omega^{SF})}}\leq C_3\epsilon, \  \  \  \|\nabla_{A_{0,t}}\hat s_t\|_{C^{0,\alpha}{(M_w, \omega^{SF})}}\leq C_3\epsilon
\ee for any  $0<\alpha <1$, by choosing $\epsilon$ small enough.  Here the constant $C_3$ depends on $U\subset N^o$.
\end{lem}
\begin{proof}
We begin the proof with the standard elliptic a priori estimate (cf. \cite{GT,Bess})
\bea
\|\hat s_t\|_{L^p_2{}}&\leq& C\left(\|\Delta_{A_{0,t}}\hat s_t\|_{L^p{}}+\|\hat s_t\|_{L^p{}}\right)\nonumber\\
&\leq&C\left(\|F_{A_t}\|_{L^p{}}+\|T_1\|_{L^p{}}+\|T_2\|_{L^p{}}+\|\hat s_t\|_{L^p{}}\right)\nonumber\\
&\leq&C\left(\epsilon+\|T_1\|_{L^p{}}+\|T_2\|_{L^p{}}\right)\nonumber
\eea
where we have used \eqref{smallcurve} and \eqref{C0s} in the last inequality. We also use the assumption that $A_{0,t} \rightarrow A_0$ smoothly, and therefore all derivatives of  $A_{0,t}$ are bounded independent of $t$. Thus all constants in the above inequality are independent of $t$.

 The necessary  bound for $T_1$ follows immediately $\|T_1\|_{L^p}\leq C\epsilon\|\hat s_t\|_{L^p_2{}}.$ For $T_2$ we use the interpolation inequality for tensors from \cite{Ham} (see also Section 7.6 in  \cite{Au})
\be
\left(\int_{M_w}|\nabla_{A_{0,t}}\hat s_t|^{2p}\right)^{\frac1p}\leq (\sqrt 2+2p-2) \|\hat s_t\|_{C^0{}}\left(\int_{M_w}|\nabla_{A_{0,t}}^2\hat s_t|^p\right)^{\frac1p}.\nonumber
\ee
This implies $\|T_2\|_{L^p{}}\leq C\epsilon\|\hat s_t\|_{L^p_2{}}$. Thus
$$
\|\hat s_t\|_{L^p_2{}}\leq C\left(\epsilon+\epsilon\|\hat s_t\|_{L^p_2{}}\right)
$$ and
for $\epsilon$ small enough
\be
\label{lq2}
\|\hat s_t\|_{L^p_2{}}\leq C\epsilon.
\ee  By Morrey's inequality, for large enough $p$ we can conclude
\be
\|\nabla_{A_{0,t}}\hat s_t\|_{C^{0,\alpha}{}}\leq C\epsilon,\nonumber
\ee
and the proof follows from \eqref{connectionformula}.
\end{proof}

Comparing this lemma to  Theorem 3.11 of \cite{Nis11},  the bound of
(\ref{connectionC0}) is stronger, i.e.  we have $\epsilon$ instead of  $\epsilon^{\frac{1}{2}}$, due to our assumption that $A_{0,t}$ and $A_{0}$ are regular.

We now turn to the proof of  the main proposition of this section.

\begin{proof}[Proof of Proposition \ref{poincare}]
Once again we begin with the standard elliptic a priori estimate
\be
\|\hat s_t\|_{L^2_2{}}\leq C \left(\|\Delta_{A_{0,t}}\hat s_t\|_{w}+\|\hat s_t\|_{w}\right).\nonumber
\ee
Since $\hat s_t$ is perpendicular to the the kernel of $d_{A_{0,t}}$, we have  
a stronger inequality
\be
\|\hat s_t\|_{L^2_2{}}\leq C \|\Delta_{A_{0,t}}\hat s_t\|_{w}\nonumber
\ee (cf. \cite{GT,Bess}).  Again we use the fact that all derivatives of  $A_{0,t}$ and  $A_0$ are bounded independent of $t$.

Next, we recall  \eqref{Tcontrol}. Applying the the interpolation inequality for tensors from the previous lemma for $p=2$, we have
\be
\|T_1+T_2\|_{w}\leq C \epsilon\|\hat s_t\|_{L^2_2{}}\leq C \epsilon\|\Delta_{A_{0,t}}\hat s_t\|_{w}.\nonumber
\ee
Let $F^o_t$ denote the projection of $\star_w F_{A_t}$ onto the kernel of $\Delta_{A_{0,t}}$, and set $F^\perp_t=\star_w F_{A_t}-F^o_t$. Because $\Delta_{A_{0,t}}\hat s_t$ is perpindicular to the kernel of $\Delta_{A_{0,t}}$, we can conclude
\bea
\|F^\perp_t\|_{w}&\geq& \|\Delta_{A_{0,t}}\hat s_t\|_{w}-\|F^\perp_t-\Delta_{A_{0,t}}\hat s_t\|_{w}\nonumber\\
&=&\|\Delta_{A_{0,t}}\hat s_t\|_{w}-\|(T_1+T_2)^\perp\|_{w}\nonumber\\
&\geq& (1-C\epsilon)\|\Delta_{A_{0,t}}\hat s_t\|_{w}\nonumber\\ & \geq & \frac12\|\Delta_{A_{0,t}}\hat s_t\|_{w}.\nonumber
\eea  We take $\epsilon$ small enough such that $C \epsilon < \frac{1}{2}$.

Now, since $(\Delta_{A_{0,t}}\hat s_t)^o=0$, we also have
\be
\|F^o_t\|_{w}\leq \|(T_1+T_2)^o\|_{w}\leq C\epsilon \|\Delta_{A_{0,t}}\hat s_t\|_{w}\leq 2 C\epsilon \|F^\perp_t\|_{w},\nonumber
\ee
which implies
\bea
\|F_{A_t}\|_{w}&\leq& \|F^o_t\|_{w}+\|F^\perp_t\|_{w}\nonumber\\
&\leq& (1+2 C\epsilon)\|F^\perp_t\|_{w}\nonumber\\ & \leq &  2 \|F^\perp_t\|_w.\nonumber
\eea
Thus, applying the Poincar\'e inequality  to $F^\perp_t$ and Lemma \ref{eigen}, we can conclude
\be
\|F_{A_t}\|_{w}\leq 2\|F_t^\perp\|_{w}\leq C\|d_{A_{0,t}}^* F_{A_t}\|_{w}.\nonumber
\ee
The proposition now follows from Lemma \ref{connectionC0lem}, which allows us to bound the difference between the connections $A_t$ and $A_{0,t}$
\bea
\|F_{A_t}\|_{w}&\leq&C  \|d_{A_{0,t}}^*F_{A_t}\|_{w}\nonumber\\
&\leq&C  \|d_{A_t}^*F_{A_t}\|_w+C \|A_t-A_{0,t}\|_{C^0{}}\|F_{A_t}\|_{w}\nonumber\\
&\leq&C  \|d_{A_t}^* F_{A_t}\|_{w}+C \epsilon \|F_{A_t}\|_{w}.\nonumber
\eea
We choose further that $C \epsilon < \frac{1}{2}$, and obtain $$ \|F_{A_t}\|_{w} \leq  2 C  \|d_{A_t}^* F_{A_t}\|_{w}.$$

For any $K\subset N^o$, we cover $K$ by finite open disks  $U_\beta$, i.e. $K\subset \bigcup U_\beta \subset N^o$, and apply the above arguments to any $U_\beta$.
  By letting $\epsilon_K=\min \{\epsilon\}$ over the covering, and $C_K$ the maximum   constant over the covering, the proposition is proved.
\end{proof}

A corollary is  the following Sobolev inequality.
 \begin{cor}\label{sobolev}
 For any $p\geq 2$, there exists a cosntant $C_p$ so that $$\|F_{A_t}\|_{L^p(M_w)} \leq C_p\|d_{A_t}^\star F_{A_t}\|_w.$$
 \end{cor}

 \begin{proof} In dimension two we have the  Sobolev inequality $$\|\xi\|_{L^p} \leq C_p(\|\nabla_{A_{0,t}} \xi\|_{w} + \|\xi\|_{w})\leq C_p(\|\nabla_{A_{t}} \xi\|_{w} + \|(A_{t}-A_{0,t})\xi\|_{w}+\|\xi\|_{w}),$$ for any smooth  section  $\xi$ of ${\rm End}(\mathcal{V})$ and some constant $C_p$   independent of $w\in U$ and $t$.
  Applying this to $ \xi=\star_w  F_{A_t}$, we obtain $$\|F_{A_t}\|_{L^p} \leq C_p(\|d_{A_{t}}^\star F_{A_t}\|_{w} + (1+\epsilon) \|F_{A_t}\|_{w})\leq 2C_pC_K\|d_{A_t}^\star F_{A_t}\|_w,$$ by Proposition \ref{poincare}.
   \end{proof}

\section{$C^0$ bounds on curvature}
\label{C0boundsection}

The main goal of this section is to prove Proposition \ref{prop2}, which establishes $C^0$ control for the curvature of a family of connections. It is a conditional result relying on assumption \eqref{hyp:curv}. To avoid confusion, we note that this   result is applied twice. In Section \ref{lowerbounds}, in the proof of  Proposition \ref{type3bubbling}, it is applied to a family of connections in scaled coordinates, for which \eqref{hyp:curv} can be verified directly. Once Proposition \ref{type3bubbling} is established, assumption \eqref{hyp:curv} holds for our main sequence of connections $\Xi_{t_k}$ from the statement of Theorem \ref{thm-main}, and so Proposition \ref{prop2} can be used to establish  Proposition \ref{thm01}.

As above, let $U\subset\subset N^o$ be an open subset, compactly contained in $N_0$, and biholomorphic to a disk in $\mathbb{C}$. We have  $f^{-1}(U)\cong (U\times \mathbb{C})/{\rm Span}_{\mathbb{Z}}\{1, \tau\}$, where the period $\tau$ is holomorphic on $U$. Let $w$ denote the complex coordinate on $U$, and $z$ the coordinate on $\mathbb{C}$.    Furthermore, we fix a  trivialization $P|_{M_U}\cong  M_U \times SU(n)$ and $\mathcal{V}|_{M_U}\cong  M_U \times \mathbb{C}^n$. Under such trivialization, the Hermitian metric $H$ is the absolute value $|\cdot|$, the connection $\Xi_t$ is a matrix valued 1-form, and the curvature $F_{\Xi_t}$ is a matrix valued 2-form, i.e. $\Xi_t\in C^\infty(T^*M_U, \mathfrak{su}(n))$  and $F_{\Xi_t} \in C^\infty(\wedge^2 T^*M_U, \mathfrak{su}(n))$.

Define   real coordinates  $(x_1, x_2)$ on $U$ satisfying $w=x_1+i x_2$, and recall that we have the decomposition $T^*M_U\cong {\rm Span}_{\mathbb{R}}\{dy_1,dy_2\}\oplus {\rm Span}_{\mathbb{R}}\{dx_1,dx_2\}$, where $z=y_1+\tau y_2$, and
 $z$ is the coordinate on $\mathbb{C}$.
 In these coordinates we write
 \be\label{decomp}
 \Xi_t=A_t+B_{t,1}dx_1+B_{t,2} dx_2,
 \ee
where  $A_t$ is a connection on the restriction to the fiber $\mathcal{V}|_{M_w}$, and $B_{t,i}$ is a section in $\Gamma(U, \Omega^0(M_w, \mathfrak{su}(n)))$ for $i=1,2$. Given this decomposition, the curvature can be written as
 \be
 \label{curve}
 F_{\Xi_t}=F_{A_t}-\kappa_{t,1}\wedge dx_1-\kappa_{t,2}\wedge dx_2-F_{B,t} dx_1\wedge dx_2.
 \ee
Here $F_{A_t}$ is the curvature of $A_t$, the mixed terms are given by
\be
\kappa_{t,i}=\frac{\partial}{\partial x_i} A_t-d_{A_t}B_{t,i}\qquad{\rm for}\qquad i=1,2,\nonumber
\ee
and the curvature in the base direction can be expressed as
\be
F_{B,t}=\frac{\partial}{\partial x_2} B_{t,1}-\frac{\partial}{\partial x_1} B_{t,2}-[B_{t,1},B_{t,2}].\nonumber
\ee
 Because of the uniform equivalence $$ C_U^{-1}\omega^{SF}_t \leq \omega_t \leq C_U \omega^{SF}_t, \  \  \ {\rm and}  \  \   \omega^{SF}_t|_{M_w}=t\omega^{SF}|_{M_w}, $$
 the norms of the different curvature components satisfy $$  |F_{A_t}|_{\omega^{SF}}= t|F_{A_t}|_{\omega_t^{SF}}, \  \  \  |\kappa_{t,i}|_{\omega^{SF}} = \sqrt{t} |\kappa_{t,i}|_{\omega_t^{SF}},  \  \  \  |F_{B,t}|_{\omega^{SF}} =|F_{B,t}|_{\omega_t^{SF}}. $$

We now state the main assumption of this section. Assume that there is a constant $C_1>0$, so that for a   $t\in (0,1]$ it holds
\be \label{hyp:curv}
\sup_{M_U} |F_{\Xi_t}|_{\omega_t} \leq C_1t^{-\frac{1}{2}}.
\ee
This implies
\be \label{hyp:curv+=} \sup_{M_U} |F_{A_t}|_{\omega^{SF}} \leq C_1t^{\frac{1}{2}},  \  \  \sup_{M_U} |\kappa_{t,i}|_{\omega^{SF}} \leq C_1,  \  \  \sup_{M_U} |F_{B,t}|_{\omega^{SF}} \leq C_1t^{-\frac{1}{2}}.\nonumber
\ee
We assume that   $t\ll 1$ small enough such that  $C_1t^{\frac{1}{2}}<\epsilon_K$, where $\epsilon_K$ is the small constant controlling the curvature in  Proposition \ref{poincare}, and $U\subset K$.
Thus by Proposition \ref{poincare}, we see that  the curvature $F_{A_t}$   satisfies the Poincar\'{e}  type inequality
\be \label{hyp:poincare}
\|F_{A_t}\|_{w}\leq C_2 \|d_{A_t}^* F_{A_t}\|_{w}.
\ee
This inequality, along with  assumption \eqref{hyp:curv}, are instrumental in the following:

\begin{prop}\label{prop2} Let $\nabla_{x_i}=\partial_{x_i}+B_{t,i}$ for $ i=1,2$  denote covariant differentiation in the base direction. If (\ref{hyp:curv}) and (\ref{hyp:poincare}) hold for $t\ll 1$, for $U'\subset\subset U$ we have
the following inequalities:  \begin{itemize} \item[i)] $$  \|F_{A_t}\|_{C^{0}(M_{U'},\omega^{SF})} \leq C_3   t, \   \  \  \|F_{B,t}\|_{C^{0}(M_{U'},\omega^{SF})} \leq C_3 , $$ \item[ii)] $$  \|\nabla_{x_i} F_{A_t}\|_{L^2(M_{U'}, \omega^{SF})} \leq C_3  t^{\frac{1}{2}} ,$$  \item[iii)] $$  \|F_{\Xi_t}\|_{C^0(M_{U'}, \omega^{SF})} \leq C_3, $$
      \end{itemize} where    the constant $C_3$ may depend on the distance from $U'$ to $\partial U$, but is   independent of $t$.

        \end{prop}

As above let $\star_w$ denote the Hodge star operator on the fiber $M_w$ with respect to the flat metric $\omega^F_w:=\omega^{SF}|_{M_w}=i{\rm Im}(\tau)^{-1}\,dz\wedge d\bar z$. Then $\star_w^2=-1$, $\star_wdz=-idz$ and $\star_w d\bar z=id\bar z$. We write the anti-self-dual equation under the decomposition (\ref{curve}).

\begin{lem}
The curvature of $\Xi_t$  satisfies
\be
\label{ASD1}
\star_w \kappa_{t,1}=\kappa_{t,2}
\ee
and
\be
\label{ASD2}
t^{-1}(1+G_0+G_1)\star_w F_{A_t}-(W+G_2) F_{B,t}=\sum_{j=1}^2\kappa_{t,j}\# G_3,
\ee
where $G_1$, $G_2$, $G_3$  are smooth functions depending on $t$ such that
\be
t^{-\frac{\nu}{2}}(\| G_1\|_{C^{0}(\omega^{SF})}+ \|\frac{\partial}{\partial z} G_1\|_{C^{\ell}(\omega^{SF})}+\|\frac{\partial}{\partial \bar{z}} G_1\|_{C^{\ell}(\omega^{SF})}+\sum_{j=2,3}\| G_j\|_{C^{\ell}(\omega^{SF})})\rightarrow 0, \nonumber
\ee
for  any $\nu\in\mathbb{N}$, and $G_0$ is a function on $U$ such that $ \|G_0\|_{C^{\ell}(U)}\rightarrow 0$,     when $t\rightarrow 0$.
\end{lem}

 \begin{proof}
 We first demonstrate that \eqref{ASD1} follows from $F_{\Xi_t}^{0,2}=F_{\Xi_t}^{2,0}=0$.  Note that
 \be
 2(\kappa_{t,1}\wedge dx_1+\kappa_{t,2} \wedge dx_2)=(\kappa_{t,1} -i \kappa_{t,2}) \wedge dw+(\kappa_{t,1}+i\kappa_{t,2})\wedge d\bar w.\nonumber
 \ee
  This implies, using $\star_wdz=-idz$ and $\star_w d\bar z=id\bar z$, that
 \be
 \star_w(\kappa_{t,1} -i \kappa_{t,2})=i (\kappa_{t,1} -i \kappa_{t,2})=i\kappa_{t,1}+\kappa_{t,2}\nonumber
 \ee
 and
 \be
  \star_w(\kappa_{t,1}+i\kappa_{t,2})=-i (\kappa_{t,1}+i\kappa_{t,2})=-i\kappa_{t,1}+\kappa_{t,2}.\nonumber
 \ee
 Adding these two equations together proves \eqref{ASD1}.

 We now concentrate on \eqref{ASD2}. Using $F_{\Xi_t}\wedge \omega_t=0$, along with the decompositions \eqref{SFmetric} and \eqref{curve}, we see
 \bea
 0=F_{\Xi_t}\wedge\omega_t&=&F_{\Xi_t}\wedge\omega_t^{SF}+F_{\Xi_t}\wedge i\partial\bar\partial\varphi_t\nonumber\\
 &=& \frac i{2}(W^{-1}+2\varphi_{t,w \bar{w}})F_{A_t}\wedge dw \wedge d\bar{w}\nonumber\\
 &&-\frac i2(tW+ 2\varphi_{t,z \bar{z}})F_{B,t}\,dx_1 \wedge dx_2 \wedge \theta \wedge \bar{\theta}\nonumber\\
 &&+ (\kappa_{t,1}\wedge dx_1+\kappa_{t,2}\wedge dx_2)\wedge {\rm Im}\left(2\varphi_{t,w\bar{z}} dw \wedge d\bar{z}\right).\nonumber
 \eea
 Next, note that  $\theta=dy_1+\tau dy_2=dz+bdw$,
 \be
 dx_1\wedge dx_2=\frac i2 dw\wedge d\bar w\qquad{\rm and}\qquad F_{A_t}=\frac i2(\star_w F_{A_t})W \theta\wedge \bar \theta.\nonumber
 \ee
 Thus, dividing out by the volume form $dz\wedge dw \wedge d\bar{z}\wedge d\bar{w}= \theta \wedge dw \wedge \bar{\theta}\wedge d\bar{w}$, the above equation can be rewritten as
 \bea
 0&=&(1+2\varphi_{t,w \bar{w}}W)\star_w F_{A_t}-(tW+ 2\varphi_{t,z \bar{z}})F_{B,t}\nonumber\\
 &&+\sum_{i=1}^2 \kappa_{t,i}\, \#\left(\varphi_{t,z\bar w}+\varphi_{t,w\bar z}\right).\nonumber
 \eea
 We set $G_0=2\chi_{t,w\bar w}W$,   $G_1=2(\varphi_{t,w\bar w}-\chi_{t,w\bar w})W$, $G_2=2t^{-1}\varphi_{t,z\bar z}$, and $G_3=t^{-1}(\varphi_{t,z\bar w}+\varphi_{t,w\bar z}).$  The proof now follows from Lemma \ref{lem-decay}.
 \end{proof}

Next we turn to a Bochner type formula for $F_{A_t}$.

 \begin{lem}\label{le2} If we denote $\Delta=\partial_{x_1}^2+\partial_{x_2}^2$, then \begin{eqnarray*}\Delta \| F_{A_t}\|_w^2 & \geq & \frac{1}{4} \sum_{i=1,2}  \|\nabla_{x_i}  F_{A_t}\|_w^2+ \frac{\delta}{t} \|  d_{A_t}^* F_{A_t}  \|_w^2\\ & & - C_4't(\sum_{j=1,2}\|[\kappa_{t,j}, \kappa_{t,j}]\|_w^2 +t^\nu)\\ & \geq & \frac{1}{4} \sum_{i=1,2}  \|\nabla_{x_i}  F_{A_t}\|_w^2+ \frac{\delta}{t} \|  d_{A_t}^* F_{A_t}  \|_w^2 - C_4t, \end{eqnarray*} for constants $\delta >0$,  $C_4>0$ and $C_4'>0$.   \end{lem}

\begin{proof}  Note  we can write the mixed and base curvature terms as $$ \nabla_{x_1}d_{A_t}-d_{A_t}\nabla_{x_1}=\kappa_{t,1}, \  \  \ \nabla_{x_2}d_{A_t}-d_{A_t}\nabla_{x_2}=\kappa_{t,2},  \  \  \ [ \nabla_{x_1}, \nabla_{x_2}] =F_{B,t}. $$  By the Bianchi identity $d_{\Xi_t}F_{\Xi_t}=0$, and so $$d_{A_t}F_{t,B}=\nabla_{x_1}\kappa_{t,2}-\nabla_{x_2}\kappa_{t,1},     \  \  \ \nabla_{x_1} F_{A_t}=d_{A_t} \kappa_{t,1},  \ \ {\rm and}  \  \   \nabla_{x_2} F_{A_t}=d_{A_t} \kappa_{t,2}.$$


Recall that $\star_w dz=-i dz$, $\star_w d\bar{z}=id\bar{z}$ and $\star_w \frac{i}{2}Wdz\wedge d\bar{z}=1$.  Also, $\star_w$ is independent of $w$ when acting  on 1-forms,   and $\partial_{x_i}\star_w= -W^{-1} (\partial_{x_i}W)\star_w  $ in the other cases. By the above formulas, we derive
\bea
(\nabla_{x_1}^2+\nabla_{x_2}^2)F_{A_t}&=&\nabla_{x_1} d_{A_t} \kappa_{t,1}+ \nabla_{x_2} d_{A_t} \kappa_{t,2}\nonumber\\
&=&d_{A_t}(\nabla_{x_1}  \kappa_{t,1} + \nabla_{x_2}  \kappa_{t,2})+\sum_{j=1,2}[\kappa_{t,j}, \kappa_{t,j}].\nonumber
\eea
By \eqref{ASD1}, we also have
$$ \nabla_{x_1} \kappa_{t,1} =-\star_{w} \nabla_{x_1} \kappa_{t,2},  \  \  {\rm and} \  \  \nabla_{x_2} \kappa_{t,2} =\star_{w} \nabla_{x_2} \kappa_{t,1}. $$
Hence, using \eqref{ASD2}, we obtain a Weitzenb\"{o}ck type formula for $F_{A_t}$:
\bea
\label{thing20}
 (\nabla_{x_1}^2+\nabla_{x_2}^2)F_{A_t} & =   & d_{A_t}\star_{w}(  \nabla_{x_2} \kappa_{t,1}-  \nabla_{x_1} \kappa_{t,2})+\sum_{j=1,2}[\kappa_{t,j}, \kappa_{t,j}]\\ &=& - d_{A_t}\star_{w} d_{A_t}F_{B,t} +\sum_{j=1,2}[\kappa_{t,j}, \kappa_{t,j}]  \nonumber\\ &=&- t^{-1} d_{A_t}\star_{w} d_{A_t}(G_4  \star_w F_{A_t})  +\sum_{j=1,2}[\kappa_{t,j}, \kappa_{t,j}]  \nonumber\\ & & +d_{A_t}\star_{w} d_{A_t}( \sum_{i=1,2} \kappa_{t,i}\# G_5 ),\nonumber
\eea
 where $$G_4= (W+ G_2)^{-1}(1+G_0+G_1),  \   \  {\rm and} \  \ G_5= (W+ G_2)^{-1}G_3.$$  Note that for any differential form $\alpha$,  $d_{A_t}\alpha =d^{f} \alpha$, where $d^{f}$ denotes the differential along the fiber direction, i.e. $d^f=\partial_{y_1}(\cdot)dy_1+\partial_{y_2}(\cdot)dy_2$,  and $\nabla_{x_i}\alpha =\partial_{x_i} \alpha$.

Since $\| F_{A_t}\|_w^2=\int_{M_w}{\rm tr}F_{A_t}\wedge \star_w F_{A_t}$, a  direct calculation shows
$$\partial_{x_i}^2\| F_{A_t}\|_w^2=\|\nabla_{x_i}  F_{A_t}\|_w^2+2{\rm Re}\langle \nabla_{x_i}^2  F_{A_t} ,   F_{A_t}  \rangle_w+T_i,    $$
where the term  $T_i$ arises from derivative on the fiber metric, and satisfies
 \begin{eqnarray*} |T_i|& \leq &  C(|\partial_{x_i}\star_w|\|\nabla_{x_i} F_{A_t}\|_w\| F_{A_t}\|_w+ |\partial_{x_i}^2 \star_w|\| F_{A_t}\|_w^2)\\ & \leq & \frac{1}{2} \|\nabla_{x_i} F_{A_t}\|_w^2+C\| F_{A_t}\|_w^2 .   \end{eqnarray*}
Using the notation $\|\nabla_x  F_{A_t}\|_w^2=\sum\limits_{i=1,2}  \|\nabla_{x_i}  F_{A_t}\|_w^2$, the above calculations give
$$ \Delta \| F_{A_t}\|_w^2=\|\nabla_x  F_{A_t}\|_w^2+2{\rm Re}\langle (\nabla_{x_1}^2+\nabla_{x_2}^2)  F_{A_t} ,   F_{A_t}  \rangle_w+T_1+T_2.      $$
To this equality, we can now apply  \eqref{thing20}. Using $d_{A_t}^*=-\star_w d_{A_t} \star_w$,  we see
\begin{eqnarray*}
{\rm Re}\langle  (\nabla_{x_1}^2+\nabla_{x_2}^2)F_{A_t} ,   F_{A_t}  \rangle_w   &=& t^{-1}  {\rm Re}\langle G_4 d_{A_t}^*  F_{A_t} ,  d_{A_t}^*  F_{A_t}  \rangle_w  \\
&&+ {\rm Re}\langle \sum_{j=1,2}[\kappa_{t,j}, \kappa_{t,j}] ,   F_{A_t}  \rangle_w \\ & & -  t^{-1}  {\rm Re}\langle \star_{w} ( d^f G_4) \star_{w}  F_{A_t} ,  d_{A_t}^*  F_{A_t}  \rangle_w\\
 & & + {\rm Re}\langle \star_{w} d_{A_t}( \sum_{i=1,2} \kappa_{t,i}\# G_5 ),   d_{A_t}^*  F_{A_t}  \rangle_w.
\end{eqnarray*}

Next, note that for a constant $\delta >0$, we have $${\rm Re}\langle G_4 d_{A_t}^*  F_{A_t} ,  d_{A_t}^*  F_{A_t} \rangle_w  \geq 8  \delta \|  d_{A_t}^*  F_{A_t}  \|_w^2.  $$ Using \eqref{hyp:curv} to bound the mixed terms, and  the Poincar\'e inequality \eqref{hyp:poincare}, we have
 \begin{eqnarray*}|\langle \sum_{j=1,2}[\kappa_{t,j}, \kappa_{t,j}]   ,   F_{A_t} \rangle_w|  & \leq &  C \sum_{j=1,2}\|[\kappa_{t,j}, \kappa_{t,j}]\|_w\| F_{A_t} \|_w\\ & \leq &  C \sum_{j=1,2}\|[\kappa_{t,j}, \kappa_{t,j}]\|_w \| d_{A_t}^*F_{A_t} \|_w\\ & \leq & C t\sum_{j=1,2}\|[\kappa_{t,j}, \kappa_{t,j}]\|_w^2 +\frac{\delta}{t}\| d_{A_t}^*F_{A_t} \|_w^2.   \end{eqnarray*}
Because $d^f W=0$,  $ d^f G_0=0$,  and $d^f G_4=o(t^{\nu})$ for $\nu\gg 1$,  it follows that
\begin{eqnarray*}  | t^{-1}  {\rm Re}\langle \star_{w} ( d^f  G_4) \star_{w}  F_{A_t} ,  d_{A_t}^*  F_{A_t} \rangle_w|& \leq & C \| F_{A_t}\|_w\|d_{A_t}^* F_{A_t}\|_w \\ &\leq &  C t \| F_{A_t}\|_w^2+\frac{\delta}{t}\|d_{A_t}^* F_{A_t}\|_w^2.  \end{eqnarray*}
Finally, $ |d^fG_5|_{\omega^{SF}}=  o(t^{\nu})$ for any  $\nu\gg 1$, and so
 \begin{eqnarray*}|\langle \star_{w} d_{A_t}( \sum_{i=1,2} \kappa_{t,i}\# G_5 ),   d_{A_t}^*  F_{A_t} \rangle_w| & \leq & C\|d_{A_t}^* F_{A_t}\|_w (t^{\nu}+\sum_{i=1,2}\|  d_{A_t} \kappa_{t,i}\|_w)\\  & = &   C\|d_{A_t}^* F_{A_t}\|_w (t^{\nu}+\sum_{i=1,2}\| \nabla_{x_i} F_{A_t}\|_w) \\ & \leq  & Ct(t^{\nu}+\| \nabla_x  F_{A_t}\|_w^2)+  \frac{\delta}{t}\|d_{A_t}^* F_{A_t}\|_w^2. \end{eqnarray*}
Putting everything together \begin{eqnarray*} {\rm Re}\langle  (\nabla_{x_1}^2+\nabla_{x_2}^2)F_{A_t} ,   F_{A_t} \rangle_w  & \geq & \frac{4\delta}{t} \|  d_{A_t}^*  F_{A_t}  \|_w^2 - Ct(t^{\nu}+ \| F_{A_t} \|_w ^2+\| \nabla_x  F_{A_t}\|_w^2\\ & & +\sum_{j=1,2}\|[\kappa_{t,j}, \kappa_{t,j}]\|_w^2 ),\end{eqnarray*}
which implies
  \begin{eqnarray*} \Delta \| F_{A_t}\|_w^2& \geq &  \|\nabla_x  F_{A_t}\|_w^2+ \frac{4\delta}{t} \|  d_{A_t}^*  F_{A_t}  \|_w^2 -\frac{1}{2} \|\nabla_x  F_{A_t}\|_w^2-2C\| F_{A_t}\|_w^2 \\ & &- Ct(t^{\nu}+ \| F_{A_t} \|_w ^2+\| \nabla_x  F_{A_t}\|_w^2+\sum_{j=1,2}\|[\kappa_{t,j}, \kappa_{t,j}]\|_w^2 ).  \end{eqnarray*} The Poincar\'e inequality  (\ref{hyp:poincare}), along with Young's inequality, gives  $$ \Delta \| F_{A_t}\|_w^2 \geq \frac{1}{4}  \|\nabla_x  F_{A_t}\|_w^2+ \frac{\delta}{t} \|  d_{A_t}^*  F_{A_t}  \|_w^2- C t(\sum_{j=1,2}\|[\kappa_{t,j}, \kappa_{t,j}]\|_w^2 +t^\nu). $$  \end{proof}

We need the following elementary lemma, and we  include the proof for the reader's convenience (cf. Sublemma 6.48 in \cite{Fuk2}). As in the pervious lemma, let  $\Delta=\partial_{x_1}^2+\partial_{x_2}^2$ denote the coordinate Laplacian in the base.

\begin{lem}\label{le3}Let $\zeta$ be a non-negative real valued function satisfying $$ \Delta \zeta \geq \frac{\delta}{t}\zeta-t$$ on a disk $U\subset \mathbb{C}$. Then for an open subset  $U'\subset\subset U$,  there exists a constant $C_{5}$, which depends on the distance from $U'$ to $\partial U$, such that $$\sup_{U'}|\zeta|\leq C_{5}t^2.$$
    \end{lem}
 \begin{proof}
For any point $w_0\in U'$, let $d= \sup \{ |w-w_0\|w\in  U\}$, and let $a$ be a positive number such that $4a^2d^2+4a <\delta$.
 Consider the function  $\xi=\zeta\exp(-\frac{a|w-w_0|^2}{\sqrt{t}})$.   If $\xi$ achieves its maximum $w_1$ on $\partial U$,  then   $$\zeta(w_0)=\xi(w_0)\leq \xi(w_1)= \zeta(w_1) \exp(-\frac{a|w_1-w_0|^2}{\sqrt{t}})\leq C  \exp(-\frac{ar^2}{\sqrt{t}}), $$ where $r$ is the distance from $w_0$ to $\partial U$.  For $t$ small enough the right hand side is smaller than $Ct^2$.

 Otherwise, at an interior maximum $w_1$, we see $$0=\partial_w \xi(w_1)=(-\frac{a(\bar{w}_1-\bar{w}_0)}{\sqrt{t}}\zeta (w_1) +\partial_w \zeta (w_1))\exp(-\frac{a|w_1-w_0|^2}{\sqrt{t}}), $$ and $\partial_{\bar{w}} \xi(w_1)=0 $.  Furthermore, since $\Delta=2\partial_w \partial_{\bar{w}}$,  at this maximum point
\bea
0 & \geq &  \Delta \xi(w_1)  \nonumber \\  &= &  2   \left(\partial_w \partial_{\bar{w}} \zeta (w_1)- \frac{a^2|w_1-w_0|^2+a \sqrt{t}}{t}\zeta (w_1) \right)\exp(-\frac{a|w_1-w_0|^2}{\sqrt{t}}) \nonumber \\
&\geq&     \left(\frac{\delta}{t}\zeta (w_1)-2 \frac{a^2d^2+a }{t}\zeta (w_1)-t \right)\exp(-\frac{a|w_1-w_0|^2}{\sqrt{t}}) \nonumber \\ &\geq&   \left(\frac{\delta}{2t}\zeta (w_1)-t \right)\exp(-\frac{a|w_1-w_0|^2}{\sqrt{t}}).    \nonumber
\eea
 Thus $$\xi(w_1)\leq \zeta (w_1) \leq 2 \delta^{-1} t^2,$$ and so $$\zeta(w_0)=\xi(w_0)\leq \xi(w_1) \leq 2 \delta^{-1} t^2.  $$ \end{proof}

\begin{lem}\label{le-l2}  For any $w\in U' \subset\subset U$,  $$\| F_{A_t}\|_w \leq C_{6} t,  \  \  \   and  \  \    \|\nabla_{x_i} F_{A_t}\|_{L^2(U', \omega^{SF})} \leq C_{6}  t^{\frac{1}{2}} , $$ for a constant $C_{6}>0$ independent of $t$ and $w$.
    \end{lem}

\begin{proof}  Lemma \ref{le2} and Lemma \ref{eigen} imply  $$ \Delta \|F_{A_t}\|_w^{2} \geq  \frac{1}{4}  \|\nabla_x  F_{A_t}\|_w^2+ \frac{\delta}{t} \|  d_{A_t}^*  F_{A_t}  \|_w^2- Ct\geq \frac{\delta' }{t} \|F_{A_t}\|_w^{2} -Ct.$$  Thus by Lemma \ref{le3}, $$\|F_{A_t}\|_w^2 \leq C t^2.$$

Let $\vartheta$ be a smooth non-negative  function on  $U$ such that $\vartheta\equiv 1$ on $U'$, and $U'\subset {\rm supp}(\vartheta) \subset U$.  By  Lemma \ref{le2},  \bea
 \int_{U'}\frac{1}{4}  \|\nabla_x  F_{A_t}\|_w^2dx_1dx_2 & \leq &  \int_{U}\vartheta \Delta \|F_{A_t}\|_w^{2}  dx_1dx_2 +C t   \nonumber\\ & \leq& \int_{U}\max \{0, \Delta \vartheta\}  \|F_{A_t}\|_w^{2}  dx_1dx_2 +C_{22}t \nonumber \\ & \leq & C  ( \int_{U}  \|F_{A_t}\|_w^{2}  dx_1dx_2 +t) \nonumber \\ & \leq & C t, \nonumber
\eea and we obtain the second estimate.
 \end{proof}

 \begin{proof}[Proof of Proposition \ref{prop2}] Firstly,  we prove the $C^0$-estimate of $
F_{A_t}$.   Assume that there is a sequence $t_k \rightarrow 0$ such that  $$ t_k^{-1}\sup_{M_{w_k}}|F_{A_{t_k}}|_{\omega^{SF}}\rightarrow\infty, $$  where $w_k \rightarrow w_0$ in $U'$.

In Section 2.4,  we saw that for
    $D_r=\{\tilde{w} \in \mathbb{C}| |\tilde{w}|<r\}$,   one can define smooth embeddings $\Phi_{k,r}: D_r \times M_{w_0}   \rightarrow M_U$ by $$(\tilde{w}, a_1+a_2 \tau (w_0))\mapsto ( w_k+ \sqrt{t_k}\tilde{w},  a_1+a_2 \tau(w_k+\sqrt{t_k}\tilde{w})),   \  \  \  a_1,a_2 \in\mathbb{R}/\mathbb{Z},$$  using the identification of $M_U$ with $(U\times\mathbb{C})/{\rm Span}_{\mathbb{Z}}\{1, \tau\}$.
We also demonstrated that  $d\Phi_{k,r}^{-1}I d\Phi_{k,r} $ $\rightarrow I_\infty$, where $I$ is the complex structure of $M$, and $I_\infty$ denotes the complex structure of $\mathbb{C} \times M_{w_0}. $ Furthermore, as $t_k\rightarrow0$, we have both $$\Phi_{k,r}^*  t_k^{-1} \omega_{t_k}^{SF} \rightarrow \omega_\infty\ \ \ {\rm and}  \ \  \ (T_{\sigma_0} \circ \Phi_{k,r})^*  t_k^{-1} \omega_{t_k} =\Phi_{k,r}^*  t_k^{-1}T_{\sigma_0}^* \omega_{t_k} \rightarrow \omega_\infty $$ in the $C^\infty$-sense on $ D_r \times M_{w_0}$.
 For any $t_k$, we identify $  D_r\times M_{w_0}$ with $ \Phi_{k,r}( D_r\times M_{w_0})$ by $\Phi_{k,r}$.
  We have the curvature bound
     $$ | F_{\Xi_{t_k}}|_{ t_k^{-1} \omega_{t_k}^{SF}}\leq C t_k^{\frac{1}{2}},  \  \ {\rm and} \ \  | F_{\Xi_{t_k}}|_{  \omega_\infty}\leq 2C t_k^{\frac{1}{2}},$$  by (\ref{hyp:curv}).

 Since   $\Xi_{t_k}$ is Yang-Mills, by the strong Uhlenbeck compactness theorem (cf. Theorem \ref{Ucompact}),  there exists a subsequence and a family of unitary gauges $u_{t_k}$, such that $$\Xi_{t_k}'=u_{t_k}(\Xi_{t_k}) \rightarrow \Xi_\infty$$ in the locally   $C^\infty$-sense on $D_r\times M_{w_0}$, where $\Xi_\infty$ is a flat $SU(n)$-connection.  Note that $ F_{\Xi_{t_k}'}=u_{t_k}F_{\Xi_{t_k}}u_{t_k}^{-1}$, and so
   $$| F_{\Xi_{t_k}'}|_{ t_k^{-1} \omega_{t_k}^{SF}}= | F_{\Xi_{t_k}}|_{ t_k^{-1} \omega_{t_k}^{SF}}\leq C t_k^{\frac{1}{2}}\ \ \  {\rm and} \ \ \ | F_{\Xi_{t_k}'}|_{  \omega_\infty}\leq 2C t_k^{\frac{1}{2}}.$$
   Furthermore we have $     \| F_{\Xi_{t_k}'}\|_{ C^{\ell}(  \omega_\infty)}\rightarrow 0  $ for any $\ell \geq 0$, when $t_k \rightarrow 0$.
   Now, recall the Weitzenb\"{o}ck formula   $$0= \Delta _{\Xi_{t_k}'}F_{\Xi_{t_k}'}=\nabla_{\Xi_{t_k}'}^\ast\nabla_{\Xi_{t_k}'}F_{\Xi_{t_k}'}+ R_{{t_k}^{-1}\omega_{t_k}}\#F_{\Xi_{t_k}'}+F_{\Xi_{t_k}'}\#F_{\Xi_{t_k}'},$$  which is an elliptic partial differential equation with smooth coefficients.   The $L^p$-estimate for elliptic equations (cf. \cite{GT}, and the appendix of \cite{Bess}) gives  $$ \|F_{\Xi_{t_k}'}\|_{L^p_2(\omega_\infty)}\leq C  \|F_{\Xi_{t_k}'}\|_{L^p(\omega_\infty)} \leq C  t_k^{\frac{1}{2}},$$ for any $p>2$.

  We have  $w-w_k=\sqrt{t_k}\tilde{w}$ through $\Phi_{k,r}$, and let $\tilde{w}=\tilde{x}_1+i\tilde{x}_2$.  By  (\ref{thing20}),  \bea
  \label{keyequality}
 (\nabla_{x_1}^2+\nabla_{x_2}^2)F_{A_{t_k}'}  &=&- t_k^{-1} d_{A_{t_k}'}\star_{w} d_{A_{t_k}'}(G_4  \star_w F_{A_{t_k}'})  +\sum_{ij}\kappa_{t_k,i}'\# \kappa_{t_k,j}'  \nonumber\\ & & +d_{A_{t_k}'}\star_{w} d_{A_{t_k}'}( \sum_{i=1,2} \kappa_{t_k,i}'\# G_5 ),
\eea
 where $\nabla_{x_j}=\partial_{x_j}+B_{t_k,j}'$,   $G_4= (W+ G_2)^{-1}(1+G_0+G_1)$ and $G_5= (W+ G_2)^{-1}G_3$. Recall $$\|G_1\|_{C^{0}}+\|d^f G_1\|_{C^{\ell}}+\|G_j\|_{C^{\ell}}\leq C t_k^\nu $$ for  $\nu \gg 1$. Let $z=\tilde{y}_1+i  \tilde{y}_2$, and set $\nabla_{A_{t_k}', y_j}=\partial_{\tilde{y}_j}+A_{t_k,j}'$. By the Weitzenb\"{o}ck formula, $$d_{A_{t_k}'}d_{A_{t_k}'}^* F_{A_{t_k}'} =\nabla_{A_{t_k}'}^*\nabla_{A_{t_k}'} F_{A_{t_k}'}+ F_{A_{t_k}'}\#F_{A_{t_k}'}.$$ The connection Laplacian above is given by   $$\nabla_{A_{t_k}'}^*\nabla_{A_{t_k}'}=-W^{-1}(\nabla_{A_{t_k}', \tilde{y}_1}^2+\nabla_{A_{t_k}',\tilde{y}_2}^2), $$ since $|\partial_{\tilde{y}_j}|^2_{\omega^{SF}}=W$. 

We want to bound terms on the right hand side of \eqref{keyequality}. Scaling gives $B_{t_k,i}'dx_i=\sqrt{t_k}B_{t_k,i}'d\tilde{x}_i$ and  $\kappa_{t_k,i}'dx_i=\sqrt{t_k}\kappa_{t_k,i}'d\tilde{x}_i$, in addition to $$   F_{B,t_k}dx_1\wedge dx_2=t_k F_{B,t_k}d\tilde{x}_1\wedge d\tilde{x}_2.$$
This leads to the following control of the mixed terms
$$ | \sqrt{t_k}\kappa_{t_k,i}' |_{\omega_\infty}\leq 2C  t_k^{\frac{1}{2}},\ \ \ \ \| \sqrt{t_k}\kappa_{t_k,i}' \|_{C^{\ell}(  \omega_\infty)} \rightarrow 0, $$
 and $$   \| \sqrt{t_k}\kappa_{t_k,i}' \|_{L^p_2(  \omega_\infty)}\leq  \|F_{\Xi_{t_k}'}\|_{L^p_2(\omega_\infty)} \leq C   t_k^{\frac{1}{2}}.$$
Additionally, writing $\nabla_{\tilde{x}_j}=\partial_{\tilde{x}_j}+\sqrt{t_k }B_{t_k,j}'$, we have $$  \nabla_{\tilde{x}_1}^2+\nabla_{\tilde{x}_2}^2=t_k(\nabla_{x_1}^2+\nabla_{x_2}^2).$$ The bound $|\partial^{\ell}_{y_j} G_5|\leq C $ gives $$\|t_k^{\frac{1}{2}}d_{A_{t_k}'}\star_{w} d_{A_{t_k}'}( \sum_{i=1,2} \kappa_{t_k,i}'\# G_5 )\|_{L^p(\omega_\infty)}\leq C  t_k^{\frac{1}{2}}$$ for any $p>2$. Furthermore
     $$\|\sum_{ij}\kappa_{t_k,i}'\# \kappa_{t_k,j}'\|_{ C^0(\omega_\infty)}\leq C.$$

Now, if we write $G_4= W^{-1}(1+G_0)+G_6$, then  $$\frac{1}{2} W^{-1}(w_0)\leq G_4 \leq 2  W^{-1}(w_0),  \ \ |\partial_{\tilde{y}_j}^{\ell}G_6|\leq C t_k^{\nu},$$
and
   \bea d_{A_{t_k}'}d_{A_{t_k}'}^* G_4 F_{A_{t_k}'}& = &G_4   d_{A_{t_k}'}d_{A_{t_k}'}^*  F_{A_{t_k}'}+d^f G_6\# \nabla_{A_{t_k}'} F_{A_{t_k}'}\nonumber  \\ & & + \partial_{\tilde{y}_i \tilde{y}_j}^2 G_6\#  F_{A_{t_k}'}. \nonumber \eea
We define the operator $$\mathcal{D}_k=\nabla_{\tilde{x}_1}^2+\nabla_{\tilde{x}_2}^2-G_4 \nabla_{A_{t_k}'}^*\nabla_{A_{t_k}'}=\nabla_{\tilde{x}_1}^2+\nabla_{\tilde{x}_2}^2+W^{-1}G_4(\nabla_{A_{t_k}', \tilde{y}_1}^2+\nabla_{A_{t_k}',\tilde{y}_2}^2), $$ which  is a uniformly  elliptic operator of order two. Then  $ F_{A_{t_k}'}$ satisfies the following  elliptic equation
          \bea & &
\mathcal{D}_k  F_{A_{t_k}'}  -   d^f G_6 \# \nabla_{A_{t_k}'} F_{A_{t_k}'}- \partial_{\tilde{y}_i \tilde{y}_j}^2 G_6 \#  F_{A_{t_k}'} \\ &= &  G_4 F_{A_{t_k}'}\#F_{A_{t_k}'}  + t_k \sum_{ij}\kappa_{t_k,i}'\# \kappa_{t_k,j}'   + t_k d_{A_{t_k}'}\star_{w} d_{A_{t_k}'}( \sum_{i=1,2} \kappa_{t_k,i}'\# G_5 ) \nonumber\\ & =& G_7. \nonumber
\eea

     By the  $L^p$-estimate for elliptic equations, for any $p>2$, $$ \| F_{A_{t_k}'}\|_{L_2^p(D_{r'}\times M_{w_0})}\leq C(\| F_{A_{t_k}'}\|_{L^2(D_{r}\times M_{w_0})}+\| G_7\|_{L^p(D_{r}\times M_{w_0})}),$$ for a $r'<r$. We obtain  $$\| F_{A_{t_k}'}\|_{L_2^p(D_{r'}\times M_{w_0})}\leq C  t_k,$$  since  $$\| G_7\|_{L^p(D_{r}\times M_{w_0})}\leq C  (\| F_{A_{t_k}'}\|_{C^0(D_{r}\times M_{w_0})}^2+t_k) \leq C t_k,$$
     and  $$ \| F_{A_{t_k}'}\|_{L^2(D_{r}\times M_{w_0})}^2=\int_{D_{r}}\|F_{A_{t_k}'}\|_w^2d\tilde{x}_1d\tilde{x}_2 \leq C  t_k^2    $$ by Lemma \ref{le-l2}.
        The Sobolev embedding theorem gives $$\|F_{A_{t_k}'}\|_{C^{1,\alpha}(D_{r'}\times M_{w_0})} \leq  C  t_k, $$ and  thus $$\|F_{A_{t_k}}\|_{C^0(M_{w_k})}=\|F_{A_{t_k}'}\|_{C^0(M_{w_k})}\leq \|F_{A_{t_k}'}\|_{C^{1,\alpha}{(D_{r'}\times M_{w_0})}}\leq C  t_k ,$$  which is a contradiction.

         Therefore we obtain  the $C^0$-estimate, i.e. $$  \|F_{A_t}\|_{C^{0}(M_{U'},\omega^{SF})} \leq C    t,$$ for a constant $C >0$, and
         $$  \|F_{B,t}\|_{C^{0}(M_{U'},\omega^{SF})} \leq C (t^{-1}  \|F_{A_t}\|_{C^{0}(M_{U'},\omega^{SF})} +  \|\kappa_{t,j}\|_{C^{0}(M_{U'},\omega^{SF})}) \leq C  , $$ by (\ref{ASD2}).
 \end{proof}

\section{Further estimates for small fiberwise curvature}
\label{furthreestimates}

We continue our discussion of the previous  section, and prove further estimates under  the exact  same setup.  Let $U\subset\subset N^o$ be an open subset,  biholomorphic to a disk in $\mathbb{C}$, and $M_U\cong (U\times \mathbb{C})/{\rm Span}_{\mathbb{Z}}\{1, \tau\}$. Fix a  trivialization $P|_{M_U}\cong  M_U \times SU(n)$ and $\mathcal{V}|_{M_U}\cong  M_U \times \mathbb{C}^n$. Under such trivialization,  the Hermitian metric $H$ is the absolute value $|\cdot|$, the connection $\Xi_t$ is a matrix valued 1-form, and the curvature $F_{\Xi_t}$ is a matrix valued 2-form. Assume that for $t\ll 1$,  (\ref{hyp:curv}) and (\ref{hyp:poincare}) hold, and thus all conclusions of Section \ref{C0boundsection}  hold.

 Recall that a fiberwise flat  connection   \begin{equation}\label{bconnections3} A_{0,t}=  \pi ({\rm Im} (\tau))^{-1} ({\rm diag}\{q_{1,t}, \cdots, q_{n,t}\}\bar{\theta}-{\rm diag}\{\bar{q}_{1,t}, \cdots, \bar{q}_{n,t}\}\theta)  \end{equation}  is induced by $D_t \cap M_U$ (see Section 3.3), i.e.  $D_t \cap M_w=\{q_{1,t}(w), \cdots, q_{n,t}(w)\}$. The goal of this section is the following proposition, which shows the relationship between the energy of curvature and the spectral covers. Here, as above, the coordinate derivative in the base is computed in our fixed frame.

\begin{prop}\label{prop2+0} If (\ref{hyp:curv}) and (\ref{hyp:poincare}) hold for $t\ll 1$,
 we have the following inequalities. For $U'\subset\subset U$,  $$  \|F_{\Xi_t}\|_{L^2(M_{U'}, \omega_t)}^2  \leq C_1(t+\int_{U'}\sum_{j=1,2}\|\partial_{x_j}A_{0,t}\|_w^2  dx_1dx_2), \  \  \  and$$ $$  \|F_{\Xi_t}\|_{L^2(M_{U'}, \omega_t)}^2  \geq C_1^{-1}(\int_{U'}\sum_{j=1,2}\|\partial_{x_j}A_{0,t}\|_w^2  dx_1dx_2-t), $$
       where    the constant $C_1$ may depend on the distance from $U'$ to $\partial U$, but is   independent of $t$.
 \end{prop}

The proof rests on several important lemmas.

\begin{lem}
There exists a constant $C_2$ such that for all $t\ll1$,
\be
\sup_{M_{U'}}|\nabla_{A_{0,t}}F_{A_t}|_{\omega^{SF}}\leq C_2t^{\frac{1}{2}}.\nonumber
\ee
\end{lem}
\begin{proof}
By \eqref{connectionC0}, it suffices to prove the above bound for $\nabla_{A_{t}}F_{A_t}$.   We argue by contradiction. Let $t_k\rightarrow 0$ such that $$\lim_{k\rightarrow \infty}t_k^{-\frac{1}{2}}\sup_{M_{U'}}|\nabla_{A_{t_k}}F_{A_{t_k}}|_{\omega^{SF}} = \infty.$$
Let $p_k\in M_{U'}$ be the points where the supremum is attained, and in addition let $f(p_k) := w_k\rightarrow w_0 \in U$. As in Section 2.5, we consider the rescaled metrics $\hat \omega_{k} = t_k^{-1}\omega_{t_k}$ and the embeddings $\Phi_{k,r}: D_r \times  M_{w_0}  \rightarrow M_U$ defined by $$( \tilde{w},  a_1+a_2 \tau (w_0))\mapsto (w_k +\sqrt{t_k}\tilde{w},  a_1+a_2 \tau(w_k+\sqrt{t_k}\tilde{w})),  \  \  a_1,a_2 \in\mathbb{R}/\mathbb{Z},$$  where $D_r = \{\tilde{w} \in \mathbb{C}| |\tilde{w}|<r\}$. We have seen that if $I$ is the complex structure of $M$, and $I_\infty$  the complex structure of $\mathbb{C} \times M_{w_0} $, then     $d\Phi_{k,r}^{-1}I d\Phi_{k,r} \rightarrow I_\infty$, and in addition
$$\Phi_{k,r}^*  t_k^{-1} \omega_{t_k}^{SF} \rightarrow \omega_\infty\ \ \ {\rm and} \ \  \  \Phi_{k,r}^*   \hat{\omega}_{k} \rightarrow \omega_\infty$$
in the $C^\infty$-sense on $ D_r \times M_{w_0}$. Here $\omega_\infty$ is  a flat   K\"{a}hler metric on  $ D_r \times M_{w_0}$.
Denote by $\hat \Xi_k$  the pull-back of $\Xi_{t_k}$ by $\Phi_{k,r}$, and  identify $  D_r\times M_{w_0}$ with $ \Phi_{k,r}( D_r\times M_{w_0})$ via  $\Phi_{k,r}$. By our hypothesis,
\begin{equation}\label{step1F}
\sup_{D_r\times M_{w_0}}t_k^{-\frac{1}{2}}|\nabla_{\hat \Xi_k}F_{\hat\Xi_k}|_{\omega_\infty} = \infty,
\end{equation} while by (\ref{hyp:curv}) we have the curvature bounds
     $$ | F_{\hat{\Xi}_{k}}|_{ t_k^{-1} \omega_{t_k}^{SF}}\leq C t_k^{\frac{1}{2}}  \  \ {\rm and} \ \  | F_{\hat{\Xi}_{k}}|_{ \hat{ \omega}_k}\leq 2C t_k^{\frac{1}{2}}.$$

 Since $\hat\omega_k$ is equivalent to a fixed metric, standard Yang-Mills theory gives the first derivative bound  $ |\nabla_{\hat\Xi_k}F_{\hat{\Xi}_{k}}|_{ \hat{ \omega}_k}\leq C $ (for instance see \cite{Wein}), but this is of course not enough to obtain a contradiction.  So following \cite{Wein}, as in the proof of Lemma \ref{energybound}, we consider the the  Bochner formula
\begin{equation}
0=\Delta_{\hat\omega_k} |F_{\hat\Xi_k}|_{\hat\omega_k}^2-2|\nabla_{\hat\Xi_k} F_{\hat\Xi_k}|_{\hat\omega_k}^2+F_{\hat\Xi_k}\#F_{\hat\Xi_k}\#F_{\hat\Xi_k}+R_{\hat\omega_k}\#F_{\hat\Xi_k}\#F_{\hat\Xi_k}.\nonumber
\end{equation}
We have seen that the curvature of the base metric satisfies $|R_{\omega_t}|_{\omega_t}^2\leq C$  on a compact subset of $N_0$, and scaling only improves this bound  $|R_{\hat\omega_k}|_{\hat\omega_k}^2\leq C t_k^2.$  Rearranging terms, and multiplying by a positive function $\chi$ yields
\begin{equation}
2\chi|\nabla_{\hat\Xi_k} F_{\hat\Xi_k}|_{\hat\omega_k}^2\leq \chi\Delta_{\hat\omega_k} |F_{\hat\Xi_k}|_{\hat\omega_k\hat\Xi_k}^2+\chi |F_{\hat\Xi_k}|_{\hat\omega_k}^3+C \chi |F_{\hat\Xi_k}|_{\hat\omega_k}^2.\nonumber
\end{equation}
If $\eta$ is a positive bump function supported in $D_{r/2}$ and satisfying $\eta\equiv1$ in $D_{r/4}$, we specify $\chi=f^{-1}(\eta)$. Integrating the above inequality gives
\bea\label{L2controlderiv}
 \int_{{D_{\frac r4}}\times M_{w_0}}|\nabla_{\hat\Xi_k} F_{\hat\Xi_k}|_{\hat\omega_k}^2\hat\omega_k^2 &\leq & \frac12\int_{D_{\frac r2}\times M_{w_0}}\Delta_{\hat\omega_k}\chi|F_{\hat{\Xi}_k}|_{\hat\omega_k}^2\hat\omega_k^2+C\int_{D_{\frac r2}\times M_{w_0}} t_k \nonumber \\ & \leq & Ct_k,
\eea
where the constant $ C$ depends on $r$, which again we take to be fixed.

We next turn to the higher order Bochner formula for Yang-Mills connections:
\begin{align}
0=&\Delta_{\hat\omega_k} |\nabla_{\hat\Xi_k} F_{\hat\Xi_k}|_{\hat\omega_k}^2-2|\nabla^2_{\hat\Xi_k} F_{\hat\Xi_k}|_{\hat\omega_k}^2+\nabla_{\hat \Xi_t} F_{\hat\Xi_k}\#\nabla_{\hat\Xi_k} F_{\hat\Xi_k}\#F_{\hat\Xi_k}\nonumber\\
&+R_{\hat\omega_k}\#\nabla_{\hat\Xi_k} F_{\hat\Xi_k}\#\nabla_{\hat\Xi_k} F_{\hat\Xi_k}+\nabla_{\hat\omega_k} R_{\hat\omega_k}\#F_{\hat\Xi_k}\#\nabla F_{\hat\Xi_k}.\nonumber
\end{align}
Since $|\nabla_{\hat\omega_k} R_{\hat\omega_k}|_{\hat\omega_k}\leq t_k|\nabla_{\omega_{t_k}} R_{\omega_{t_k}}|_{\omega_{t_k}}\leq C_U t_k$, we have
\begin{equation}
-\Delta_{\hat\omega_k} |\nabla_{\hat\Xi_k} F_{\hat\Xi_k}|_{\hat\omega_k}^2\leq  C(t_k^{\frac{1}{2}} |\nabla_{\hat\Xi_k} F_{\hat\Xi_k}|_{\hat\omega_k}^2+t_k^{\frac{3}{2}} |\nabla_{\hat\Xi_k} F_{\hat\Xi_k}|_{\hat\omega_k}).\nonumber
\end{equation}
Set
$$\psi_k:= |\nabla_{\hat\Xi_k} F_{\hat\Xi_k}|_{\hat\omega_k}^2/\sup_{D_r\times M_{w_0}}|\nabla_{\hat\Xi_k} F_{\hat\Xi_k}|^2_{\hat\omega_k}.$$
The above Bochner formula, in addition to our hypothesis \eqref{step1F}, gives $$-\Delta_{\hat\omega_k} \psi_k\leq C( t_k^{\frac{1}{2}}+ t_k)\leq 1,$$ for $k\gg 1$. We now follow the argument used in Lemma \ref{C0s}. Let $\hat p_k$ be the pullbacks of the points $p_k$ via $\Phi_{k,r}$. These are the points realizing the supremum of $|\nabla_{\hat\Xi_k} F_{\hat\Xi_k}|^2_{\hat\omega_k}$, so that $\psi_k(\hat p_k) = 1$. Now construct a sequence of functions $u_k$ solving $\Delta_{\hat\omega_k} u_k=-1$ and $u_k(\hat p_k)=1$. Working on a small ball   $B_{\hat\omega_k}(\hat p_k, r_0)$, we can assume that $u_k>\varepsilon_0$ for some $\varepsilon_0>0$ independent of $k$. Then since $-\Delta(\psi_k-u_k)\leq 0$, by the mean value inequality, there exists a $\delta>0$ depending only $\varepsilon_0$ and $r_0$ such that
\begin{equation}
\delta<\int_{B_{\hat\omega_k}(\hat p_k,r_0)}u_k\leq \int_{B_{\hat\omega_k}(\hat p_k,r_0)}\psi_k\leq  \int_{D_{r/4}\times M_{w_0}}\psi_k\leq  \frac{C_8t_k}{\sup\limits_{D_r\times M_{w_0}}|\nabla_{\hat\Xi_k} F_{\hat\Xi_k}|^2_{\hat\omega_k}}\nonumber
\end{equation}
where the final inequality follows from \eqref{L2controlderiv}. This contradicts \eqref{step1F},  completing the proof.
\end{proof}
Next, we have a  $C^{1,\alpha}$-estimate for $A_t$.

\begin{lem}\label{prop2+00} For all  $w\in U'$, and for all $t\ll 1$, $0<\alpha <1$, $$\|A_t-A_{0,t}\|_{C^{1,\alpha}(M_w)}\leq C_{3}t^{\frac{1}{2}} \ \  and  \  \   \|\nabla^2_{A_{0,t}}\hat s_t\|_{C^{0,\alpha}(M_w)}\leq C_{4}t^{\frac{1}{2}},$$ for   constants $C_{3}$ and $C_{4}$ independent of $w$ and $t$.
 \end{lem}
\begin{proof}
We begin by recalling inequality \eqref{lq2}, which follows from Proposition \ref{prop2},  and properties of $\hat s_t$
\be
\|\hat s_t\|_{L^p_2(M_w)}\leq Ct.\nonumber
\ee
We would like to extend the above estimate to the case of $p=\infty$. To accomplish this, we  turn to the higher order elliptic a priori estimate
\bea
\|\hat s_t\|_{L^p_3(M_w)}&\leq& C \left(\|\Delta_{A_{0,t}}\hat s_t\|_{L^p_1(M_w)}+\|\hat s_t\|_{L^p(M_w)}\right)\nonumber\\
&\leq&C \left(\|\Delta_{A_{0,t}}\hat s_t\|_{L^p_1(M_w)}+ t \right).\nonumber
\eea
Taking one fiber derivative of \eqref{curvatureformula}, and using the fact that $\|\hat s_t\|_{C^0(M_w)}$ and $\|\nabla_{A_{0,t}}\hat s_t\|_{C^0(M_w)}$ are controlled by $t$, we see that
\be
\|\Delta_{A_{0,t}}\hat s_t\|_{L^p_1(M_w)}\leq \|\nabla_{A_{0,t}}F_{A_t}\|_{L^p(M_w)}+t\|\hat s_t\|_{L^p_3(M_w)}+t\|\hat s_t\|_{L^p_2(M_w)}.\nonumber
\ee
Thus, for $t$ small enough
\be
\|\hat s_t\|_{L^p_3(M_w)}\leq C (t+ \|\nabla_{A_{0,t}}F_{A_t}\|_{L^p(M_w)})\leq C  t^{\frac{1}{2}}.\nonumber
\ee
By Morrey's inequality we have
\be
\label{thing6}
\|\nabla^2_{A_{0,t}}\hat s_t\|_{C^{0,\alpha}(M_w)}\leq C t^{\frac{1}{2}}.
\ee
\end{proof}

If we let $\Xi_{t}^0=e^{-\hat{s}_t}(\Xi_t)$, then $\Xi_{t}^0|_{M_w}=A_{0,t}$, and we write  $$\Xi_{t}^0=A_{0,t}+B_{t,1}^0dx_1+B_{t,2}^0dx_2,  \  \  \  {\rm and} \  \  F_{\Xi_{t}^0}=-\kappa_{t,1}^0dx_1-\kappa_{t,2}^0dx_2-F_{B,t}^0dx_1\wedge dx_2,$$ where
$$\kappa_{t,j}^0=\partial_{x_j} A_{0,t}-d_{A_{0,t}}B_{t,j}^0.$$ Note that  we still have $F_{\Xi_{t}^0}^{0,2}=0$, which implies
\begin{equation}\label{HYMB++}\star_w \kappa_{t,1}^0=\kappa_{t,2}^0,\end{equation} and thus $$ \star_w\partial_{x_1} A_{0,t}-\partial_{x_2} A_{0,t}=\star_w d_{A_{0,t}}B_{t,1}^0-d_{A_{0,t}}B_{t,2}^0.  $$
 Since $$\star_w\partial_{x_1} A_{0,t}-\partial_{x_2} A_{0,t} \in \ker \Delta_{A_{0,t}}, \ \ d_{A_{0,t}}B_{t,2}^0\in {\rm Im}d_{A_{0,t}},  \ \  {\rm and} \  \  \star_w d_{A_{0,t}}B_{t,1}^0 \in {\rm Im}d_{A_{0,t}}^*,$$ we have $ \star_w\partial_{x_1} A_{0,t}=\partial_{x_2} A_{0,t}$   and $d_{A_{0,t}}B_{t,j}^0=0$ by the Hodge decomposition.  As a result we obtain
\begin{equation}\label{curv+++} \kappa_{t,j}^0=\partial_{x_j} A_{0,t}.\end{equation}
A direct calculation shows
\bea\label{for+==}
\kappa_{t,j}-\kappa_{t,j}^0 & =&\partial_{x_j} (A_t-A_{0,t})-d_{A_{t}}B_{t,j} \\
&= &\nabla_{x_j} (A_t-A_{0,t})-[B_{t,j}, A_t-A_{0,t}]\nonumber\\ & & -d_{A_{0,t}}B_{t,j} + [A_{0,t}-A_t,B_{t,j}]\nonumber \\ &=& \nabla_{x_j} (A_t-A_{0,t})-d_{A_{0,t}}B_{t,j}.
\nonumber \eea

Now, by (\ref{ASD1}), (\ref{HYMB++}) and (\ref{for+==}),  $$ \star_w\nabla_{x_1} (A_t-A_{0,t})-\nabla_{x_2} (A_t-A_{0,t})=\star_w d_{A_{0,t}}B_{t,1}-d_{A_{0,t}}B_{t,2},  $$ and since $\star_w d_{A_{0,t}}B_{t,1} \bot d_{A_{0,t}}B_{t,2}$, i.e. $\langle\star_w d_{A_{0,t}}B_{t,1} , d_{A_{0,t}}B_{t,2}\rangle_w=0$,  we have
$$ \|d_{A_{0,t}}B_{t,j}\|_w \leq  \sum_{i=1,2} \|\nabla_{x_i} (A_t-A_{0,t}) \|_w,$$ for any  $w\in  U$.   Consequently, for $j=1,2$ \begin{equation}\label{?+++}\|\kappa_{t,j}-\kappa_{t,j}^0\|_w \leq 2 \sum_{i=1,2} \|\nabla_{x_i} (A_t-A_{0,t}) \|_w. \end{equation}
 Furthermore, if we decompose $B_{t,j}=B_{t,j}^o+B_{t,j}^\bot$, where $B_{t,j}^o\in\ker d_{A_{0,t}}$ and $B_{t,j}^\bot \bot\ker d_{A_{0,t}}$, then \begin{equation}\label{?+++} \|B_{t,j}^\bot\|_w \leq C  \|d_{A_{0,t}}B_{t,j}\|_w \leq C  \sum_{i=1,2} \|\nabla_{x_i} (A_t-A_{0,t}) \|_w,\end{equation} by Lemma \ref{eigen}. We need one more Lemma before we are ready to prove Proposition \ref{prop2+0}.

\begin{lem}
\label{prop1}
On  $U'\subset\subset U$,   we have
\be
\int_{U}\sum_{j=1,2}\|\nabla_{x_j}(A_t-A_{0,t})\|_w^2dx_1dx_2\leq C_{5}(t^2+\int_U\sum_{j=1,2}\|\nabla_{x_j} F_{ A_t} \|_{w}^2dx_1dx_2),\nonumber
\ee for a constant $C_{5}>0$. Consequently, by ii) of  Proposition \ref{prop2}, $$ \int_{U}\sum_{j=1,2}\|\kappa_{t,j}-\kappa_{t,j}^0\|_w^2  dx_1dx_2\leq C_{6} t.$$ \end{lem}
\begin{proof}

We denote two important terms by $$\Lambda= \sum_{j=1,2}\|\nabla_{x_j}(A_t-A_{0,t})\|_w, \  \  \ \Theta=\sum_{j=1,2}\|\nabla_{x_j} F_{ A_t} \|_{w}. $$

First, for $j=1,2$, we decompose   $\nabla_{x_j}\hat s_t=\nabla_{x_j}\hat s_t^o+\nabla_{x_j}\hat s_t^\bot$, where  $\nabla_{x_j}\hat s_t^\bot$ is perpendicular to the kernel of $d_{A_{0,t}}$,
and $\nabla_{x_j}\hat s_t^o\in \ker d_{A_{0,t}}$.  Recall that $\ker d_{A_{0,t}}=\{{\rm diag}\{\eta_1, \cdots, \eta_n\}\in\mathfrak{sl}(n, \mathbb{C})\}$, and  as a volume form $\omega^{SF}|_{M_w}=dv$ is independent of $w$ under the identification $M_w \cong T^2$. For any $\eta\in \ker d_{A_{0,t}}$, since $[B_{t,j}^o, \eta]=0$,   $$\nabla_{x_j} \eta=\partial_{x_j}\eta+[B_{t,j},\eta]=[B_{t,j}^\bot,\eta].$$
  Thus $$0=\partial_{x_j} \langle \hat s_t, \eta \rangle_w = \langle \nabla_{x_j} \hat s_t, \eta \rangle_w+\langle  \hat s_t, \nabla_{x_j} \eta \rangle_w=\langle \nabla_{x_j} \hat s_t^o, \eta \rangle_w+\langle  \hat s_t, [B_{t,j}^\bot, \eta] \rangle_w ,$$ and  by (\ref{?+++})
   $$ \|\nabla_{x_j}\hat s_t^o\|_w\leq C  \|\hat s_t\|_{C^0} \|B_{t,j}^\bot \|_w\leq C  t\Lambda.$$  Along with Lemma \ref{eigen}, this implies $$  \|\nabla_{x_j}\hat s_t\|_w\leq C  (\|\nabla_{x_j}\hat s_t^\bot\|_w+ t\Lambda)\leq  C  (\|d_{A_{0,t}}\nabla_{x_j}\hat s_t\|_w+ t\Lambda). $$    Since $$d_{A_{t}}\nabla_{x_j}\hat s_t=d_{A_{0,t}}\nabla_{x_j}\hat s_t+[A_t-A_{0,t},\nabla_{x_j}\hat s_t], \  \  {\rm  and}  \ \   \| A_t-A_{0,t}\|_{C^0}\leq Ct, $$ we obtain $$  \|\nabla_{x_j}\hat s_t\|_w\leq   C (\|d_{A_{t}}\nabla_{x_j}\hat s_t\|_w+ t\Lambda).  $$

Next, take the derivative of \eqref{connectionformula} in the base direction to see
\bea
\|\nabla_{x_j}(A_t-A_{0,t})\|_{w}^2&\leq& 2\|\nabla_{x_j}(\Upsilon(\hat s_t))d_{A_{t}}\hat s_t\|_{w}^2+2\|\Upsilon(\hat s_t)\nabla_{x_j}(d_{A_{t}}\hat s_t)\|_{w}^2.\nonumber
\eea  We concentrate on the two terms on the right hand side above separately. By Lemma \ref{connectionC0lem} and Proposition \ref{prop2},  $\hat s_t$, $\nabla_{A_{0,t}}\hat s_t$ and $A_t-A_{0,t}$ are  bounded in $C^0$ by $t$, and so the first term satisfies
\be
 \|\nabla_{x_j}(\Upsilon(\hat s_t))d_{A_{t}}\hat s_t\|_{w}^2\leq t^2 C \|\nabla_{x_j}\hat s_t\|_{w}^2\leq t^2 C (\|d_{A_{t}}\nabla_{x_j}\hat s_t\|_{w}^2+t^2\Lambda^2).\nonumber
\ee
To bound the second of the two terms, note that $\kappa_{t,j}$ is bounded, and $ \nabla_{x_j} d_{A_{t}}-d_{A_{t}}\nabla_{x_j}=\kappa_{t,j} $. Thus
\be
\|\Upsilon(\hat s_t)\nabla_{x_j}(d_{A_{t}}\hat s_t)\|_{w}^2\leq C \|\hat s_t\|_{w}^2+2\|d_{A_{t}}\nabla_{x_j}\hat s_t\|_{w}^2\leq C t^2+2\|d_{A_{t}}\nabla_{x_j}\hat s_t\|_{w}^2,\nonumber
\ee
from which we conclude
\be
\Lambda^2\leq 2\sum_{j=1,2}\|\nabla_{x_j}(A_t-A_{0,t})\|_{w}^2\leq 6\sum_{j=1,2}\|d_{A_{t}}\nabla_{x_j}\hat s_t\|_{w}^2+C t^2.\nonumber
\ee
Therefore it suffices to bound $\|d_{A_{t}}\nabla_{x_j}\hat s_t\|_{w}^2$.

Integration by parts, along with Lemma \ref{eigen}, gives\bea
\int_{M_w}|d_{A_{t}}\nabla_{x_j}\hat s_t|^2\omega^{SF}&\leq& \int_{M_w}|\nabla_{x_j}\hat s_t\|\Delta_{A_{t}}\nabla_{x_j}\hat s_t|\omega^{SF} \nonumber\\
&\leq&\|\nabla_{x_j}\hat s_t\|_{w}\|\Delta_{A_{t}}\nabla_{x_j}\hat s_t\|_{w}\nonumber\\
&\leq&C (\|d_{A_{t}}\nabla_{x_j}\hat s_t\|_{w}+t\Lambda)\|\Delta_{A_{t}}\nabla_{x_j}\hat s_t\|_{w}\nonumber
\eea
and so
\be
\label{thing7}
\|d_{A_{t}}\nabla_{x_j}\hat s_t\|_{w}^2\leq C \|\Delta_{A_{t}}\nabla_{x_j}\hat s_t\|_{w}^2+t^2\Lambda^2.
\ee
Thus we obtain $$\Lambda^2 \leq C (\sum_{j=1,2}\|\Delta_{A_{t}}\nabla_{x_j}\hat s_t\|_{w}^2+t^2).$$

In order to bound $\Delta_{A_{t}}\nabla_{x_j}\hat s_t$, we turn to  the  equality  (\ref{bigcurveexpression}) for the curvature of $A_t$,  using  the fact  that $A_{0,t}$ is flat,  \bea
F_{ A_t} &=&i\, d_{A_{t}}\star_w d_{A_{t}} \hat s_t    - \tilde{\Upsilon}(\hat s_t ) \bar\partial_{A_{t}}\partial_{A_{t}} \hat s_t +  \tilde{\Upsilon}(-\hat s_t)\partial_{A_{t}}\bar\partial_{A_{t}}\hat s_t \nonumber \\
&&      -  \bar\partial_{A_{t}} \tilde{\Upsilon}(\hat s_t) \wedge \partial_{A_{t}}\hat s_t + \partial_{A_{t}}\tilde{\Upsilon}(-\hat s_t ) \wedge \bar\partial_{A_{t}}\hat s_t \nonumber\\
&&  -  \Upsilon(\hat s_t )\partial_{A_{t}}\hat s_t \wedge \Upsilon(-\hat s_t)\bar\partial_{A_{t}}\hat s_t+  \Upsilon(-\hat s_t)\bar\partial_{A_{t}}\hat s_t\wedge \Upsilon(\hat s_t)\partial_{A_{t}}\hat s_t  .\nonumber
\eea We take the derivative of this equation  in the base direction, and calculate $\nabla_{x_j}F_{A_t}$.
Firstly,
\bea   \nabla_{x_j} d_{A_{t}}\star_w d_{A_{t}} \hat s_t& = & d_{A_{t}}\nabla_{x_j}\star_w d_{A_{t}} \hat s_t+\kappa_{t,j}\# d_{A_{t}} \hat s_t
\nonumber \\ & = & d_{A_{t}}\star_w d_{A_{t}}\nabla_{x_j} \hat s_t+d_{A_{t}}[\star_w\kappa_{t,j} ,\hat s_t ]+\kappa_{t,j}\# d_{A_{t}} \hat s_t
\nonumber\\ & =& d_{A_{t}}\star_w d_{A_{t}}\nabla_{x_j} \hat s_t\pm [\nabla_{x_i}F_{A_t} ,\hat s_t ]+\kappa_{t,j}\# d_{A_{t}} \hat s_t  \nonumber \eea by $\nabla_{x_i}F_{A_t}=d_{A_t}\kappa_{t,i}=\pm d_{A_t}\star_w\kappa_{t,j}$, which implies
 \bea
& & | \nabla_{x_j} d_{A_{t}}\star_w d_{A_{t}} \hat s_t-  d_{A_{t}}\star_w d_{A_{t}}\nabla_{x_j} \hat s_t|\nonumber \\
&& \qquad\qquad\qquad\qquad\leq  C (|\nabla_{A_{0,t}}\hat s_t|+|A_t-A_{0,t}\| \hat s_t|  +|\nabla_{x_i}F_{A_t}\|\hat s_t|),\nonumber \\
&& \qquad\qquad\qquad\qquad\leq  C t(1+\sum_{i=1,2}|\nabla_{x_i}F_{A_t}|). \nonumber  \eea
As a result, we have   $$ \| \nabla_{x_j} d_{A_{t}}\star_w d_{A_{t}} \hat s_t-  d_{A_{t}}\star_w d_{A_{t}}\nabla_{x_j} \hat s_t\|_w  \leq   C  t(1+\Theta). $$

  Secondly, note that   $\nabla_{A_t}= \nabla_{A_{0,t}}+(A_t-A_{0,t})$, and  $$\nabla_{A_t}^2=\nabla_{A_{0,t}}^2+(A_t-A_{0,t})\# \nabla_{A_{0,t}} +\nabla_{A_{0,t}}(A_t-A_{0,t}) +(A_t-A_{0,t})\# (A_t-A_{0,t}). $$
  A direct calculation shows
  \bea
  &&\|\nabla_{x_j}( \tilde{\Upsilon}(\hat s_t) \bar\partial_{A_{t}}\partial_{A_{t}} \hat s_t)\|_w\nonumber\\
 && \qquad\qquad\leq  C ( \|\nabla_{x_j}\hat s_t\|_w \|\nabla_{A_{t}}^2\hat s_t\|_{C^0}+\|\nabla_{x_j}\bar\partial_{A_{t}}\partial_{A_{t}} \hat s_t \|_w\| \hat s_t\|_{C^0}) \nonumber \\
 && \qquad\qquad \leq  C  \|\nabla_{x_j}\hat s_t\|_w (\|\nabla_{A_{0,t}}^2\hat s_t\|_{C^0}+\|A_t-A_{0,t}\|_{C^1}\|\hat s_t\|_{C^1}) \nonumber\\
 & &  \qquad\qquad \phantom{\leq}+C (\|\Delta_{A_t}\nabla_{x_j} \hat s_t \|_w +   1+t\Theta)\| \hat s_t\|_{C^0} \nonumber \\
 &&  \qquad\qquad\leq  C (t\|\Delta_{A_t}\nabla_{x_j} \hat s_t \|_w+t^{\frac{1}{2} }\|\nabla_{x_j}\hat s_t\|_w +t+t^2\Theta) \nonumber \\
 && \qquad\qquad\leq  C (t^{\frac{1}{2}}\|\Delta_{A_t}\nabla_{x_j} \hat s_t \|_w+ t+ t\Lambda +t^2\Theta),   \nonumber \eea
  where we used Lemma \ref{prop2+00}. For the later terms, we have
  \bea & & \|\nabla_{x_j}(\bar\partial_{A_{t}} \tilde{\Upsilon}(\hat s_t) \wedge \partial_{A_{t}}\hat s_t)\|_w +\|\nabla_{x_j}(\Upsilon(\hat s_t)\partial_{A_{t}}\hat s_t \wedge \Upsilon(-\hat s_t)\bar\partial_{A_{t}}\hat s_t)\|_w  \nonumber \\  && \qquad \leq  C  (\|\nabla_{A_{0,t}}\hat s_t\|_{C^0}+\|A_t-A_{0,t}\|_{C^0}\|\hat s_t\|_{C^0})(\|\nabla_{x_j}\hat s_t\|_w \nonumber\\
  &&\qquad\qquad\qquad\qquad\qquad\qquad\qquad\qquad\qquad+\|d_{A_{t}}\nabla_{x_j}\hat s_t\|_w+\|\hat s_t\|_w) \nonumber\\
   && \qquad\leq  C (t^2+t\|\Delta_{A_{t}}\nabla_{x_j} \hat s_t\|_w+t^2\Lambda).    \nonumber
   \eea

Returning to \eqref{thing7}, we put everything together to see $$\|\nabla_{x_j}F_{A_t}-i  d_{A_{t}}\star_w d_{A_{t}}\nabla_{x_j} \hat s_t\|_w \leq C (t^{\frac{1}{2}}\|\Delta_{A_{t}}\nabla_{x_j}\hat s_t\|_{w}+t\Lambda +t\Theta +t),  $$
$$
\|\Delta_{A_{t}}\nabla_{x_j}\hat s_t\|_{w}
\leq  C (\Theta+t+t\Lambda),$$ and  $$\|d_{A_{t}}\nabla_{x_j}\hat s_t\|_{w}\leq C (\Theta+t+t\Lambda). $$
Thus we conclude
$$\Lambda^2 \leq C (\Theta^2+t^2),$$
proving the lemma.
\end{proof}

Now, we are ready to prove Proposition \ref{prop2+0}.

\begin{proof}[Proof of Proposition \ref{prop2+0}]
Note that  we have
\be
\|F_{ \Xi_t}\|^2_{L^2(M_{ U'},\omega_{ t})}\leq 2 \int_{M_{ U'}}( t^{-1}|F_{A_t}|^2_{\omega^{SF}}+\sum_{j=1,2}| \kappa_{ t,j}|^2_{\omega^{SF}}+ t|F_{ B,t}|^2_{\omega^{SF}})(\omega^{SF})^2.\nonumber
\ee
By (\ref{ASD2}), we have
\be
 t|F_{ B,t}|^2_{\omega^{SF}}\leq C  ( t^{-1}|F_{A_t}|^2_{\omega^{SF}}+t \sum_{j=1,2}| \kappa_{ t,j}|^2_{\omega^{SF}}),\nonumber
\ee
which in turn implies
\bea
\|F_{ \Xi_t}\|^2_{L^2(M_{ U'}, \omega_{ t})}&\leq& C \int_{ U'}(t^{-1}\|F_{A_t}\|^2_{w}+\sum_{j=1,2}\|\kappa_{ t,j}\|^2_{w})d x_1 d x_2\nonumber\\
&\leq& C ( t + \sum_{j=1,2}\int_{ U'} \|\kappa_{ t,j}-\kappa_{ t,j}^0\|^2_{w} d x_1 d x_2 \nonumber \\ & & + \sum_{j=1,2} \int_{U'} \|\partial_{x_j}A_{0,t}\|_w^2 d x_1 d x_2) \nonumber\\ &\leq& C  ( t +  \sum_{j=1,2} \int_{U'} \|\partial_{x_j}A_{0,t}\|_w^2 d x_1 d x_2)\nonumber
\eea
For the second inequality above we used $\|F_{A_t}\|^2_{w}\leq C t^2$ and $ \kappa_{ t,j}^0 =\partial_{x_j}A_{0,t} $.

Finally,   \bea
\|F_{ \Xi_t}\|^2_{L^2(M_{ U'}, \omega_{ t})}&\geq & \frac{1}{2}\int_{ U'}\sum_{j=1,2}\|\kappa_{ t,j}\|^2_{w}d x_1 d x_2\nonumber\\
&\geq &  \frac{1}{2} ( \sum_{j=1,2} \int_{U'} \|\partial_{x_j}A_{0,t}\|_w^2 d x_1 d x_2 \nonumber \\ & & - \sum_{j=1,2}\int_{ U'} \|\kappa_{ t,j}-\kappa_{ t,j}^0\|^2_{w} d x_1 d x_2 ) \nonumber\\ &\geq & C  (  \sum_{j=1,2} \int_{U'} \|\partial_{x_j}A_{0,t}\|_w^2 d x_1 d x_2-t),\nonumber
\eea and we obtain the conclusion.
\end{proof}

We finish this section by a lemma that  is needed in the proof of Theorem \ref{thm-main2}.

\begin{lem}\label{lem9.1}    $$\sum_{j=1,2}\|[\kappa_{t,j}, \kappa_{t,j}]\|_{L^{2}(M_{U'},\omega^{SF})}^2 \leq C_7   t,$$ for a constant $C_7>0$.
 \end{lem}

\begin{proof} Recall that
$ \kappa_{t,j}^0=\partial_{x_j} A_{0,t}$ by (\ref{curv+++}), and thus  $$ [\kappa_{t,j}^0,\kappa_{t,j}^0]=0,  \  \  \  j=1,2.$$
 We have $$[\kappa_{t,j},\kappa_{t,j}]=2 [\kappa_{t,j}^0,\kappa_{t,j}-\kappa_{t,j}^0]+[\kappa_{t,j}-\kappa_{t,j}^0,\kappa_{t,j}-\kappa_{t,j}^0],  $$ and  by $|\kappa_{t,j}|\leq C$,  $$|[\kappa_{t,j},\kappa_{t,j}]|\leq C | \kappa_{t,j}-\kappa_{t,j}^0|.$$
 Lemma \ref{prop1} shows that
$$\int_{U}\sum_{j=1,2}\|[\kappa_{t,j},\kappa_{t,j}]\|_w^2  dx_1dx_2 \leq C \int_{U}\sum_{j=1,2}\|\kappa_{t,j}-\kappa_{t,j}^0\|_w^2  dx_1dx_2\leq C t.$$ We obtain the conclusion.
\end{proof}

\section{Proof of Proposition \ref{type3bubbling}}
\label{lowerbounds}

Now, we have the tools to verify assumption \eqref{hyp:curv} along our main subsequence of times $t_k$, which is chosen in Proposition \ref{estimate 1}.
\begin{proof}[Proof of Proposition \ref{type3bubbling}]
We work via contradiction, and assume the Proposition is false, in other words assumption \eqref{hyp:curv} fails for our sequence $\Xi_{t_k}$. By passing to a subsequence, there exists a sequence  of  points $p_k'\in M_K$ so that
\be
\label{contraassump}
t_k^{\frac{1}{2}}|F_{\Xi_{t_k}}|_{\omega_{t_k}}(p_k')\rightarrow\infty,
\ee
 and $f(p_k')$ converges to a point $x\in K$, as $t_k\rightarrow 0$.

 Applying Lemma \ref{pointpick}, we can pick new points near  $p_k$ to carry out our argument. Specifically, if $r=\frac{1}{2}{\rm dist}_{\omega}(x, N\backslash K)$,  there exists a sequence of real numbers $0<\rho_{k}<r$ and a sequence $p_k \in M$  so that $d_{\omega_{t_k}}(p_k,p_k')\leq r$,
 \be
 \sup_{B_{\omega_{t_k}}(p_k, \rho_{k})}
 |F_{\Xi_{t_k}}|_{\omega_{t_k}}\leq 2|F_{\Xi_{t_k}}|_{\omega_{t_k}}(p_k),\nonumber
 \ee
 and
 $$ 2\rho_{k}|F_{\Xi_{t_k}}|_{\omega_{t_k}}(p_k)\geq r |F_{\Xi_{t_k}}|_{\omega_{t_k}}(p_k').$$
If we set $\delta_{k}:=t_k^{-\frac{1}{2}}|F_{\Xi_{t_k}}|_{\omega_{t_k}}^{-1}(p_k)$, then \eqref{contraassump} and the above inequalities give $\delta_{k} \rightarrow 0$,  and
 $$\rho_{k} \delta_{k}^{-1}\geq r t_k^{\frac{1}{2}}|F_{\Xi_{t_k}}|_{\omega_{t_k}}(p_k') \rightarrow\infty.$$
 Furthermore, define
\be
\ti t_k:= t_k \delta_{k}^{-2}={t_k}^2|F_{\Xi_{t_k}}|_{\omega_{t_k}}^{2}(p_k)\leq \epsilon_k^2\rightarrow 0,\nonumber
\ee
which goes to zero as $t_k\rightarrow0$ by Proposition  \ref{estimate 1}.

We now consider the scaled matric $\tilde{\omega}_{\ti t_k}= \delta_{k}^{-2}\omega_{t_k}$, and claim that $\tilde{\omega}_{\ti t_k}$ satisfies the same collapsing properties of $\omega_{t_k}$.  If $\tilde{w} = \delta_{k}^{-1} w$ denotes the scaled coordinate   on $D_r=\{|w|<r\delta_{k}\}=\{|\tilde{w}|<r\}$, where $f(p_k)$ is given by $w=0$,  then  \be
\delta_{k}^{-2}\omega_{t_k}^{SF}=\frac {i}2\left(\tilde{t}_k W(dz+\tilde{b}d\ti w)\wedge\overline{(dz+ \tilde{b}d\ti w)}+W^{-1} d\ti w\wedge d\bar {\ti w}\right),\nonumber
\ee where $\ti b= -\frac{{\rm Im}(z)}{{\rm Im}(\tau)}\frac{\partial\tau}{\partial \ti w}$.
 For a certain  fiberwise  translation $T_{\sigma_0}$, we write
\bea
 T_{\sigma_0}^*\delta_{k}^{-2} \omega_{t_k}- \delta_{k}^{-2}\omega^{SF}_{t_k} & = &  \delta_{k}^{-2}\varphi_{{t_k}, z \bar{z}}dz\wedge d\bar{z}+ \varphi_{{t_k}, w \bar{w}}d\tilde{w}\wedge d\bar{\tilde{w}}\nonumber \\ & & +\delta_{k}^{-1} \varphi_{{t_k}, w \bar{z}}d\tilde{w}\wedge d\bar{z}+ \delta_{k}^{-1} \varphi_{{t_k}, z \bar{w}}dz\wedge d\bar{\tilde{w}}.\nonumber
\eea
By   Lemma  \ref{lem-decay},
  for $\nu\gg  1$,   $$\|\delta_{k}^{-2}\varphi_{t_k, z \bar{z}}\|_{C_{\rm loc}^\ell} +   \|\delta_{k}^{-1} \varphi_{t_k, z \bar{w}}\|_{C_{\rm loc}^\ell}+ \|\delta_{k}^{-1} \varphi_{t_k, w \bar{z}}\|_{C_{\rm loc}^\ell}\leq   C_\ell \tilde{t}_k^\nu,$$
  and
   $$\|\frac{\partial}{\partial z}\varphi_{t_k, w \bar{w}}\|_{C_{\rm loc}^\ell}+\|\frac{\partial}{\partial \bar{z}}\varphi_{t_k, w \bar{w}}\|_{C_{\rm loc}^\ell}\leq C_{\ell} \tilde{t}_k^{\nu}, \ \ \,\,\, \  \  \|\varphi_{t_k, w \bar{w}}-\chi_{t_k,w\bar{w}}\|_{C_{\rm loc}^0}\leq C_0   \tilde{t}_k^{\nu}. $$
   Here we used $t_k \leq \tilde{t}_k$, and that $\chi_{t_k, w \bar{w}}$ is a function on $D_r$ that satisfies    $\chi_{t_k, w \bar{w}} \rightarrow 0$ in the $C^\infty$-sense as $t_k\rightarrow 0$. The $C_{\rm loc}^\ell$-norms are calculated in coordinates $z$ and $ \tilde{w}$.

Working in the scaled metrics, we have that $d_{\omega_{\ti t_k}}(p_k, p)\leq \rho_{k}\delta_{k}^{-1}$ for any $p\in B_{\omega_{t_k}}(p_k, \rho_{k})$, so the radius of the disk approaches infinity. In particular this implies that on $ B_{\ti \omega_{ \ti t_k}}(p_k,\rho_{k}\delta_{k}^{-1})$, we have the bound
\be
 |F_{ \Xi_{t_k}}|_{\ti \omega_{\ti t_k}}= \delta_{k}^{2}|F_{\Xi_{t_k}}|_{\omega_{t_k}}  \leq 2 \delta_{k}^{2} |F_{\Xi_{t_k}}|_{\omega_{t_k}}(p_k)=    2 t_k^{-1} |F_{\Xi_{t_k}}|_{\omega_{t_k}}^{-1}(p_k)=  2\ti t_k^{-\frac12}.\nonumber
\ee
Now, because the energy $\E_{t_k}(p, R_{t_k}(p_k))$ is scale invariant,
\bea
\varepsilon &= & \E_{t_k} (p_k, R_{t_k}(p_k))\nonumber\\ &=&
\frac{\delta_{k}^{-4}R_{t_k}(p_k)^{4}}{{\rm Vol}(B_{\ti \omega_{ \ti t_k}}(p_k,\delta^{-1}_{t_k}R_{t_k}(p_k)))}\int_{B_{\ti \omega_{\ti t_k}}(p_k,\delta^{-1}_{t_k}R_{t_k}(p_k))}|F_{ \Xi_{t_k}}|^2_{\ti \omega_{\ti t_k}}\ti\omega_{\ti t_k}^2.   \nonumber
\eea
Additionally, note that
$$
\delta_{k}^{-1}R_{t_k}(p_k) = {t_k}^{\frac12}|F_{\Xi_{t_k}}|_{\omega_{t_k}}(p_k) R_{t_k}(p_k) \leq 4 t_k^{\frac12} |F_{\Xi_{t_k}}|_{\omega_{t_k}}^{\frac{1}{2}}(p_k) =4 \ti t_k^{\frac14},\nonumber
$$ since  $ |F_{\Xi_{t_k}}|_{\omega_{t_k}}(p_k)\leq 4 R_{t_k}^{-2}(p_k)$ by (\ref{curvradconclusion}).  Thus, on $B_{\ti \omega_{ \ti t_k}}(p_k,\rho_{k}\delta_{k}^{-1})$ we have  \be
\label{thing10}
 |F_{ \Xi_{t_k}}|_{\ti \omega_{\ti t_k}}\leq   2\ti t_k^{-\frac12}
 \ee
and
\be
\label{thing9}
\varepsilon \leq\frac{4^4 \ti t_k}{{\rm Vol}(B_{\ti \omega_{ \ti t_k}}(p_k,4 \ti t_k^{\frac14}))}\int_{B_{\ti \omega_{\ti t_k}}(p_k,4 \ti t_k^{\frac14})}|F_{ \Xi_{t_k}}|^2_{\ti \omega_{\ti t_k}}\ti\omega_{\ti t_k}^2.
\ee
Inequality  \eqref{thing10} gives assumption \eqref{hyp:curv} for our connections in scaled coordinates (with scaled parameter $\ti t$).   Also \eqref{hyp:poincare} is also satisfied  since the scaling does not effect the fiber direction. Thus  Proposition \ref{prop2} holds in scaled coordinates, which in turn allows us to conclude Proposition \ref{prop2+0} as well.

 To achieve our contradiction, we show these bounds force the energy on the right hand side of \eqref{thing9} to go to zero. We continue to use the notation $\|\cdot\|_{w}:=\|\cdot\|_{L^2(M_{w},\ti\omega^{SF})}$ since scaling does not affect the fiber direction.

Applying Proposition \ref{prop2+0},  on any $K\subset D_r$ we have
 $$  \|F_{\Xi_{t_k}}\|_{L^2(M_{K}, \ti \omega_{\ti t_k})}^2  \leq C(\ti t_k +\int_{K}\sum_{j=1,2}\|\partial_{\tilde{x}_j}A_{0,t_k}\|_w^2  d\tilde{x}_1d\tilde{x}_2) $$
for a uniform constant $C$,  where $\tilde{x}_1+i\tilde{x}_2=\tilde{w}$.  Since $A_{0,t_k} \rightarrow A_{0}$ in the $C^\infty$-sense on $M_U$, we have
 $$\|\partial_{\tilde{x}_j}A_{0,t_k}\|_w^2=\delta_k^2\|\partial_{x_j}A_{0,t_k}\|_w^2\leq C\delta_k^2,  $$
and thus
 $$  \|F_{\Xi_{t_k}}\|_{L^2(M_{K}, \ti \omega_{\ti t_k})}^2  \leq C(\ti t_k +\delta_k^2 \int_K d\tilde{x}_1d\tilde{x}_2 ). $$
Because the radius $\ti t_k^{\frac14}$ grows slower than the injectivity radius of the elliptic fibers in the metric $\ti \omega_{\ti t_k}$ (which is roughly  $\ti t_k^{\frac12}$), we see that for $\ti t_k$ small enough
\be
\frac{\ti t_k}{{\rm Vol}(B_{\ti \omega_{ \ti t_k}}(p,4\ti t_k^{\frac14}))}\leq \frac{C \ti t_k }{\ti t_k \ti t_k^{\frac12}}=\frac{C }{ \ti t_k^{\frac12}}. \nonumber
\ee
Also $B_{\ti \omega_{ \ti t_k}}(p_k, 4\ti t_k^{\frac14})\subset M_{D_r}$. Thus, returning to \eqref{thing9}, we have
\bea
   \varepsilon & \leq & \frac{4^4 \ti t_k}{{\rm Vol}(B_{\ti \omega_{ \ti t_k}}(p_k,4\ti t_k^{\frac{1}{4}}))}\int_{B_{\ti \omega_{\ti t_k}}(p_k,4\ti t_k^{\frac{1}{4}})}|F_{ \Xi_{t_k}}|^2_{\ti \omega_{\ti t_k}}\ti\omega_{\ti t_k}^2\nonumber \\
&\leq& \frac{C}{ \ti t_k^{\frac12}}(\ti t_k +\delta_k^2 \ti t_k^{\frac12} )\nonumber \\ &\leq &  C(\ti t_k^{\frac12}+\delta_k^2).\nonumber
\eea
The right hand side above goes to zero, a contradiction.
\end{proof}

\section{The proof of Theorem \ref{thm-main2}}

At last, we prove Theorem \ref{thm-main2} in this section.  Under the same setup as in Section 6,  the first lemma shows that for any fixed $p\geq 2$,  $$\| F_{B,t}\|_{L^p(M_U, \omega^{SF})}\rightarrow 0$$ when $t \rightarrow 0$.

\begin{lem}\label{prop9.2}   If (\ref{hyp:curv}) and (\ref{hyp:poincare}) hold for $t\ll 1$, for any $p\geq 2$,   we have
the following inequalities  $$ \| F_{A_t}\|_{L^p(M_U, \omega^{SF})}^p\leq C_1t^{1+p}, \   \   and  \  \   \| F_{B,t}\|_{L^p(M_U, \omega^{SF})}^p\leq  C_1t^{1+\frac{1}{p}} , $$ where    the constant $C_1$  is   independent of $t$.
 \end{lem}

 \begin{proof}  By Lemma \ref{le2},
  $$ \Delta \| F_{A_t}\|_w^2  \geq  \frac{\delta}{t} \|  F_{A_t}  \|_w^2 - CtZ_t, $$
   where $$Z_t=\sum_{j=1,2}\|[\kappa_{t,j}, \kappa_{t,j}]\|_w^2 +t^\nu,$$ for $\nu\gg 1$ and a constant $C>0$.    Lemma \ref{lem9.1} implies that $$\int_U  Z_tdx_1dx_2 \leq Ct.$$

   Let $\eta$ be  a smooth function such that $0\leq \eta \leq 1$ and ${\rm supp} (\eta) \subset U$. Then
  \begin{eqnarray*} \int_U  \|   F_{A_t}  \|_w^2 dx_1dx_2 & \leq &  t \delta^{-1} \int_U  \|   F_{A_t}  \|_w^2  \Delta \eta dx_1dx_2 +t^2 C  \int_U \eta Z_tdx_1dx_2\\ & \leq &  t\tilde{C}\delta^{-1}   \int_U  \|   F_{A_t}  \|_w^2 dx_1dx_2+ t^2 C  \int_U  Z_tdx_1dx_2,  \end{eqnarray*} for a constant $\tilde{C}\geq \sup\limits_{U} \Delta \eta$. Thus for $t\ll 1$, $$ \int_U  \|   F_{A_t}  \|_w^2 dx_1dx_2 \leq Ct^{3}.  $$
    For any $p\geq 2$,  $$ \| F_{A_t}\|_{L^p(M_U, \omega^{SF})}^p\leq C t^{p-2} \int_U  \|   F_{A_t}  \|_w^2 dx_1dx_2 \leq Ct^{p+1},  $$ by Lemma \ref{le-l2}, and  $$ \| F_{B,t}\|_{L^p(M_U, \omega^{SF})}^p\leq  Ct^{1+\frac{1}{p}},  $$ by (\ref{ASD2}).
\end{proof}

 Recall that for any sequence $t_k \rightarrow 0$, a subsequence of  $\Xi_{t_k}$ $L^p_1\cap C^{0,\alpha}$-converges to a $L^p_1\cap C^{0,\alpha}$-connection $\Xi_0$ by preforming  certain further   unitary gauge changes   if necessary on $M_K$
 in Theorem \ref{thm-main}, where $K\subset N^o$. Thus the curvature $F_{\Xi_{t_k}}$ $L^p$-converges to $F_{\Xi_{0}}$ on $M_K$.

  On any open disc $U\subset K$,  we have the decompositions  $$\Xi_{0}=\tilde{A}_0+\tilde{B}_{0,1}dx_1+\tilde{B}_{0,2}dx_2, \  \  \  {\rm and}$$  $$F_{\Xi_{0}} =F_{\tilde{A}_0}-\tilde{\kappa}_{0,1}\wedge dx_1-\tilde{\kappa}_{0,2}\wedge dx_2-F_{\tilde{B},0} dx_1\wedge dx_2,$$ where $\tilde{\kappa}_{0,j}=\frac{\partial}{\partial x_j} \tilde{A}_0-d_{\tilde{A}_0}\tilde{B}_{0,j} $.
By Lemma \ref{prop9.2} and the convergence,  we obtain that $$F_{\tilde{A}_0}\equiv 0, \  \   F_{\tilde{B},0} \equiv 0,  \  \  {\rm and}  \  \  \star_w \tilde{\kappa}_{0,1}=\tilde{\kappa}_{0,2}.$$ Thus $\Xi_{0}$ is an anti-self-dual connection with respect to $(\omega^{SF}, \Omega)$, i.e. $$ F_{\Xi_{0}}\wedge \omega^{SF} =0, \  \  {\rm and} \  \  F_{\Xi_{0}}\wedge \Omega =0.$$  It is standard (cf. Theorem 9.4 of  \cite{Weh2}) that by preforming  a further unitary gauge change  if necessary, we can have  that $\Xi_{0} $ is smooth.

\begin{lem}\label{prop9.3}  There is a unitary gauge $u$ such that $$u(\Xi_{0})=A_0$$ on $M_U$, where  $A_0$ is given by (\ref{bconnections2}).
 \end{lem}

 \begin{proof} By Theorem \ref{thm-main}, for any $w\in U$, there is a unitary gauge $u_w$ on $M_w$ such that $u_w(\Xi_{0}|_{M_w})=A_0|_{M_w}$, and $u_w$ is smooth since both $\Xi_{0}|_{M_w}$ and $A_{0}|_{M_w}$ are smooth. We claim that one can choose $u_w$ depending on  $w$ smoothly.

Note that $M_w \cong T^2$ and $P|_{M_w}\cong M_w\times SU(n)$.  Let $\mathcal{A}^{\ell, p}$ be the space of $L^{ p}_{\ell}$ $SU(n)$-connections on the trivial bundle on $T^2$, $\ell \geq 1 $,  and $\mathcal{G}^{\ell+1, p}$ be the $L^{ p}_{\ell+1}$ unitary gauge  group. We have  identifications  $\mathcal{A}^{\ell, p}= L^{ p}_{\ell}(T^2, \mathfrak{sl}(n))$ and $\mathcal{G}^{\ell+1, p}=L^{ p}_{\ell+1}(T^2, SU(n))$ under the trivialization, and $\mathcal{G}^{\ell+1, p}$ acts on $\mathcal{A}^{\ell, p}$ by $u(A)=u^{-1}Au+u^{-1}du$. If we denote the orbit $O_w=\{u(\Xi_{0}|_{M_w})| u\in \mathcal{G}^{\ell+1, p}\}\subset \mathcal{A}^{\ell, p}$ for any $w\in U$, then $ A_{0}|_{M_w} \in O_w$. Define the orbit  map $$\Psi: \mathcal{G}^{\ell+1, p}\times U \rightarrow \bigcup_{w\in U}O_w\subset \mathcal{A}^{\ell, p},  \  \  {\rm by} \  \  \Psi(u,w)=u(\Xi_{0}|_{M_w}). $$

 For a fixed $w_0\in U$, let $\varrho_w: O_w\rightarrow O_{w_0}$ by $A \mapsto v(A_{0}|_{M_{w_0}})$, where  $v(A_{0}|_{M_w})=A$ for a unitary gauge $v$.  If $v'$ is an  another unitary gauge such that $v'(A_{0}|_{M_w})=A$, then $v'v^{-1}(A_{0}|_{M_w})=A_{0}|_{M_w}$, and thus $v'v^{-1}\in T^{n-1}\subset SU(n)$, i.e. a diagonal matrix. Since $A_{0}|_{M_{w_0}}$ is a diagonal matrix valued 1-form, we have $v (A_{0}|_{M_{w_0}})=v' (A_{0}|_{M_{w_0}})$, and
 $\varrho_w$ is well-defined.

    Let  $\Psi'=\varrho_w\circ \Psi:  \mathcal{G}^{\ell+1, p}\times U \rightarrow O_{w_0}$ be the the composition.
 Note that the tangent space $T_{A_{0}|_{M_{w_0}}}O_{w_0}={\rm Im}(d_{A_{0}|_{M_{w_0}}})$, and the first partial  derivative of  $\Psi'$  at $(u,w)$ such that $\Psi'(u,w)= A_{0}|_{M_{w_0}}$ is $D_1 \Psi'=-d_{A_{0}|_{M_{w_0}}}$. Thus $A_{0}|_{M_{w_0}}$ is a regular value of $\Psi'$, and $\Psi'^{-1}(A_{0}|_{M_{w_0}})$ is a smooth submanifold.  Furthermore, the projection $\mathcal{G}^{\ell+1, p}\times U\rightarrow U$ induces a $T^{n-1}$-bundle structure on $\Psi'^{-1}(A_{0}|_{M_{w_0}})$ with fiber $T^{n-1}\subset SU(n)$.

 If $\tilde{u}: U\rightarrow \Psi'^{-1}(A_{0}|_{M_{w_0}})$ is a smooth section, then $\tilde{u}(w)(\Xi_{0}|_{M_w})=A_{0}|_{M_w}$, and we can regard $\tilde{u}$ as a smooth  unitary gauge change on $M_U$.  Therefore we have
 $$\tilde{u}(\Xi_{0})=A_{0}+ B_{0,1}dx_1+B_{0,2}dx_2, $$ which still satisfies $$\star_w \kappa_{0,1}=\kappa_{0,2},  \  \  {\rm with} \ \ \kappa_{0,j}=\frac{\partial}{\partial x_j} A_0-d_{A_0}B_{0,j},  \  \ j=1,2, \   \   {\rm and} $$ $$0=F_{B,0}=\frac{\partial}{\partial x_2} B_{0,1}-\frac{\partial}{\partial x_1} B_{0,2}-[B_{0,1},B_{0,2}].$$

 Note that $\frac{\partial}{\partial x_j} A_0 \in \ker \Delta_{A_0}$,  $j=1,2$,  on any $M_w$, and $$\star_w \frac{\partial}{\partial x_1} A_0-\frac{\partial}{\partial x_2} A_0= \star_w d_{A_0}B_{0,1}-d_{A_0}B_{0,2}  . $$   By the Hodge decomposition, $\ker \Delta_{A_0}$, ${\rm Im} ( d_{A_0}^*)$ and ${\rm Im} (d_{A_0})$ are orthogonal to each other. Thus  $$d_{A_0}B_{0,j} \equiv 0, \  \  j=1,2,$$ on any $M_w$, and $B_{0,j}|_{M_w}$ is a diagonal matrix in $\mathfrak{sl}(n) $. If we write  $B_{0,j}=i{\rm diag}\{b_{j,1}, \cdots, b_{j,n}\}$, then $\frac{\partial}{\partial x_2} B_{0,1}=\frac{\partial}{\partial x_1} B_{0,2}$ implies that there are real functions $ \vartheta_{\ell}$ on $U$ such that $b_{1,\ell}dx_1+b_{2,\ell}dx_2=- d \vartheta_{\ell}$,  $\ell=1, \cdots,n$.  If $\tilde{v}={\rm diag}\{\exp(i\vartheta_{1}), \cdots, \exp(i\vartheta_{n})\}$, and we regard $\tilde{v}$ as a unitary gauge change on $M_U$,  then  $$ \tilde{v}(\tilde{u}(\Xi_{0}))=A_{0}.$$ We obtain the conclusion by letting $u=\tilde{v}\cdot\tilde{u}$.
\end{proof}

\begin{proof}[Proof of Theorem \ref{thm-main2}]  Let $\{U_\lambda| \lambda\in\Lambda\}$ be an  open cover of $N^o$ such that any intersection $U_{\lambda_1}\cap \cdots \cap U_{\lambda_h}$ is contractible.  For any $U_\lambda$, $D_0^o \cap M_{U_\lambda}=U_\lambda^1 \cup \cdots \cup  U_\lambda^n$ is a disjoint union of open sets biholomorphic  to $U_\lambda$, and $\{U_\lambda^j|\lambda\in\Lambda, j=1, \cdots, n \}$ is an open cover of $D_0^o \cap M_{N^o}$ such that any intersections are contractible.

 On any $M_{U_\lambda}$, there is a unitary gauge $u_\lambda$ such that $u_\lambda(\Xi_0)=A_0$ by Lemma \ref{prop9.3}. Recall that $$A_0={\rm diag}\{\alpha_1, \cdots, \alpha_n\}, \  \    \  \alpha_j=\pi({\rm Im}( \tau))^{-1}(q_j \bar{\theta}-\bar{q}_j\theta),$$ where $\{(w, q_j(w))\}=U_\lambda^j$ is one component of $D_0^o\cap M_{U_\lambda}$, and  $\alpha_j$ is not unitary  gauge equivalent to $\alpha_i$ if $j\neq i$.
On any intersection $M_{U_\lambda \cap U_\mu}$, $A_0=u_\mu \cdot u_\lambda^{-1}(A_0)$. Thus $ u_\mu \cdot u_\lambda^{-1} |_{M_w} \in T^{n-1}\subset SU(n)$ for any $w\in U_\lambda \cap U_\mu$. We can  write $u_\mu \cdot u_\lambda^{-1}= {\rm diag}\{g_{\mu\lambda}^{1j_1}, \cdots, g_{\mu\lambda}^{nj_n}\}$, where $g_{\mu\lambda}^{ij_i}$ is a $U(1)$-valued function on $U_\lambda \cap U_\mu$, and  is the unitary gauge change  between $\alpha_i$ on $M_{U_\mu}$ and  $\alpha_{j_i}$ on $M_{U_\lambda}$.   Hence we have that  $U_\mu^i \cap U_\lambda^{j_i} \neq\emptyset$, and  $d \log g_{\mu\lambda}^{ij_i} =0$, which implies that  $g_{\mu\lambda}^{ij_i}$, $i=1, \cdots, n$, are $U(1)$-valued constant functions on $U_\lambda \cap U_\mu$. By regarding $g_{\mu\lambda}^{ij_i}$ as a function on $U_\mu^i \cap U_\lambda^{j_i}$,   we obtain a 1-chain $\{(U_\mu^i \cap U_\lambda^{j_i}, g_{\mu\lambda}^{ij_i})\} \in \mathcal{C}^1(\{U_\lambda^j\}, \mathcal{U}_c(1))$  for the   $U(1)$-valued  locally constant sheaf $\mathcal{U}_c(1)$ on $D_0^o \cap M_{N^o}$.

If $U_\mu^i \cap U_\lambda^{j}\cap U_\nu^{k}\neq\emptyset$,  then $U_\mu \cap U_\lambda \cap U_\nu \neq\emptyset$,  and by $u_\mu \cdot u_\lambda^{-1} \cdot u_\lambda \cdot u_\nu^{-1}\cdot u_\nu \cdot u_\mu^{-1}={\rm Id}$,   we obtain  that  $g_{\mu\lambda}^{ij}g_{\lambda\nu}^{jk}g_{\nu\mu}^{ki}=1$. Therefore $\{(U_\mu^i \cap U_\lambda^{j_i}, g_{\mu\lambda}^{ij_i})\}$ satisfies the cocycle condition, and
 defines a cohomological class $\Theta=[\{(U_\mu^i \cap U_\lambda^{j_i}, g_{\mu\lambda}^{ij_i})\}]\in H^1(D_0^o \cap M_{N^o}, \mathcal{U}_c(1))$, which is equivalent to a flat $U(1)$-connection on $D_0^o \cap M_{N^o}$.  From the construction in Subsection 2.6, it is clear that $\Xi_0\in \mathcal{FM}(D_0^o \cap M_{N^o}, \Theta)$. \end{proof}

\appendix
 \section{Collapsing rate  of Ricci-flat   K\"{a}hler-Einstein metrics}
\label{app}

Here we study the collapsing rate of Ricci-flat   K\"{a}hler-Einstein metrics on general Calabi-Yau manifolds, which is used in the proof of the main theorem.

 Let  $M$ be  a Calabi-Yau $m$-manifold, i.e. $M$ is  projective with trivial canonical bundle $\mathcal{K}_{M}\cong \mathcal{O}_M$.  Assume  $M$
  admits a holomorphic fibration  $
f:M\rightarrow N$, where $N$ is smooth projective manifold with  $n=\dim_{\mathbb{C}}N < m$. As above, let  $S_N$ denotes   the discriminant locus $f$, and  $N_0=N\backslash S_N$ the regular locus. For any $w\in N_0$, the smooth  fiber  $M_w=f^{-1}(w)$ is a Calabi-Yau manifold of dimension $m-n$.
Let  $\alpha$ be an ample class on $M$, and $\alpha_0$ an ample class on $N$. Then for $t\in[0,1)$,   $\alpha_{t} =t \alpha + f^*\alpha_0$  is a family of K\"{a}hler classes. Denote by $\omega_t \in \alpha_{t} $  the unique Ricci-flat K\"{a}hler-Einstein metric, which satisfies the complex Monge-Amp\`ere equation
$$ \omega_t^m=c_t t^{m-n} (-1)^{\frac{m^2}{2}}  \Omega \wedge \overline{\Omega}.$$
Here $\Omega$ is a holomorphic volume  form on $M$, and $c_t $ has  a positive limit  when $t\rightarrow 0$.

    The behavior of $\omega_t$ when $t\rightarrow 0$ has been studied intensively in the  literature (see cf. \cite{GW,To1,GTZ,GTZ2,HT,TWY,TZ0,TZ,HT1}, among others). We briefly recall some of the important developments, and refer the readers to the above sources for details.  Under the assumption that  $M$ is an elliptically fibered K3 surface with only singular fibers of Kodaira type $I_1$, Gross-Wilson first proved that $(M, \omega_t)$ converges to a compact metric space homeomorphic  to the sphere $S^2$  \cite{GW}. In the case of general fibered Calabi-Yau manifolds, Tosatti proved that $\omega_t$ converges to $f^{*}\omega$ in the current sense \cite{To1}, where $\omega$ is the K\"ahler metric  on $N_{0}$ with $${\rm Ric}(\omega) =\omega_{WP}$$ obtained in \cite{To1,ST,ST1}, and  $\omega_{WP}$ is the  Weil-Petersson metric of the fibers  on $N_0$.

 If  $M$ is an Abelian fibered Calabi-Yau $m$-manifold, then Gross-Tosatti-Zhang improved the convergence of $\omega_{t}$ to $C^{\infty}$ away from the singular fibers  \cite{GTZ}. More precisely $\omega_{t}$  converges smoothly  to $f^{*}\omega$ on $f^{-1}(K)$ for any compact $K\subset N_{0}$ when $t\rightarrow 0$, and additionally the curvature of $\omega_t$ is locally uniformly bounded  on $f^{-1}(N_0)$.  The Gromov-Hausdorff convergence of $(M, \omega_{t})$ is obtained in \cite{GTZ2} for  the case of one dimensional base $N$, which generalizes the Gross-Wilson's result to any elliptically fibered K3  surface. In a recent paper of Tosatti-Zhang \cite{TZ}, the Gromov-Hausdorff convergence of $(M, \omega_{t})$ is generalized   to the case when  $M$ is a holomorphic symplectic manifold admitting a holomorphic Lagrangian fibration, and $\omega_{t}$ is a HyperK\"{a}hler metric.

However, despite  this later progress, one important property is still missing for the general cases of Calabi-Yau manifolds that appears in   the original work of Gross-Wilson. In their setting they show that $\omega_{t}$ approaches  a  semi-flat K\"{a}hler metric exponentially fast on compact subsets away from the singular fibers. This behavior is expected in general. In fact, motivated by physics, Gaiotto-Moore-Neitzke propose a construction of complete HyperK\"{a}hler metrics on certain compactifications of complex, completely  integrable systems, which asserts the exponential approximations by semi-flat K\"{a}hler metrics  \cite{GMN}. In particular, the  asymptotic behavior of HyperK\"{a}hler metrics on the Hitchin moduli spaces is studied in a recent paper \cite{MSWW}.

The goal of this appendix is to study the asymptotic rate  of  $\omega_{t}$ for any Abelian fibered Calabi-Yau manifolds. From now on assume any smooth  fiber $M_w$ is an Abelian variety.
For an open subset  $U\subset N_0$ biholomorphic to a polydisk, $f:M_U \rightarrow U$ is a family of Abelian varieties, which is isomorphic to
 $f:(U\times\mathbb{C}^{m-n})/\Lambda\to U,$ where $\Lambda\to U$ is a lattice bundle with fiber $\Lambda_w\cong \mathbb{Z}^{2m-2n}$, so that $M_w\cong \mathbb{C}^{m-n}/\Lambda_w.$  We denote the universal covering map
   $p:U\times\mathbb{C}^{m-n}\to M_U$, which satisfies that $f\circ p(w,z)=w$ for all $(w,z)\in U\times\mathbb{C}^{m-n}$.

 For completeness we recall the construction of the semi-flat K\"{a}hler metric on
  $M_U$ (cf. \cite{GSVY,GTZ}).   Note that the ample class $\alpha$
   gives an ample polarization of type $(d_1,\ldots,d_{m-n})$ of the fiber $M_w$, where $d_i \in\mathbb{N}$ and
 $d_1|d_2|\cdots|d_{m-n}$.
Then  $\Lambda_w$ is generated by
$d_1e_1,\ldots,d_{n-m}e_{m-n}, Z_1,\ldots,Z_{m-n}\in \mathbb{C}^{m-n}$, where
$e_1,\ldots,e_{m-n}$ denotes the standard basis for $\mathbb{C}^{m-n}$, and  the matrix $Z=[Z_1,\ldots,Z_{m-n}]$ is the period matrix of $M_w$, which satisfies the Riemann relationship
 $$Z=Z^t,  \  \   {\rm  and}  \  \ {\rm Im} Z>0. $$
  If $z_1, \cdots, z_{m-n}$ denote the coordinates on $\mathbb{C}^{m-n}$, then
  on the fiber $M_w$, the flat K\"ahler form
  $$i\sum\limits_{kl}(\Im Z)_{kl}^{-1}dz_k\wedge d\bar z_l$$
 represents  $\alpha|_{M_w}$. Using the notation $W_{kl}=({\rm Im} Z)^{-1}_{kl},$ by Section 3 in \cite{GTZ}, if
$$
\eta(w,z)=-\frac{1}{2}\sum_{k,l=1}^{m-n}W_{kl}(w)(z_k-\bar z_k)(z_l-\bar
z_l) ,
$$ then  $i\partial\bar\partial \eta$ is
invariant under translation by  sections of  $\Lambda$, and therefore,  defines a semi-positive $(1,1)$-form on $M_U$.    The semi-flat metric is defined as
\be
\label{ASFmetric}
\omega_t^{SF}  = i t \partial\bar\partial \eta + f^* \omega,
\ee
for any $t\in(0,1]$, which satisfies that
 $\omega_t^{SF}|_{M_w}$ is the flat metric in the class $t\alpha|_{M_w}$. Again $\omega\in\alpha_0$ is the    K\"ahler metric  on $N$   whose Ricci curvature is the Weil-Petersson metric of fibers on the regular part.

The main result of the appendix is the following:

 \begin{thm}\label{ttm-decay}  For any $\nu \in\mathbb{N}$, there is a constant $C_{\nu}>0$ such that   $$ \|T_{\sigma_0}^*\omega_t- \omega^{SF}_t-f^*\chi_t\|_{C_{\rm loc}^0(M_U, \omega^{SF}_t)}\leq C_{\nu} t^{\frac{\nu}{2}}, $$ for a certain local section $\sigma_0$,    where $\chi_t$ is a $(1,1)$-form  on $U$ such that $\chi_t \rightarrow 0$ in the $C^\infty$-sense when $t\rightarrow 0$, and   $T_{\sigma_0}$ is the fiberwise  translation by $\sigma_0$.
 \end{thm}

 Note that $ \omega^{SF}_t+f^*\chi_t$ is still a semi-flat metric for $0< t\ll 1$. Thus this theorem asserts that as $t\rightarrow0$, $\omega_t $  approaches  a  semi-flat metric  faster than any polynomial rate. We remark that this decay rate is not as fast as the one demonstrated by Gross-Wilson \cite{GW},
  where $$ T_{\sigma_0}^*\omega_t= \omega^{SF}_t+f^*\chi_t+o(e^{-\frac{C'}{\sqrt{t}}})$$ is obtained.
  However a sufficiently high  polynomial decay  rate  is enough for the proof of the main theorem of the present paper. We leave the exponential rate for future study.

 \begin{proof}[Proof of Theorem \ref{ttm-decay}]
 By Proposition 3.1 in \cite{GTZ}, for any K\"{a}hler metric $\omega_M \in\alpha$,
  there is a holomorphic section $\sigma_0: U \rightarrow M_U$ such that $$\omega+t \omega_{M}=T_{-\sigma_0}^*\omega_t^{SF}+i\partial\overline{\partial} \xi_t.$$
Thus
$$T_{\sigma_0}^*\omega_t=\omega+t T_{\sigma_0}^* \omega_{M} + i\partial\overline{\partial}\phi_t\circ  T_{\sigma_0} =\omega_t^{SF}+i\partial\overline{\partial}\varphi_t,$$ where $\varphi_t= (\phi_t+\xi_t)\circ  T_{\sigma_0} $.
  If we denote $\lambda_t: U\times \mathbb{C}^{m-n} \rightarrow U\times \mathbb{C}^{m-n}$ the dilation given by $ \lambda_t(w,z)=(w, t^{-\frac{1}{2}}z)$, then $\lambda_t^*i t \partial\bar\partial \eta =i \partial\bar\partial \eta$, and
  $$ \lambda_t^*p^* \omega_t^{SF}  = i  \partial\bar\partial \eta + f^* \omega.    $$
 By Proposition 4.3 in \cite{GTZ},
 $$ \| \lambda_t^*p^* T_{\sigma_0}^*\omega_t \|_{C_{\rm loc}^{\ell}} \leq C_{\ell}$$
 for constants $C_{\ell}>0$, and by Lemma 4.7 in  \cite{GTZ} (also Proposition 3.2 of \cite{TZ}),  $$  \lambda_t^*p^* T_{\sigma_0}^*\omega_t \rightarrow   i  \partial\bar\partial \eta + f^* \omega$$ when $t\rightarrow 0,$ in the locally  $C^{\infty}$-sense.

If we denote  $\psi_t= \varphi_t\circ T_{\sigma_0} \circ p \circ \lambda_t$, then $\psi_t$ is $t^{\frac{1}{2}} \Lambda$-periodic, i.e.  $$\psi_t(w,z)=\psi_t(w,z+t^{\frac{1}{2}} a+ t^{\frac{1}{2}}bZ)$$ where $a+bZ=(a_1+b_1Z_1, \cdots,  a_{m-n}+b_{m-n}Z_{m-n})$ for any  $a_j, b_j \in\mathbb{Z}$. By the above we can write $$  \lambda_t^*p^* T_{\sigma_0}^*\omega_t =   i  \partial\bar\partial \eta +  \omega + i\partial\overline{\partial} \psi_t,$$ and note that    $ \| i\partial\overline{\partial} \psi_t \|_{C_{\rm loc}^{\ell}} \leq C_{\ell}$,   and  $i\partial\overline{\partial} \psi_t \rightarrow 0$  as $t\rightarrow 0$, on $ U\times \mathbb{C}^{m-n}$.

\begin{lem}\label{decay} Denote $$\psi_{t,w_k\bar{w}_l}= \frac{\partial^2 \psi_t}{\partial w_k\partial \bar{w}_l} ,  \  \  \psi_{t,z_k\bar{z}_l}= \frac{\partial^2 \psi_t}{\partial z_k\partial \bar{z}_l} , \  \  {\rm and}  \  \ \psi_{t,z_k\bar{w}_l}= \frac{\partial^2 \psi_t}{\partial z_k\partial \bar{w}_l}. $$
For any $\nu\in \mathbb{N}$ and $\ell\geq 0$, there is a constant $C'_{\ell,\nu}>0$ such that $$ \|\psi_{t,w_k\bar{w}_l}-\chi_{t,kl} \|_{C_{\rm loc}^0} \leq C'_{0,\nu} t^{\frac{\nu}{2}},$$
and
 $$\|\frac{\partial}{\partial z_j}\psi_{t,w_k\bar{w}_l}\|_{C_{\rm loc}^\ell}+ \|\psi_{t,z_k\bar{z}_l}\|_{C_{\rm loc}^\ell} + \|\psi_{t,z_k\bar{w}_l}\|_{C_{\rm loc}^\ell} \leq C'_{\ell,\nu} t^{\frac{\nu}{2}}, $$ where $\chi_{t,kl}$ are functions  on $U$.
 \end{lem}

 \begin{proof}
For any $t\in (0,1]$, let $h_t$ be  a  $\sqrt{t}\Lambda$-periodic real  function on $U \times \mathbb{C}^{m-n}$ such that $$\big|\partial^{\beta}_{\beta_1, \cdots,  \beta_{2(m-n)}}h_t\big|\leq C_\beta,$$
where $$
\partial^{\beta}_{\beta_1, \cdots,  \beta_{2(m-n)}}h_t= \frac{\partial^{\beta}h_t}{\partial^{\beta_1}y_1 \cdots \partial^{\beta_{2(m-n)}}y_{2(m-n)}},$$
and $z_j=y_j+y_{m-n+j}Z_j$,   $\beta=\beta_1+\cdots +\beta_{2(m-n)}$, and $ C_\beta$ is  independent of $t$.  For $w\in U$, let $D_w\subset \{w\}\times \mathbb{C}^{m-n}$ be the fundamental domain of the  $\sqrt{t}\Lambda_w$-action.
  For any $p_1$ and $p_2\in D_w$, if we denote by $\gamma \subset D_w$ the  line segment connecting  $p_1$ and $p_2$, then
  \bea  & &
|\partial^{\beta}_{\beta_1, \cdots,  \beta_{2(m-n)}}h_t(p_1)-\partial^{\beta}_{\beta_1, \cdots,  \beta_{2(m-n)}}h_t(p_2)|  \nonumber\\ &&\qquad\qquad\qquad\qquad\qquad\qquad\leq \Big|\int_{\gamma}\partial_{\dot{\gamma}} \partial^{\beta}_{\beta_1, \cdots,  \beta_{2(m-n)}}h_t (\gamma (s)) ds \Big|  \nonumber\\
&&\qquad\qquad\qquad\qquad\qquad\qquad\leq  C \sqrt{t}  \sum_{j=1}^{2(m-n)} \sup |\partial_{y_j}\partial^{\beta}_{\beta_1, \cdots,  \beta_{2(m-n)}}h_t |.   \nonumber
\eea
Since $h_t$ is periodic we can choose $p_2$ to be  a local max, which implies  $\partial^{\beta}_{\beta_1, \cdots,  \beta_{2(m-n)}}h_t (p_2)=0$. Thus for any $ k \geq 1$, we obtain  $$ |h_t-\bar{h}_t|\leq C_{0,\nu} t^{\frac{\nu}{2}}, \   \ {\rm and }  \  \  |\partial^{\beta}_{\beta_1, \cdots,  \beta_{2(m-n)}}h_t|\leq C_{\beta,\nu} t^{\frac{\nu}{2}},   $$ for constants $C_{\beta,\nu}$ independent of $t$, where $\bar{h}_t=\sup\limits_{z \in D_w}h_t$ is a function on $U$.

   The first inequality in the lemma is obtained by letting $h_t=\psi_{t,w_k\bar{w}_l}$ and $\bar{h}_t= \chi_{t,kl}$, and the second inequality follows by taking $$h_t= \frac{\partial^{\ell}\psi_t}{\partial^{\ell_1}y_1 \cdots \partial^{\ell_{2(m-n)}}y_{2(m-n)}}$$ for any $\ell \geq 1$.
\end{proof}

We obtain the desired conclusion by letting $\chi_t =i\sum\limits_{kl}\chi_{t,kl} dw_k \wedge d\bar{w}_l$. Note that the convergence in Lemma \ref{decay} is slightly stronger than Theorem  \ref{ttm-decay}, and  we  use Lemma \ref{lem-decay},  a simplified version of  Lemma \ref{decay},   in the proof of Theorem \ref{thm-main}.
\end{proof}

\end{document}